\documentclass[11pt]{amsart}
\usepackage{palatino}

\newcommand{\CC}{\mathbb{C}}
\newcommand{\EE}{\mathbb{E}}

\newcommand{\FF}{\mathbb{F}}
\newcommand{\GG}{\mathbb{G}}
\newcommand{\NN}{\mathbb{N}}
\newcommand{\PP}{\mathbb{P}}

\newcommand{\ZZ}{\mathbb{Z}}

\newcommand{\cB}{\mathcal{B}}
\newcommand{\cC}{\mathcal{C}}

\newcommand{\cG}{\mathcal{G}}
\newcommand{\cH}{\mathcal{H}}

\newcommand{\cK}{\mathcal{K}}
\newcommand{\cM}{\mathcal{M}}
\newcommand{\cN}{\mathcal{N}}
\newcommand{\cO}{\mathcal{O}}
\newcommand{\cP}{\mathcal{P}}

\newcommand{\cU}{\mathcal{U}}

\newcommand{\cX}{\mathcal{X}}

\newcommand{\cZ}{\mathcal{Z}}

\newcommand{\fa}{\mathfrak{a}}

\newcommand{\fb}{\mathfrak{b}}

\newcommand{\fC}{\mathfrak{C}}

\newcommand{\fe}{\mathfrak{e}}
\newcommand{\ff}{\mathfrak{f}}

\newcommand{\fg}{\mathfrak{g}}
\newcommand{\fgl}{\mathfrak{gl}}
\newcommand{\fh}{\mathfrak{h}}
\newcommand{\fl}{\mathfrak{l}}
\newcommand{\fL}{\mathfrak{L}}

\newcommand{\fK}{\mathfrak{K}}

\newcommand{\fn}{\mathfrak{n}}

\newcommand{\fpgl}{\mathfrak{pgl}}

\newcommand{\fr}{\mathfrak{r}}

\newcommand{\fs}{\mathfrak{s}}

\newcommand{\fsl}{\mathfrak{sl}}
\newcommand{\fso}{\mathfrak{so}}
\newcommand{\fsp}{\mathfrak{sp}}
\newcommand{\ft}{\mathfrak{t}}
\newcommand{\fU}{\mathfrak{U}}
\newcommand{\fu}{\mathfrak{u}}

\newcommand{\fv}{\mathfrak{v}}

\newcommand{\dact}{\boldsymbol{.}}
\newcommand{\lra}{\longrightarrow}

\newcommand{\msrke}{\mathsf{rk}_p}

\DeclareMathOperator{\Ab}{Ab}
\DeclareMathOperator{\Ad}{Ad}
\DeclareMathOperator{\ad}{ad}
\DeclareMathOperator{\add}{\mathsf{add}}
\DeclareMathOperator{\ann}{ann}
\DeclareMathOperator{\Aut}{Aut}
\DeclareMathOperator{\Bor}{Bor}
\DeclareMathOperator{\Char}{char}
\DeclareMathOperator{\CJT}{\mathsf{CJT}}

\DeclareMathOperator{\CR}{\mathsf{CR}}

\DeclareMathOperator{\Der}{Der}

\DeclareMathOperator{\EIP}{\mathsf{EIP}}

\DeclareMathOperator{\End}{End}
\DeclareMathOperator{\Exp}{Exp}
\DeclareMathOperator{\Ext}{Ext}
\DeclareMathOperator{\GL}{GL}
\DeclareMathOperator{\Gr}{Gr}

\DeclareMathOperator{\Ht}{Ht}
\DeclareMathOperator{\Hom}{Hom}

\DeclareMathOperator{\Irr}{Irr}
\DeclareMathOperator{\im}{im}
\DeclareMathOperator{\id}{id}

\DeclareMathOperator{\Jt}{Jt}

\DeclareMathOperator{\Lie}{Lie}

\DeclareMathOperator{\modd}{mod}

\DeclareMathOperator{\msdeg}{\mathsf{deg}}
\DeclareMathOperator{\msd}{\mathsf{d}}
\DeclareMathOperator{\msim}{\mathsf{im}}
\DeclareMathOperator{\mspl}{\mathsf{pl}}
\DeclareMathOperator{\msrk}{\mathsf{rk}}
\DeclareMathOperator{\msrkm}{\mathsf{rk}_{\rm min}}
\DeclareMathOperator{\Sw}{\mathsf{Sw}}

\DeclareMathOperator{\Nor}{Nor}
\DeclareMathOperator{\ord}{ord}
\DeclareMathOperator{\PGL}{PGL}

\DeclareMathOperator{\Rad}{Rad}

\DeclareMathOperator{\res}{res}
\DeclareMathOperator{\rk}{rk}

\DeclareMathOperator{\SL}{SL}

\DeclareMathOperator{\Soc}{Soc}
\DeclareMathOperator{\Sp}{Sp}

\DeclareMathOperator{\Top}{Top}
\DeclareMathOperator{\tr}{tr}

\DeclareMathOperator{\wt}{wt}
\DeclareMathOperator{\Proj}{Proj}

\usepackage{pictex}
\usepackage{amscd}
\usepackage{latexsym}
\usepackage{amssymb}
\usepackage{eucal}
\usepackage{eufrak}
\usepackage{verbatim}
\usepackage{hyperref}

\numberwithin{equation}{section}

\newtheorem{Theorem}{Theorem}[section]
\newtheorem{Lemma}[Theorem]{Lemma}

\theoremstyle{Theorem}
\newtheorem{Thm}{Theorem}[subsection]
\newtheorem{Lem}[Thm]{Lemma}
\newtheorem{Prop}[Thm]{Proposition}
\newtheorem{Cor}[Thm]{Corollary}

\newtheorem*{thm*}{Theorem A}
\newtheorem*{thm**}{Theorem B}

\theoremstyle{remark}
\newtheorem*{Remark}{Remark}
\newtheorem*{Remarks}{Remarks}
\newtheorem*{Definition}{Definition}
\newtheorem*{Example}{Example}
\newtheorem*{Examples}{Examples}

\numberwithin{equation}{section}

\begin{document}

\title{Varieties of Elementary Abelian Lie Algebras and Degrees of Modules}

\author{Hao Chang \lowercase{and} Rolf Farnsteiner}

\address[Hao Chang]{School of Mathematics and Statistics, Central China Normal University, 430079 Wuhan, People's Republic of China}
\email{chang@mail.ccnu.edu.cn}
\address[Rolf Farnsteiner]{Mathematisches Seminar, Christian-Albrechts-Universit\"at zu Kiel, Ludewig-Meyn-Str. 4, 24098 Kiel, Germany}
\email{rolf@math.uni-kiel.de}

\date{\today}

\makeatletter
\makeatother


\begin{abstract} Let $(\fg,[p])$ be a restricted Lie algebra over an algebraically closed field $k$ of characteristic $p\!\ge \!3$. Motivated by the behavior of geometric invariants of the so-called
$(\fg,[p])$-modules of constant $j$-rank ($j \in \{1,\ldots,p\!-\!1\}$), we study the projective variety $\EE(2,\fg)$ of two-dimensional elementary abelian subalgebras. If $p\!\ge\!5$, then the topological space
$\EE(2,\fg/C(\fg))$, associated to the factor algebra of $\fg$ by its center $C(\fg)$, is shown to be connected. We give applications concerning categories of $(\fg,[p])$-modules of constant $j$-rank
and certain invariants, called $j$-degrees. \end{abstract}

\dedicatory{Dedicated to Jens Carsten Jantzen on the occasion of his 70th birthday}

\maketitle

\section*{Introduction}
Following Quillen's seminal papers \cite{Qu1,Qu2} on the cohomology rings of finite groups, elementary abelian groups have played a prominent role in modular representation theory. The idea of
detecting important properties by restriction to such subgroups has been generalized to other algebraic structures, such as finite group schemes. In this article, we will pursue this approach in the context
of infinitesimal group schemes of height $1$. These groups and their representations are well-known to correspond to restricted Lie algebras and their modules, cf.\ \cite[(II,\S7, ${\rm n}^{\rm o}$4)]{DG}.

Let $(\fg,[p])$ be a restricted Lie algebra over an algebraically closed field $k$ of characteristic $p\!>\!0$. By definition, $\fg$ is a Lie algebra over $k$, which is equipped with a map
$[p] : \fg \lra \fg \ ; \ x \mapsto x^{[p]}$ that has the formal properties of a $p$-th power map. We say that $\fg$ is {\it elementary abelian}, provided $\fg$ is abelian and $[p]\!=\!0$. Such Lie algebras first
appeared in Hochschild's work on restricted Lie algebra cohomology \cite{Ho}, where they were referred to as ``strongly abelian''. In a recent article \cite{CFP13}, they were called elementary Lie algebras.
Our choice of terminology derives from the fact that restricted enveloping algebras of elementary abelian Lie algebras are isomorphic (as associative algebras) to group algebras of $p$-elementary abelian
groups. The, up to isomorphism, unique elementary abelian Lie algebra of dimension $r$ will be denoted $\fe_r$.

One-dimensional elementary abelian subalgebras can be construed as elements of the projectivized nullcone $\PP(V(\fg))$, which is associated to the closed conical variety
\[ V(\fg) := \{x \in \fg \ ; \ x^{[p]}=0\}.\]
In \cite{Ca} Carlson proved a fundamental result, which implies in particular that the variety $\PP(V(\fg))$ is connected.

The varieties $\EE(r,\fg)$ of elementary abelian subalgebras of dimension $r$, which were introduced in \cite{CFP13}, are natural generalizations of $\PP(V(\fg))$. Our motivation for studying the
particular case where $r\!=\!2$ derives from the observation that certain geometric invariants of $(\fg,[p])$-modules are determined by their restrictions to $2$-dimensional elementary abelian subalgebras,
cf.\ \cite{Fa17}. Basic examples, such as particular one-dimensional central extensions of $\fsl(2)$, show that Carlson's result does not hold for $\EE(2,\fg)$. Using techniques based on sandwich elements
and exponentials of restricted Lie algebras, we establish in Section \ref{S:Co} a result, which implies the following:

\bigskip

\begin{thm*} Suppose that $p\!\ge\!5$. If $(\fg,[p])$ is a restricted Lie algebra with center $C(\fg)$, then the variety $\EE(2,\fg/C(\fg))$ is connected. \end{thm*}

\bigskip
\noindent
In fact, in case $\fg$ is algebraic, that is, if $\fg=\Lie(G)$ is the Lie algebra of an algebraic group $G$, the space $\EE(2,\fg)$ is connected whenever $p\!\ge\!3$.

Results on the topology of $\EE(r,\fg)$ sometimes provide insight into the structure of certain full subcategories of the module category $\modd U_0(\fg)$ of the restricted enveloping algebra $U_0(\fg)$
of $\fg$: According to \cite[(4.13)]{CFP13}, the category of modules of constant $(r,j)$-radical rank is closed under direct summands, whenever $\EE(r,\fg)$ is connected.

Our interest in the structure of $\EE(2,\fg)$ is motivated by properties of the so-called {\it modules of constant $j$-rank}, cf.\ \cite{FP10}. Every such module $M$ gives rise to a morphism $\PP(V(\fg)) \lra
\Gr_d(M)$ with values in a suitable Grassmannian, which in turn defines a function
\[ \msdeg^j_M : \EE(2,\fg) \lra \NN_0,\]
cf.\ \cite{Fa17}. If $\EE(2,\fg) \ne \emptyset$, the $j$-degree functions $\msdeg^j_M$ provide invariants of $M$ enabling us to distinguish modules of the same Jordan type.  As we show in Section
\ref{S:LCF}, the  function $\msdeg^j_M$ is locally constant. In conjunction with Theorem A this implies:

\bigskip

\begin{thm**} Suppose that $p\ge 5$. Let $(\fg,[p])$ be a restricted Lie algebra such that $\dim V(C(\fg))\!\ne\!1$. If $M$ is a module of constant $j$-rank, then $\msdeg^j_M$ is constant. \end{thm**}

\bigskip
\noindent
The condition on the nullcone of $C(\fg)$ is essential for the validity of the foregoing Theorem: Using basic facts from Auslander-Reiten theory, we provide an example of a module $M$ of a
four-dimensional restricted Lie algebra with a one-dimensional unipotent center, whose degree functions $\msdeg^j_M$ are not constant.

Following the work of Carlson-Friedlander-Pevtsova \cite{CFP08}, modules of constant Jordan type and related classes of modules have received considerable attention (cf.\ \cite{Be} and the references
therein). In the context of restricted Lie algebras, these modules are given by conditions on nilpotent operators associated to non-zero elements of the nullcone. Given $j \in \{1,\ldots,p\!-\!1\}$, we denote
by $\EIP^j(\fg) \subseteq \CR^j(\fg)$ the full subcategories of $\modd U_0(\fg)$ consisting of modules with the equal $j$-image property and having constant $j$-rank, respectively. A module belonging
to $\CJT(\fg):=\bigcap_{j=1}^{p-1}\CR^j(\fg)$ is said to have constant Jordan type. In view of \cite{Be,Bi}, the category $\CJT(\fe)_{\le 3}$ of those modules of constant Jordan type whose blocks
have size $\le\!3$ has wild representation type whenever $\fe$ is elementary abelian of $p$-rank $\ge 2$. Our abovementioned results imply that for other classes of Lie algebras, the defining conditions of
these subcategories may be much more restrictive: For a reductive Lie algebra $\fg$ of large rank (and hence of wild representation type), the subcategory consisting of those $M \in \CR^{p-1}(\fg)$ such
that $M|_\fe \in \EIP^{p-1}(\fe)$ for some $\fe \in \EE(2,\fg)$ turns out to be semisimple (see Proposition \ref{Jt2}).

\emph{Throughout this paper, $k$ denotes an algebraically closed field of characteristic $\Char(k)=p\!\ge\!3$. All vector spaces are assumed to be finite-dimensional over $k$}.

\bigskip
\noindent
\textbf{Acknowledgments.} Parts of this article were written during a two-month visit of the second named author to the Collaborative Research Center 701 of the University of
Bielefeld. He would like to take this opportunity and thank Henning Krause and his research team for their hospitality.

H.C.\ is partially supported by NSFC (No.\ 11801204).

\bigskip

\section{Exponentials and Sandwich Elements}
Let $\fg$ be a finite-dimensional Lie algebra over $k$. An element $c\in\fg$ is called an {\it absolute zero-divisor} if $(\ad c)^2=0$ (see \cite{Kos}). In more recent terminology, $c$ is sometimes
referred to as a {\it sandwich element}. Since Lie algebras containing non-zero sandwich elements have a degenerate Killing form, Kostrikin referred to them as {\it algebras with strong degeneration}.
We let
\[ \Sw(\fg):=\{c\in\fg \ ; \ (\ad c)^2=0\}\]
be the Zariski-closed conical subset of sandwich elements of $\fg$. For $i \in \{2,\ldots,p\}$, we also consider the closed conical variety
\[\cN_i(\fg) := \{x \in \fg \ ; \ (\ad x)^i = 0\}.\]
In this section we collect structural properties of $\fg$ that are determined by conditions on $\Sw(\fg)$.

\bigskip
\subsection{Exponential maps}
The map
\[ \exp : \cN_p(\fg) \lra \GL(\fg) \ \ ; \ \ x \mapsto \sum_{i=0}^{p-1} \frac{(\ad x)^i}{i!}\]
is readily seen to be a morphism of conical affine varieties. Being a union of lines through the origin, any closed conical subset $X \subseteq \cN_p(\fg)$ is connected. It follows that
$\exp(X) \subseteq \GL(\fg)$ also enjoys this property.

We denote by $\Aut(\fg)$ the group of automorphisms of $\fg$. Being a closed subgroup of $\GL(\fg)$, $\Aut(\fg)$ is an affine algebraic group. As $p \ge 3$,
\[ \cN_{\frac{p+1}{2}}(\fg) \subseteq \cN_p(\fg)\]
is a closed, conical subset of $\fg$ that contains $\Sw(\fg)$.

For ease of reference we record the following well-known result:

\bigskip

\begin{Lem} \label{Em1} The following statements hold:
\begin{enumerate}
\item If $x \in \cN_{\frac{p+1}{2}}(\fg)$, then $\exp(x) \in \Aut(\fg)$.
\item If $x,y \in \cN_{\frac{p+1}{2}}(\fg)$ are such that $[x,y]=0$, then $\exp(x\!+\!y)=\exp(x)\circ\exp(y)$.  \hfill $\square$  \end{enumerate}\end{Lem}

\bigskip
\noindent
Note that the subsets $\Sw(\fg) \subseteq \cN_{\frac{p+1}{2}}(\fg) \subseteq \cN_p(\fg)$ are stable with respect to the canonical action of $\Aut(\fg)$ on $\fg$. Moreover, we have
\[ \varphi\circ \exp(x) = \exp(\varphi(x))\circ \varphi\]
for all $x \in \cN_p(\fg)$ and $\varphi \in \Aut(\fg)$. Thus, if $X \subseteq \cN_{\frac{p+1}{2}}(\fg)$ is an $\Aut(\fg)$-stable subset, then the subgroup of $\Aut(\fg)$ that is generated by
$\exp(X)$ is normal in $\Aut(\fg)$.

We denote by $C(\fg)$ the {\it center} of $\fg$ and call $\fg$ {\it centerless} in case $C(\fg)=(0)$. We say that $S \subseteq \fg$ is a {\it Lie subset}, provided $[s,t] \in S$ for all $s,t \in S$.

The automorphism group of a restricted Lie algebra $(\fg,[p])$ will be denoted $\Aut_p(\fg)$. As it coincides with the stabilizer of the $p$-map, $\Aut_p(\fg)$ is a closed subgroup of
$\Aut(\fg) \subseteq \GL(\fg)$. We let $G_\fg:=\Aut_p(\fg)^{\circ}$ be the identity component of $\Aut_p(\fg)$.

If $(\fg,[p])$ is centerless, then $\cN_p(\fg) = V(\fg)$ is the nullcone of $\fg$.

\bigskip

\begin{Lem}\label{Em2} Let $\fg$ be a Lie algebra.
\begin{enumerate}
\item If $X \subseteq \fg$ is a conical closed $\Aut(\fg)$-stable subset, then $[c,x] \in X$ for all $c \in \Sw(\fg)$ and $x \in X$.
\item $\Sw(\fg)$ is a Lie subset of $\fg$.
\item If $(\fg,[p])$ is restricted and centerless, then $\Sw(\fg) \subseteq V(\fg)$ is a Lie subset and $\exp(\Sw(\fg))\subseteq G_\fg$.
\item If $(\fg,[p])$ is restricted and $\fn \unlhd \fg$ is an elementary abelian $p$-ideal, then $\exp(\fn) \subseteq G_\fg$ is an abelian, connected subgroup. \end{enumerate}\end{Lem}

\begin{proof} (1) In view of Lemma \ref{Em1}, we have $\exp(c) \in \Aut(\fg)$ for all $c \in \Sw(\fg)$. Since the set $X$ is stable under $\Aut(\fg)$, it follows that
$\exp(c)(X)\subseteq X$ for all $c \in \Sw(\fg)$. Consequently, the map
\[ f_{(x,c)} : k \lra \fg \ \  ; \ \ \alpha\mapsto\alpha x\!+\![c,x] \ \ \ \ \ \ \ ((x,c) \in X\!\times\!\Sw(\fg))\]
is a morphism such that $f_{(x,c)}(k\!\smallsetminus\!\{0\})\subseteq X$. Hence
\[f_{(x,c)}(k)=f_{(x,c)}(\overline{k\!\smallsetminus\!\{0\}})\subseteq\overline{f_{(x,c)}(k\!\smallsetminus\!\{0\})}\subseteq\overline{X}=X,\]
so that $[c,x]=f_{(x,c)}(0)\in X$.

(2) Since $\Sw(\fg)$ is a closed conical $\Aut(\fg)$-stable subset of $\fg$, this is a direct consequence of (1).

(3) As $\fg$ is centerless, it follows that $\Aut(\fg)= \Aut_p(\fg)$. In addition, $\exp(\Sw(\fg))$ is connected and contains $\id_\fg$; hence it is contained in the connected component $G_\fg$
of $\id_\fg$.

(4) Since $\fn$ is elementary abelian, we have $\fn \subseteq \Sw(\fg)$. As $\fn$ is an abelian ideal, direct computation shows that the map $\exp$ restricts to a morphism
\[ \exp : \fn \lra \Aut(\fg)\]
of algebraic groups. In particular, the group $\exp(\fn)$ is abelian and connected.

Let $a \in \fg$ and $b \in \fn$ and suppose that $T$ is an indeterminate over $k$. Observing $[b,a] \in \fn$ in conjunction with $\fn$ being abelian, we obtain the
following identities in the Lie algebra $\fg\!\otimes_k\!k[T]$:
\begin{eqnarray*}
\ad(a\otimes T\!+\!b\otimes 1)^{p-1}(a\otimes 1) & = & \ad(a\otimes T\!+\!b\otimes 1)^{p-2}([b,a]\otimes 1) = (\ad a)^{p-2}([b,a])\otimes T^{p-2}\\
                                                                              & = & -(\ad a)^{p-1}(b)\otimes T^{p-2}.
\end{eqnarray*}
In view of Jacobson's formula \cite[(II.1)]{SF}, this implies
\[ (a\!+\!b)^{[p]} = a^{[p]}\!+\!\!(\ad a)^{p-1}(b).\]
Let $x \in \fn$ and $a \in \fg$. Then we have $[x,a] \in \fn$ and the foregoing formula yields
\begin{eqnarray*}
\exp(x)(a)^{[p]} & = & (a\!+\![x,a])^{[p]} = a^{[p]}\!+\!(\ad a)^{p-1}([x,a]) = a^{[p]}\!-\!(\ad a)^p(x) = a^{[p]}\!-\![a^{[p]},x] \\
& = & a^{[p]}\!+\![x,a^{[p]}] = \exp(x)(a^{[p]}).
\end{eqnarray*}
Consequently, $\exp(x) \in \Aut_p(\fg)$. We conclude that $\exp(\fn) \subseteq G_\fg$.  \end{proof}

\bigskip

\begin{Remark} {\it Let $\fg:=\Lie(G)$ be the Lie algebra of a reductive algebraic group $G$. Then we have $\Sw(\fg)\!=\!C(\fg)$}.

\begin{proof} According to Lemma \ref{Em2}, $\Sw(\fg)$ is a conical $G$-stable Lie subset of $\fg$. The Engel-Jacobson Theorem \cite[(I.3.1)]{SF} now shows that the linear span $\fn:= \langle \Sw(\fg)
\rangle$ is a $p$-ideal of $\fg$ that acts strictly triangulably on $\fg$. In particular, $\fn$ is nilpotent.

Since $G$ is reductive, the factor group $G':= G/C(G)^\circ$ of $G$ by the connected component of the center $C(G)$ is semisimple and $C(G)^\circ$ is a torus. We put $\fg':= \Lie(G')$ and consider the
canonical projection $\pi : G \lra G'$. The differential $\msd(\pi) : \fg \lra \fg'$ is surjective (cf.\ \cite[(5.2.3),(3.2.21)]{Sp}), while $\ker\msd(\pi)=\Lie(C(G))$ is a torus. In view of \cite[(10.2)]{Hu67}, the
$p$-ideal $\msd(\pi)(\fn)$ of $\fg'$ is a torus, so that $\fn$, being an extension of tori, is also a torus. Hence $\fn \subseteq C(\fg)$, whence $\Sw(\fg)=C(\fg)$. \end{proof} \end{Remark}

\bigskip

\begin{Definition} Let $\fg$ be a Lie algebra, $L \subseteq \cN_{\frac{p+1}{2}}(\fg)$ be a conical Lie subset. We denote by $\Exp(L) \subseteq \Aut(\fg)$ the smallest closed subgroup of $\Aut(\fg)$
containing $\exp(L)$. \end{Definition}

\bigskip

\begin{Lem}\label{Em3} Let $(\fg,[p])$ be a restricted Lie algebra, $L \subseteq \cN_{\frac{p+1}{2}}(\fg)$ be a conical Lie subset such that $\exp(L) \subseteq \Aut_p(\fg)$. Then the following
statements hold:
\begin{enumerate}
\item $\Exp(L)$ is a closed, connected, unipotent subgroup of $G_\fg$.
\item We have $\ad(\langle L \rangle) \subseteq \Lie(\Exp(L))$.
\item If $L$ is $G_\fg$-stable, then $\Exp(L)$ is normal in $G_\fg$. \end{enumerate} \end{Lem}

\begin{proof} (1) By assumption, $L$ is a Lie subset that acts via the adjoint representation on $\fg$ by nilpotent transformations. Thus, the Engel-Jacobson
theorem provides a basis of $\fg$, such that, for each element $x \in \langle L \rangle$, the transformation $\ad x$ is represented by a strictly upper triangular
matrix. Consequently, $\Exp(L)$ is a subgroup of the closed subgroup of upper triangular unipotent matrices and thus is in particular unipotent.

Each element $c\in L$ gives rise to a one-parameter subgroup
\[\varphi_c: k \lra \Exp(L) \ \ ; \ \ \lambda\mapsto \exp(\lambda c).\]
As $\Exp(L)$ is generated by all $\varphi_c(k)$, general theory (cf.\ \cite[Theorem 2.4.6]{Ge}) provides $c_1,\ldots, c_n \in L$ such that
\[\Exp(L) =\varphi_{c_1}(k)\cdots\varphi_{c_n}(k).\]
Hence $\Exp(L)$ is a closed, connected subgroup of $\Aut(\fg)$. Our assumption $\exp(L) \subseteq \Aut_p(\fg)$ now yields $\Exp(L)\subseteq G_\fg$.

(2) Given $c \in L$, we have
\[ \varphi_c(\lambda) = \sum_{i=0}^{\frac{p-1}{2}} \lambda^i \frac{(\ad c)^i}{i!},\]
so that $\ad c =\msd(\varphi_c)(0) \in \Lie(\Exp(L))$. This implies $\ad(\langle L \rangle) \subseteq \Lie(\Exp(L))$.

(3) This was observed before. \end{proof}

\bigskip

\begin{Cor} \label{Em4} Let $(\fg,[p])$ be a restricted Lie algebra such that $\exp(\Sw(\fg))\subseteq \Aut_p(\fg)$. If $\{0\} \subsetneq X$ is a closed conical $G_\fg$-stable subset of $\fg$,
then there is $x_0 \in X\!\smallsetminus\!\{0\}$ such that $[c,x_0]=0$ for all $c \in \Sw(\fg)$. \end{Cor}

\begin{proof} Thanks to Lemma \ref{Em3}, $\Exp(\Sw(\fg))$ is a unipotent subgroup of $G_\fg$ that acts on the projective variety $\PP(X)$.
Borel's fixed point theorem \cite[(7.2.5)]{Sp} provides $x_0 \in X\!\smallsetminus\!\{0\}$ and a function $\lambda : \Sw(\fg) \lra k$ such that
\[ [c,x_0] = (\lambda(c)\!-\!1)x_0 \ \ \ \ \ \forall \ c \in \Sw(\fg).\]
Since $\ad c$ is nilpotent, it follows that $\lambda(c)\!-\!1 = 0$ for every $c \in \Sw(\fg)$. \end{proof}

\bigskip

\subsection{Inner ideals and almost classical Lie algebras}
Let $\fg$ be a finite-dimensional Lie algebra. A subspace $V \subseteq \fg$ is called an {\it inner ideal}, provided $[V,[V,\fg]] \subseteq V$.

\bigskip

\begin{Thm}[\cite{Pr86,Pr87a}] \label{ACL0} Suppose that $p\!\ge\!5$. Every non-zero Lie algebra $\fg$ affords a one-dimensional inner ideal. \hfill $\square$ \end{Thm}

\bigskip
\noindent
A finite-dimensional Lie algebra, which is representable as a direct sum of ideals which are simple Lie algebras of classical type, is called a {\it classical semisimple} Lie algebra. The reader is referred
to \cite[Chap.II]{Se67} for a comprehensive account concerning classical Lie algebras. Following Premet \cite{Pr87a}, we will say that $\fg$ is an {\it almost classical semisimple} Lie algebra, if there
exists a classical semisimple Lie algebra $\fL$ such that $\Der\fL\supseteq\fg\supseteq\ad\fL$. We record the following useful theorem, which was established by Premet (cf.\ \cite[Thm.3]{Pr87a}, \cite{Pr87b}):

\bigskip

\begin{Thm}\label{ACL1} Suppose that $p\!\ge\!5$. Let $\fg$ be a non-zero finite-dimensional Lie algebra. The following conditions are equivalent:
\begin{enumerate}
\item[(i)] $\Sw(\fg)=\{0\}$.
\item[(ii)] $\fg$ is an almost classical semisimple Lie algebra. \hfill $\square$ \end{enumerate}  \end{Thm}

\bigskip

\begin{Lem}\label{ACL2} Let $(\fg,[p])$ be almost classical semisimple. Then $[\fg,\fg]$ is a classical semisimple $p$-ideal of $\fg$ such that $V([\fg,\fg])=V(\fg)$.\end{Lem}

\begin{proof} Given $m \in \NN$, we consider the classical simple Lie algebra $\tilde{A}_{mp-1}:=\fsl(mp)/k$. Using \cite[(1.2)]{Fa88} one can show by direct computation that
\[\Der\tilde{A}_{mp-1} =  \ad\tilde{A}_{mp-1}\!\oplus\! kt,\]
where $kt \subseteq \Der\tilde{A}_{mp-1}$ is a torus. In particular, $[\Der\tilde{A}_{mp-1},\Der\tilde{A}_{mp-1}] = \ad\tilde{A}_{mp-1}$ is a $p$-subalgebra of $\Der\tilde{A}_{mp-1}$
and $V([\Der\tilde{A}_{mp-1},\Der\tilde{A}_{mp-1}]) = V(\Der\tilde{A}_{mp-1})$.

Since $\fg$ is almost classical, there exists a classical semisimple Lie algebra $\fL$ such that
\[ \Der\fL\supseteq\fg\supseteq\ad\fL.\]
Writing $\fL=\fa_1\!\oplus\cdots\oplus\!\fa_\ell$, where the $\fa_i$ are simple ideals of classical type, it follows that
\[\Der\fL\cong\Der\fa_1\oplus\cdots\oplus\Der\fa_\ell\]
is a direct sum of restricted Lie algebras. It is well-known that $\ad\fa=\Der\fa$ for every classical simple Lie algebra $\fa\not \cong\tilde{A}_{mp-1}$, cf.\ \cite[p.160]{Pr86}. The observation above
now implies,
\[[\Der\fL,\Der\fL]=\ad\fL,\]
as well as $\ad\fL = [\ad\fL,\ad\fL] \subseteq [\fg,\fg] \subseteq [\Der\fL,\Der\fL] =\ad\fL$. Accordingly, $[\fg,\fg] = \ad\fL$ is a $p$-subalgebra of $\Der\fL$, whose $p$-th power map
is denoted by $x \mapsto x^p$.

Given $x \in [\fg,\fg]$, we consider the derivation $d:= x^{[p]}\!-\!x^p \in \Der\fL$. We obtain $(0) = [d,[\fg,\fg]] = [d,\ad\fL] = \ad d(\fL)$, so that $d(\fL) \subseteq C(\fL)=(0)$. Thus, $x^{[p]} = x^p
\in [\fg,\fg]$, showing that $[\fg,\fg]$ is a $p$-ideal of $\fg$.

In view of the above, there exists a torus $\ft \subseteq \Der\fL$ such that $\Der\fL = \ft \! \oplus \! \ad\fL$. This readily implies $V(\Der\fL)=V(\ad\fL)$.

Let $d \in V(\fg)$. For $y \in \fL$ we obtain $0 = (\ad d)^p(\ad y) = \ad d^p(y)$, showing that $d^p=0$. Consequently, $d \in \ad\fL = [\fg,\fg]$, whence $V(\fg)=V([\fg,\fg])$. \end{proof}

\bigskip

\section{Connectedness of $\EE(2,\fg)$} \label{S:Co}
Let $(\fg,[p])$ be a restricted Lie algebra. Recall that
\[ V(\fg) := \{x \in \fg \ ; \ x^{[p]}=0\}\]
is the {\it (restricted) nullcone} of $V(\fg)$. For $r \in \NN_0$, we denote by $\Gr_r(\fg)$ the Grassmann variety of $r$-dimensional subspaces of $\fg$. In view of
\cite[(7.2),(7.3)]{Fa04}, the set $\Ab_r(\fg) \subseteq \Gr_r(\fg)$ of all abelian Lie subalgebras of dimension $r$ is closed, while the map
\[ \Gr_r(\fg) \lra \NN_0 \ \ ; \ \ X \mapsto \dim X \cap V(\fg)\]
is upper semicontinuous. Consequently,
\[ \EE(r,\fg):= \{ \fe \in \Ab_r(\fg) \ ; \ \fe \subseteq V(\fg)\}\]
is a closed subset of $\Gr_r(\fg)$. This is the projective variety of $r$-dimensional elementary abelian subalgebras of $\fg$, that was first introduced by Carlson-Friedlander-Pevtsova in \cite{CFP13}.
Given a $p$-subalgebra $\fh \subseteq \fg$, the canonical inclusion provides a closed immersion
\[ \EE(r,\fh) \hookrightarrow \EE(r,\fg),\]
whose image is $\{\fe \in \EE(r,\fg) \ ; \ \fe \subseteq \fh\}$. We shall henceforth identify these two spaces without further notice.

Carlson's result \cite{Ca} implies that the variety $\EE(1,\fg)$ is connected. This is no longer the case for $\EE(2,\fg)$. In fact, as the following example
illustrates, the property of connectedness may depend on the characteristic of $k$.

\bigskip

\begin{Example} We consider the $3$-dimensional Heisenberg algebra $\fh := kx\!\oplus\!ky\!\oplus kz$, whose bracket and $p$-map are given by
\[ [x,y]=z \ , \ [x,z]=0=[y,z] \ \ ; \ \ x^{[p]}=0 \ , \ y^{[p]}=0 \ , \ z^{[p]}=0,\]
respectively.

If $\fe \in \EE(2,\fh)$ satisfies $\fe\cap kz=(0)$, then $\fh=\fe\!\oplus\!kz$, so that $[\fh,\fh] = [\fe,\fe] = (0)$, a contradiction. Consequently, $kz \subseteq \fe$ for every $\fe \in \EE(2,\fh)$.

Our general assumption $p\!\ge\!3$ yields $\fh^{[p]}=\{0\}$, and the foregoing observation implies that the variety $\EE(2,\fh)\cong \PP^1$ is irreducible. However, for $p\!=\!2$, an element
$\fe \in \EE(2,\fh)$ is of the form $\fe := k(\alpha x\!+\!\beta y)\!\oplus\!kz$, with $\alpha\beta = 0$ and $(\alpha,\beta)\in k^2\!\smallsetminus\!\{0\}$. Consequently, the variety $\EE(2,\fh)=\{kx\!\oplus\!kz,
ky\!\oplus\!kz\}$ is not connected. \end{Example}

\bigskip

\begin{Remark} Let $\cG \subseteq \GL(n)$ be a closed subgroup scheme of exponential type and consider the variety $V_2(\cG)=\Hom(\GG_{a(2)},\cG)$ of infinitesimal one-parameter
subgroups (cf.\ \cite[\S1]{SFB1}). In view of \cite[(1.7)]{SFB1}, the exponential maps provide an isomorphism between $V_2(\cG)$ and the commuting variety $C_2(V(\fg))$ of the nullcone
associated to the Lie algebra $\fg:=\Lie(\cG)$. The variety $\EE(2,\fg)$ is an image of an open subset of $C_2(V(\fg))$, cf.\ \cite[(1.4)]{CFP13}. Moreover, Carlson's Theorem \cite[(7.7)]{SFB2}
ensures that the projective space $\Proj(C_2(V(\fg))$, defined by the action
\[ \alpha\dact(x,y)=(\alpha x,\alpha^py).\]
of $k^\times$ on $\fg\!\times\!\fg$, is connected. The example above shows that for the Heisenberg group $\cH$ and $p\!=\!2$, the variety $\Proj(C_2(V(\fh)))$ is connected, while $\EE(2,\fh)$ is not.
\end{Remark}

\bigskip

\subsection{A basic criterion}
Let $(\fg,[p])$ be a restricted Lie algebra. We define the \textit{$p$-rank} of $(\fg,[p])$ via
\[ \msrke(\fg) := \max \{\dim_k\fe \ ; \ \fe \subseteq \fg \ \text{elementary abelian}\}.\]
If $G$ is a closed subgroup of the automorphism group $\Aut_p(\fg)$ of $(\fg,[p])$, then $G$ acts on the variety $\EE(r,\fg)$ via
\[ g\dact \fe := g(\fe)\]
for all $g \in G$ and $\fe \in \EE(r,\fg)$.

Recall that $\Lie(G_\fg) \subseteq \Der \fg$ is a $p$-subalgebra of the algebra of derivations of $\fg$. Thus, if $\rho : H \lra G_\fg$ is a homomorphism of algebraic groups, its
differential $\msd(\rho)$ sends $\Lie(H)$ to $\Der\fg$.

\bigskip

\begin{Lem} \label{Bc1}Let $H$ be a connected, solvable algebraic group, $\rho:H\lra G_\fg$ be a homomorphism of algebraic groups such that there is a $p$-subalgebra
$\fh \subseteq \fg$ with the following properties:
\begin{enumerate}
\item[(a)] $\msrke(\fh)\!\ge\!2$, and
\item[(b)] there is a homomorphism $\zeta : \fh \lra \Lie(H)$ of restricted Lie algebras such that
                \[ \msd(\rho)(\zeta(h))(x)=[h,x]\]
for all $h \in \fh$, $x \in \fg$. \end{enumerate}
Then the variety $\EE(2,\fg)$ is connected. \end{Lem}

\begin{proof} The variety $\EE(2,\fg)$ affords only finitely many irreducible components. Hence it also has only finitely many connected components.
We let $H$ act on $\EE(2,\fg)$ via
\[ h\dact \fe := \rho(h)\dact\fe \  \   \   \  \  \  \  \  \ \forall \ h \in H, \, \fe \in \EE(2,\fg).\]
Let $X\subseteq \EE(2,\fg)$ be a connected component. Then $X$ is closed, and the stabilizer $H_X := \{h \in H\ ; \ h\dact X \subseteq X\}$ is a closed subgroup
of $H$ of finite index. As $H$ is connected, it follows that $H_X=H$, so that $X$ is $H$-stable. Since $H$ is solvable and connected, Borel's fixed point theorem \cite[(7.2.5)]{Sp}
provides an element $\fe_X \in X$ such that $H\dact \fe_X = \{\fe_X\}$.

Differentiation in conjunction with (b) yields $[\fh,\fe_X] \subseteq \fe_X$.  In view of (a), there is $\fe_0 \in \EE(2,\fh)$. Engel's Theorem readily implies
that
\[ (\ast) \ \ \ \ \ \ \ [(\ad a)\circ (\ad b)](\fe_X) = (0) \ \ \ \text{for all} \ a,b \in \fe_0.\]
Hence $\fl:=\fe_X\!+\!\fe_0$ is a $p$-subalgebra of dimension $2\! \le\! \dim_k\fl\! \le\! 4$ such that $[\fl,\fl] \subseteq [\fe_0,\fe_X] \subsetneq \fe_X$ has dimension $\le 1$.  Jacobson's formula
(cf.\ \cite[(II.1)]{SF}) in conjunction with $p\ge 3$ now yields $\fl \subseteq V(\fg)$.

Let $\fl' \subseteq \fl$ be a subalgebra. Then
\[ (\ast\ast) \ \ \ \ \ \ \ \EE(2,\fl') \ \text{is connected and} \ \EE(2,\fl') \cap X \ne \emptyset \ \Rightarrow \ \EE(2,\fl') \subseteq X.\]
If $\fl$ is abelian, then $\EE(2,\fl)=\Gr_2(\fl)$ is irreducible. As $\fe_X,\fe_0 \in \EE(2,\fl)$, implication ($\ast\ast$) gives $\fe_0 \in X$.

Alternatively, $\dim_k[\fl,\fl] = \dim_k[\fe_0,\fe_X]=1$, and $\dim_k\fl \in \{3,4\}$. If $\dim_k\fl=3$, then $\fl$ is the Heisenberg algebra of the example above. Thus, $\EE(2,\fl)\cong \PP(\fl/[\fl,\fl])$ is irreducible 
and the above arguments yield $\fe_0 \in X$.

If $\dim_k\fl = 4$, then $\dim_k C(\fl) = 2$. Since $\fe_X\!+\!C(\fl)$ is abelian, implication ($\ast\ast$) yields $C(\fl) \in X$. It now follows that $\fe_0\!+\!C(\fl)$ is
abelian with $\EE(2,\fe_0\!+\!C(\fl))\cap X \ne \emptyset$, whence $\fe_0 \in X$.

In sum, we have shown that $\fe_0 \in X$ in all cases, so that $X$ is the only connected component of $\EE(2,\fg)$. \end{proof}

\bigskip

\begin{Remark} The interested reader is referred to \cite{CF} for a description of the restricted Lie algebras of $p$-rank $\msrke(\fg)\! \le\! 1$.\end{Remark}

\bigskip
\noindent
We record the following important consequence:

\bigskip

\begin{Cor} \label{Bc2} Let $(\fg,[p])$ be a restricted Lie algebra. Suppose there exists $\fe \in \EE(2,\fg)$ such that
\begin{enumerate}
\item[(a)] $\fe = \langle \fe \cap \cN_{\frac{p+1}{2}}(\fg) \rangle$, and
\item[(b)] $\exp(\fe\cap\cN_{\frac{p+1}{2}}(\fg)) \subseteq \Aut_p(\fg)$. \end{enumerate}
Then $\EE(2,\fg)$ is connected. \end{Cor}

\begin{proof} We consider the exponential map
\[ \exp : \fe \lra \GL(\fg).\]
As $\fe\cap \cN_{\frac{p+1}{2}}(\fg)$ is closed and conical, condition (b) ensures that $\exp(\fe\cap\cN_{\frac{p+1}{2}}(\fg)) \subseteq G_\fg$.

In view of condition (a), there are $x,y \in \cN_{\frac{p+1}{2}}(\fg)$ such that $\fe = kx\!\oplus\!ky$. Given elements $a=a_xx\!+\!a_yy$ and $b=b_xx\!+\!b_yy$ of $\fe$, repeated application of
Lemma \ref{Em1}(2) yields
\[ \exp(a\!+\!b) = \exp((a_x\!+\!b_x)x)\circ \exp((a_y\!+\!b_y)y) = \exp(a_xx)\circ\exp(b_xx)\circ\exp(a_yy)\circ \exp(b_yy),\]
so that $\exp(a\!+\!b) \in G_\fg$ and $\exp(a\!+\!b)=\exp(a)\circ\exp(b)$. As a result, the map $\exp : \fe \lra G_\fg$ is a homomorphism of algebraic groups such that $\msd(\exp)(a)= \ad a$
for all $a \in \fe$. Setting $H:= \fe \cong \GG^2_a$, we obtain $\Lie(H)=\fe$. The assertion now follows from Lemma \ref{Bc1} with $\rho:=\exp$ and $\zeta$ being the identity of $\fh:=\fe$.
\end{proof}

\bigskip

\begin{Cor} \label{Bc3}Suppose that $\fn \unlhd \fg$ is an elementary abelian $p$-ideal of dimension $\ge\! 2$. Then $\EE(2,\fg)$ is connected. \end{Cor}

\begin{proof}  Since $\fn$ is an abelian ideal, it follows that $\fe \subseteq \cN_{\frac{p+1}{2}}(\fg)$, so that (a) of Corollary \ref{Bc2} holds. Condition (b) is a direct consequence of Lemma \ref{Em2}(4).
Consequently, Corollary \ref{Bc2} yields the assertion. \end{proof}

\bigskip

\begin{Remarks} (1) The example of the Heisenberg algebra shows that Corollary \ref{Bc3} does not hold for $p\!=\!2$.

(2) Suppose that $\fg$ is centerless. Then every abelian ideal is elementary abelian.

(3) Let $(\fg,[p])$ be a $p$-trivial Lie algebra, i.e.,  $[p]\!=\!0$. If $\dim_k\fg \ge 2$, Engel's Theorem provides an elementary abelian ideal $\fn \unlhd \fg$ of dimension $\ge 2$. Hence the
variety $\EE(2,\fg)$ is connected. However, it is not necessarily irreducible, see Section \ref{S:IrrD} below.

(4) According to Corollary \ref{Bc3}, the variety $\EE(2,\fg)$ is connected, whenever $\dim V(C(\fg))\!\ge\!2$. \end{Remarks}

\bigskip

\begin{Cor} \label{Bc4} Let $(\fg,[p])$ be a centerless restricted Lie algebra. If $\fg\!=\! \fg_{(-r)} \supsetneq \cdots \supsetneq \fg_{(s)} \supsetneq \fg_{(s+1)}=(0) \ \ (r,s \in \NN)$
is filtered such that
\begin{enumerate}
\item[(a)] $\frac{p-1}{2}(s\!-\!1)\!\ge\! r\!+\!2$, or
\item[(b)] $\dim_k\fg_{(s)}\!\ge\!2$ and $\frac{p-1}{2}s\!\ge\! r\!+\!1$,
\end{enumerate}
then $\EE(2,\fg)$ is connected.\end{Cor}

\begin{proof} Conditions (a) and (b) ensure that there is $\fe \in \EE(2,\fg_{(s-1)})$ or $\fe \in \EE(2,\fg_{(s)})$ such that $\fe \subseteq \cN_{\frac{p+1}{2}}(\fg)$. The assertion thus follows from
Corollary \ref{Bc2}. \end{proof}

\bigskip

\subsection{Algebraic Lie algebras}
We shall employ our criterion (\ref{Bc1}) to treat the case, where $\fg\! =\! \Lie(G)$ is the Lie algebra of a connected algebraic group $G$.

Let $G$ be a connected algebraic group with Lie algebra $\fg$. A $p$-subalgebra $\fb \subseteq \fg$ is called a {\it Borel subalgebra}, if there exists a Borel subgroup $B \subseteq G$ such that $\fb\! =\! \Lie(B)$.
We denote by $\Bor(G)$ and $\Bor(\fg)$ the sets of Borel subgroups and Borel subalgebras, respectively.

Let $T\subseteq G$ be a maximal torus, whose normalizer and centralizer in $G$ will be denoted $N_G(T)$ and $C_G(T)$, respectively. Then $W(G,T)\!:=\!N_G(T)/C_G(T)$ is the {\it Weyl group} of $G$
relative to $T$. We denote by $U_G$ the unipotent radical of $G$ and consider the reductive group $G':= G/U_G$ along with the canonical projection $\pi : G \lra G'$. By general theory
(cf.\ \cite[(21.3C),(24.1B)]{Hu81}), the group $T'\!:=\!\pi(T)$ is a maximal torus of $G'$ and $\pi$ induces an isomorphism $W(G,T) \cong W(G',T')$. By the same token, we have $\Bor(G')=\{\pi(B) \ ; \ B \in
\Bor(G)\}$.

Let $B \subseteq G$ be a Borel subgroup containing $T$. Since $T'$ is self-centralizing (cf.\ \cite[(22.3),(26.2)]{Hu81}), it follows that $C_G(T)\subseteq U_GT \subseteq B$. Thus, if $g,h \in
N_G(T)$ are representatives of $w \in W(G,T)$, then we have $Bg\!=\!Bh$, so that the coset $Bw\!:=\!Bg$ is well-defined.

Let $R_G$ be the solvable radical of $G$, $\hat{\Phi}$ be the root system of the semisimple group $\hat{G}\!:=\!G/R_G$. Following \cite{LMT}, we say that a prime number $q$ is a {\it torsion prime for $G$},
if it is a torsion prime for some irreducible component of the root system $\hat{\Phi}$, or if $q$ is a divisor of the order of the fundamental group $\pi_1(\hat{G})$ of $\hat{G}$. If $G$ is solvable, then, by definition, 
no prime number is a torsion prime for $G$. A complete list of the torsion primes for irreducible root systems can be found in \cite[(1.13)]{St}. We refer the reader to \cite[(9.14)]{MT11} for the definition of the fundamental 
group $\pi_1(G)$ of a connected semisimple algebraic group $G$. Note that $\pi_1(G)$ is a finite group, whose order $\ord(\pi_1(G))$ divides $\det(C_\Phi)$, the determinant of the Cartan matrix $C_\Phi$ associated to the root system $\Phi$ of $G$. The group $G$ is simply connected if and only if $\pi_1(G)\!=\{1\}$.

We require the following:

\bigskip

\begin{Lem} \label{AL1} Let $G$ be a connected algebraic group with Lie algebra $\fg$. Then $\bigcap_{\fb \in \Bor(\fg)} \fb$ is the solvable radical of $\fg$. \end{Lem}

\begin{proof} We put $\hat{G} := G/R_G$ as well as $\hat{\fg} := \Lie(\hat{G})$ and denote by $\hat{\pi} : G \lra \hat{G}$ the canonical projection. As $\hat{\pi}$ is separable, there results
an exact sequence
\[ (0) \lra \Lie(R_G) \lra \fg \stackrel{\msd(\hat{\pi})} \lra \hat{\fg}\lra (0),\]
see \cite[(5.2.3)]{Sp} and \cite[(3.2.21)]{Sp}. Since $\Lie(R_G)$ is a solvable ideal of $\fg$, it is contained in the solvable radical $\fr$ of $\fg$. In view of \cite[(10.2)]{Hu67}, it follows that $\fr/\Lie(R_G)
\cong C(\hat{\fg})$ is a torus.

Let $B \in \Bor(G)$ be a Borel subgroup of $G$, $T \subseteq B$ be a maximal torus of $G$. Owing to \cite[(13.2)]{Hu67}, $\ft:=\Lie(T) \subseteq \Lie(B)$ is a maximal
torus of $\fg$. As $\msd(\hat{\pi})$ is surjective, $\msd(\hat{\pi})(\ft)$ is a maximal torus of $\hat{\fg}$ (cf.\ \cite[(II.4.5)]{SF}), so that $C(\hat{\fg}) \subseteq \msd(\hat{\pi})(\ft)$. This implies $\fr
\subseteq \Lie(R_G)\!+\!\ft$. In view of $R_G\subseteq B$, we obtain $\fr \subseteq \Lie(B)\!+\!\ft \subseteq \Lie(B)$, whence $\fr \subseteq \bigcap_{B \in \Bor(G)} \Lie(B) =
\bigcap_{\fb \in \Bor(\fg)}\fb$.

Since the set $\bigcap_{B \in \Bor(G)}\Lie(B)$ is $G$-stable, it is a solvable ideal of $\fg$ and hence is contained in $\fr$. \end{proof}

\bigskip

\begin{Remark} Our result fails for $p\!=\!2$. In that case, $\fsl(2)\!=\!\Lie(\SL(2))$ is solvable, while $\bigcap_{\fb\in \Bor(\fsl(2))}\fb = kI_2$ is the center of $\fsl(2)$. \end{Remark}

\bigskip
\noindent
The adjoint representation $\Ad : G \lra \Aut_p(\fg)$ induces an action of $G$ on $\EE(2,\fg)$:
\[ g\dact\fe := \Ad(g)(\fe) \ \ \ \ \ \ \  \forall \ g \in G, \, \fe \in \EE(2,\fg).\]

\bigskip

\begin{Thm} \label{AL2} Let $G$ be a connected algebraic group with Lie algebra $\fg$ such that $p$ is a non-torsion prime for $G$. Then the following statements hold:
\begin{enumerate}
\item We have $\EE(2,\fg)=\bigcup_{\fb \in \Bor(\fg)}\EE(2,\fb)$.
\item If $\fb \in \Bor(\fg)$, then $\EE(2,\fg) = G\dact\EE(2,\fb)$.
\item If $B \subseteq G$ is a Borel subgroup with maximal torus $T$, then $\EE(2,\fg)=\bigcup_{w \in W(G,T)}Bw\dact\EE(2,\Lie(B))$.
\end{enumerate} \end{Thm}

\begin{proof} (1) We put $\fr\!:=\!\Lie(R_G)$ as well as $\hat{\fg}\! := \!\Lie(\hat{G})$, where $\hat{G}\!:=\!G/R_G$. Since the natural projection $\hat{\pi} : G \lra \hat{G}$ is separable (\cite[(5.2.3)]{Sp}),
there results an exact sequence
\[ (0) \lra \fr \lra \fg \stackrel{\msd(\hat{\pi})}{\lra} \hat{\fg}\lra (0)\]
of restricted Lie algebras, cf.\  \cite[(3.2.21)]{Sp}.

Let $\fe \in \EE(2,\fg)$ and put $\hat{\fe}\!:=\! \msd(\hat{\pi})(\fe) \subseteq \hat{\fg}$. We denote by $\hat{\cU} \subseteq \hat{G}_1$ the unipotent subgroup scheme of the first Frobenius kernel
$\hat{G}_1$ of $\hat{G}$ such that $\hat{\fe}=\Lie(\hat{\cU})$, cf.\ \cite[(II,\S7,4.3)]{DG}. Since $p$ is a non-torsion prime for $\hat{G}$, \cite[(2.2)]{LMT} ensures the existence of a Borel
subgroup $\hat{B}\subseteq \hat{G}$ such that $\hat{\cU} \subseteq \hat{B}$. General theory provides a Borel subgroup $B_\fe \subseteq G$ such that $\hat{B}=\hat{\pi}(B_\fe)$ (cf.\
\cite[(21.3C)]{Hu81}). Since $R_G\subseteq B_\fe$ and $\hat{\pi}|_{B_\fe}$ is separable, we have $\Lie(\hat{B})=\msd(\hat{\pi})(\Lie(B_{\fe}))$.

Now let $x \in \fe$. By the above, there exists $y \in \Lie(B_\fe)$ such that $x\!-\!y \in \fr$. As Lemma \ref{AL1} yields $\fr \subseteq\Lie(B_\fe)$, we obtain $x \in \Lie(B_\fe)$.
Thus, $\fe \subseteq \Lie(B_\fe)$. This implies (1).

(2) Let $\fb=\Lie(B)$ be a Borel subalgebra. Since all Borel subgroups of $G$ are conjugate, (1) yields $\EE(2,\fg)=\bigcup_{g \in G}\EE(2,\Ad(g)(\fb)) =
\bigcup_{g \in G}g\dact\EE(2,\fb) = G\dact \EE(2,\fb)$.

(3) Let $B \subseteq G$ be a Borel subgroup of $G$ containing the maximal torus $T$. Setting $G':= G/U_G$ and letting $\pi : G \lra G'$ be the canonical projection, we put $B':= \pi(B)$
and $T':=\pi(T)$. Then the Bruhat decomposition for $G'$ yields
\[ G' = \bigsqcup_{w' \in W(G',T')} B'w'B'.\]
Let $g \in G$. Then there exist $b_1,b_2 \in B$ and $\tilde{w} \in N_G(T)$ such that $\pi(g)=\pi(b_1)\pi(\tilde{w})\pi(b_2)$. Consequently, $u:= g(b_1\tilde{w}b_2)^{-1} \in U_G\subseteq B$,
so that $g= ub_1\tilde{w}b_2 \in B\tilde{w}B$. It thus follows that
\[ G = \bigcup_{w \in W(G,T)}BwB.\]
Let $\fb:=\Lie(B)$, so that $\EE(2,\fb)$ is $B$-stable. Now (2) implies $\EE(2,\fg)=  \bigcup_{w \in W(G,T)}BwB\dact\EE(2,\fb) = \bigcup_{w \in W(G,T)}Bw\dact
\EE(2,\fb)$. \end{proof}

\bigskip

\begin{Remarks} (1) The foregoing result holds more generally for the varieties $\EE(r,\fg)$ $(r\!\ge\!1)$. 

(2) Consider the simply connected covering
\[ \pi : \SL(p) \lra \PGL(p) \ \ ; \ \ g \mapsto gZ(\GL(p)).\]
We have $\ord(\pi_1(\PGL(p)))\!=\!p$ and $\ker\msd(\pi)\!=\!C(\fsl(p))$. In \cite[\S3]{LMT}, the authors construct a three-dimensional $p$-subalgebra $\fh \subseteq \fsl(p)$ containing $C(\fsl(p))$, which is not trigonalizable, 
and such that $\msd(\pi)(\fh) \in \EE(2,\fpgl(p))$. If $\msd(\pi)(\fh) \subseteq \Lie(B')$ for some $B' \in \Bor(\PGL(p))$, then $\fh \subseteq \msd(\pi)^{-1}(\Lie(B'))$. Let $B \in \Bor(\SL(p))$ be such that $B'\!=\!\pi(B)$ and 
$B_0$ be the Borel subgroup of $\SL(p)$ consisting of upper triangular matrices. Then $C(\fsl(p)) \subseteq \Lie(B_0)$, so that conjugacy of Borel subgroups implies $C(\fsl(p)) \subseteq \Lie(B)$. This implies
$\dim_k \im \msd(\pi)|_{\Lie(B)} = \dim_k \im\msd(\pi)\cap \Lie(B')$, so that $\msd(\pi)^{-1}(\Lie(B'))\!=\!\Lie(B)$. We therefore obtain $\fh\!=\!\msd(\pi)^{-1}(\msd(\pi)(\fh)) \subseteq \msd(\pi)^{-1}(\Lie(B')) \subseteq \Lie(B)$,
which contradicts $\fh$ is not being trigonalizable. As a result, $\PGL(p)\dact \EE(2,\fb') \subsetneq \EE(2,\fpgl(p))$ for every $\fb' \in \Bor(\fpgl(p))$. \end{Remarks} 

\bigskip

\begin{Thm} \label{AL3} Let $\fg = \Lie(G)$ be the Lie algebra of a connected algebraic group $G$. Then the variety $\EE(2,\fg)$ is connected. \end{Thm}

\begin{proof} Let $B \subseteq G$ be a Borel subgroup with Lie algebra $\fb$. Setting $H\!:=\!B$, $\fh\!:=\!\fb$, $\zeta\!:=\!\id_\fb$ and $\rho\!:=\!\Ad|_B : B \lra G_\fg$ in Lemma \ref{Bc1}, we see that
$\EE(2,\fg)$ is connected, whenever $\msrke(\fb)\!\ge\!2$.

If $\msrke(\fb)\!=\!1$, then \cite[(4.1.1)]{CF} ensures that
\[ \fb = \ft\!\ltimes\!(kx)_p\]
for some torus $\ft \subseteq \fg$ and some $p$-nilpotent element $x \in \fg\!\smallsetminus\!\{0\}$. Here $(kx)_p:=\sum_{n \ge 0} kx^{[p]^n}$ is the $p$-unipotent radical of $\fb$. Thanks to
\cite[(13.3)]{Hu67}, $\ft = \Lie(T)$ is the Lie algebra of a maximal torus $T$ of $B$. As $T$ acts on $(kx)_p$ via homomorphisms of restricted Lie algebras, there is a character $\alpha:T \lra k^\times$
such that the root space decomposition of $\fb$ relative to $T$ has the form
\[ \fb = \Lie(T)\!\oplus\!\bigoplus_{n\ge 0} \fb_{p^n\alpha}.\]
Let $\hat{\pi} : G \lra \hat{G}:=G/R_G$ be the canonical projection. In view of \cite[(21.3C)]{Hu81} and \cite[(27.1)]{Hu81}, the Borel subalgebra $\hat{\fb} = \msd(\hat{\pi})(\fb)$ has only one root
relative to the maximal torus $\hat{\pi}(T)$ of $\hat{G}$. Consequently, $\hat{G}$ has rank $1$. Thus, $\hat{G}$ is almost simple with root system $A_1$ and the order of fundamental group $\pi_1(\hat{G})$
divides $2$. Now \cite[(1.13)]{St} shows that the prime $p\!\ge\! 3$ is a non-torsion prime for $G$. Theorem \ref{AL2} yields $\EE(2,\fg)\!=\!G\dact\EE(2,\fb)\!=\!\emptyset$.

In the remaining case where $\msrke(\fb)\!=\!0$, the Lie algebra $\fb$ is a torus, so that $B\!=\!T$ and $G\!=\!\bigcup_{w\in W(G,T)}TwT=\bigcup_{w\in W(G,T)}wT$. Since $G$ is connected, we conclude
that $G\!=\!T$ is a torus, whence $\EE(2,\fg)\!=\!\emptyset$. \end{proof}

\bigskip

\begin{Lem} \label{AL4} Let $G$ be a connected algebraic group with Borel subgroup $B\subseteq G$. We put $\fg:=\Lie(G)$ and $\fb:=\Lie(B)$. If $\fn \unlhd \fg$ is a $G$-stable $p$-ideal
such that there is $\fe \in \EE(2,\fb)$ with $\fe\cap \fn = (0)$, then $\EE(2,\fg/\fn)$ is connected. \end{Lem}

\begin{proof} By assumption, the group $B$ acts on $\fg':=\fg/\fn$ via
\[ \widetilde{\Ad} : B \lra G_{\fg'} \ \ ; \ \ \widetilde{\Ad}(b)(x\!+\!\fn) = \Ad(b)(x)\!+\!\fn.\]
We put $H:=\widetilde{\Ad}(B) \subseteq G_{\fg'}$, let $\rho : H \hookrightarrow G_{\fg'}$ be the canonical inclusion, and set $\fh:= \{b\!+\!\fn \ ; \ b \in \fb\}$. Since
\[ \msd(\widetilde{\Ad})(b)(x\!+\!\fn) = [b,x]\!+\!\fn\]
for all $b \in \fb$ and $x \in \fg$, it follows that $\ad(\fh)\subseteq \Lie(H)$. Consequently,
\[ \zeta : \fh \lra \Lie(H) \ \ ; \ \ b\!+\!\fn \lra \ad(b\!+\!\fn)\]
is a homomorphism of restricted Lie algebras such that
\[ \msd(\rho)(\zeta(b\!+\!\fn))(x\!+\!\fn) = \zeta(b\!+\!\fn)(x\!+\!\fn)= [b\!+\!\fn,x\!+\!\fn]\]
for all $b \in \fb$ and $x \in \fg$.

Our assumption on $\EE(2,\fb)$ yields $\msrke(\fh)\!\ge\! 2$. The assertion thus follows from Lemma \ref{Bc1}. \end{proof}

\bigskip

\begin{Cor} \label{AL5} Let $\fg$ be a classical semisimple Lie algebra. Then $\EE(2,\fg)$ is connected. \end{Cor}

\begin{proof} Since $\fg$ is classical semisimple,
\[ \fg=\bigoplus_{i=1}^n\fg_i\]
is a direct sum of classical simple Lie algebras. Hence there exist for each $i \in \{1,\ldots,n\}$ a connected algebraic group $\tilde{G}_i$ and a toral $\tilde{G}_i$-stable $p$-ideal $\tilde{\fn}_i$
of $\tilde{\fg}_i:=\Lie(\tilde{G}_i)$ such that $\fg_i \cong \tilde{\fg}_i/\tilde{\fn}_i$. Moreover, we have $\tilde{\fn}_i \ne (0)$ if and only if $\tilde{G}_i = \SL(n_ip)$.

Setting $\tilde{G}:=\prod_{i=1}^n\tilde{G}_i$ as well as $\tilde{\fg}:=\Lie(\tilde{G})$, we have $\fg \cong \tilde{\fg}/\tilde{\fn}$, where $\tilde{\fn} := \bigoplus_{i=1}^n\tilde{\fn}_i$ is a toral ideal of
$\tilde{\fg}=\bigoplus_{i=1}^n\tilde{\fg}_i$.

Let $\tilde{\fb}_i$ be a Borel subalgebra of $\tilde{\fg}_i$. It follows from \cite[(21.3C)]{Hu81} that $\tilde{\fb}:=\bigoplus_{i=1}^n \tilde{\fb}_i$ is a Borel subalgebra of $\tilde{\fg}$. If $\EE(2,\tilde{\fb})
=\emptyset$, then $n\!=\!1$ and $\tilde{\fg} = \fsl(2)$. Since $p\!\ge\!3$, it follows that $\EE(2,\tilde{\fb})\ne\emptyset$, whenever $\tilde{\fn}\ne(0)$. In view of $\tilde{\fn}$ being a torus, Lemma \ref{AL4}
thus shows that $\EE(2,\fg)$ is connected in case $\tilde{\fn} \ne (0)$. Alternatively, our result follows from Theorem \ref{AL3}. \end{proof}

\bigskip

\subsection{Admissible Lie algebras with strong degeneration}
We are interested in Lie algebras $\fg$ {\it with strong degeneration}, i.e.\ the case where $\Sw(\fg)\!\ne\!\{0\}$. We denote by $\langle \Sw(\fg) \rangle$ the linear span of the sandwich elements
of $\fg$. Owing to Lemma \ref{Em2}, $\langle \Sw(\fg) \rangle$ is a Lie subalgebra of $\fg$ containing the center $C(\fg)$. Thus, if $(\fg,[p])$ is a restricted Lie algebra, then $\langle \Sw(\fg) \rangle$
is a $p$-subalgebra of $\fg$.

Our approach necessitates the consideration of restricted Lie algebras, whose subsets $\cN_3(\fg)$ satisfy certain technical conditions that turn out to be fulfilled in most cases of interest.

\bigskip

\begin{Definition} A restricted Lie algebra $(\fg,[p])$ is called \textit{admissible}, provided
\begin{enumerate}
\item[(i)] $\cN_3(\fg) \subseteq V(\fg)$, and
\item[(ii)] $\exp(\cN_3(\fg)) \subseteq \Aut_p(\fg)$. \end{enumerate} \end{Definition}

\bigskip

\begin{Remarks} (1) If $p\!\ge\!5$ and $\fg$ is centerless, then $\fg$ is admissible.

(2) The four-dimensional restricted Lie algebra $\fsl(2)_s$ to be discussed in Section \ref{S:Ex} below satisfies (i), but not (ii). \end{Remarks}

\bigskip

\begin{Lem} \label{SL1} Suppose that $p\!\ge\!5$. If $(\fg,[p])$ is a restricted Lie algebra, then $\fg/C(\fg)$ is admissible. \end{Lem}

\begin{proof} We put $\fg':=\fg/C(\fg)$. Given $x:=a\!+\!C(\fg) \in \cN_3(\fg')$, we have $(\ad a)^4(\fg)=(0)$, whence $a^{[p]}\in C(\fg)$. Consequently, $x \in V(\fg')$.

Let $x:=a\!+\!C(\fg) \in \cN_3(\fg')$ be an element, so that $(\ad a)^3(\fg) \subseteq C(\fg)$ (and $(\ad a)^4=0$). If $b,c \in \fg$, then the Leibniz rule for $(\ad a)^3$ and $(\ad a)^4$ implies
\[ \exp(a)([b,c]) \equiv [\exp(a)(b),\exp(a)(c)] \ \ \ \ \ \modd C(\fg)\]
for all $b,c \in \fg$. Hence
\[ [\exp(a)(b)^{[p]}\!-\!\exp(a)(b^{[p]}),\exp(a)(c)] \equiv 0 \ \ \modd C(\fg) \ \ \ \ \ \forall \ b,c \in \fg,\]
so that $\exp(x) \in \Aut_p(\fg')$. As a result, the restricted Lie algebra $\fg'$ is admissible. \end{proof}

\bigskip

\begin{Prop} \label{SL2} If $(\fg,[p])$ is an admissible restricted Lie algebra such that $\Sw(\fg)$ is not an ideal of dimension $\le 1$, then $\EE(2,\fg)$ is connected. \end{Prop}

\begin{proof} We first assume that $\Sw(\fg)$ contains an element $c\ne 0$ such that $kc$ is not an ideal. We recall that
\[ \fg:= \fg_{(-1)} \supseteq \fg_{(0)} \supseteq \fg_{(1)} \supseteq (0),\]
defined via $\fg_{(1)} := \im(\ad c)\!+\!kc$ and $\fg_{(0)}:=\ker(\ad c)$ is a filtration of $\fg$:
\begin{enumerate}
\item[(a)] Since $(\ad c)^2=0$, the Leibniz rule readily implies that $[\im(\ad c),\im(\ad c)]=(0)$, so that $\fg_{(1)}$ is an abelian subalgebra of $\fg$.
\item[(b)] If $x \in \fg$ and $y\! =\! [c,z]\!+\!\alpha c \in \fg_{(1)}$, then (a) implies $(\ad c)([x,y]) = [[c,x],[c,z]] = 0$, whence $[\fg_{(-1)},\fg_{(1)}] \subseteq \fg_{(0)}$.
\item[(c)] Since $\fg_{(0)}$ is a subalgebra of $\fg$, we have $[\fg_{(0)},\fg_{(0)}] \subseteq \fg_{(0)}$.
\item[(d)] If $x \in \fg_{(0)}$ and $y = [c,z]\!+\!\alpha c \in \fg_{(1)}$, then $[x,y] = [x,[c,z]] = [c,[x,z]]$, whence $[\fg_{(0)},\fg_{(1)}] \subseteq \fg_{(1)}$. \end{enumerate}
Given $x \in \fg_{(1)}$, we thus obtain
\[ (\ad x)^3(\fg) \in [\fg_{(1)},(\ad x)^2(\fg)] \subseteq [\fg_{(1)},\fg_{(1)}] = (0),\]
so that $\fg$ being admissible ensures that $\fg_{(1)}$ is an elementary abelian ideal of $\fg_{(0)}$. If $\dim_k\fg_{(1)}=1$, then $\im (\ad c) \subseteq kc$, implying that $kc$ is an ideal, a contradiction.
Corollary \ref{Bc3} now shows that the variety $\EE(2,\fg_{(0)}) \ne \emptyset$ is connected.

Given $\alpha \in k$, we consider the linear map
\[ f_\alpha : \fg \lra \fg \ \ ; \ \ x \mapsto \alpha x\!+\![c,x].\]
There results a morphism
\[ k \lra \End_k(\fg\wedge\fg) \ \ ; \ \ \alpha \mapsto f_\alpha \wedge f_\alpha\]
of affine varieties. Thus, if $v \in \fg\wedge \fg$ is such that $(f_\alpha \wedge f_\alpha)(v) \ne 0$ for all $\alpha \in k$, then
\[ k \lra \PP(\fg\wedge\fg) \ \ ; \ \ \alpha \mapsto k(f_\alpha \wedge f_\alpha)(v)\]
is a morphism. In particular, if $V\in \Gr_2(\fg)$ is a subspace such that $\dim_kf_\alpha(V)=2$ for all $\alpha \in k$, then
\[ f_V : k \lra \Gr_2(\fg) \ \ ; \ \ \alpha \mapsto f_\alpha(V)\]
is a morphism, cf.\ \cite[(2.1.2)]{Fa17}.

Let $\cX \subseteq \EE(2,\fg)$ be a connected component. We put
\[ d_\cX := \min \{\dim_k\fe\cap\fg_{(0)} \ ; \ \fe \in \cX\}.\]
If $d_\cX=0$, then there is $\fe \in \cX$ with $\fe\cap\fg_{(0)}=(0)$. Since $\fg$ is admissible, we have $\exp(\alpha c)(\cX)\subseteq \cX$ for all $\alpha \in k$, and it follows that $f_\alpha(\fe) \in
\cX$ for all $\alpha \in k^\times$. If $\dim_kf_0(\fe) \le1$, then $\fe\cap\fg_{(0)}\ne (0)$, a contradiction. By the above, the map
\[ f_\fe : k \lra \EE(2,\fg) \ \ ; \ \ \alpha \mapsto f_\alpha(\fe)\]
is a morphism such that $f_\fe(k^\times)\subseteq \cX$. As $\cX$ is closed, we have $f_0(\fe) \in \cX\cap\EE(2,\fg_{(0)})$, and $\EE(2,\fg_{(0)})$ being connected yields $\EE(2,\fg_{(0)}) \subseteq \cX$.

If $d_\cX = 1$, there is $\fe \in \cX$ such that $\fe = kx\!\oplus\!ky$, where $x \not\in \fg_{(0)}$ and $y \in \fg_{(0)}$. Application of $\exp(\alpha c)$ then shows that
\[ \fe_\alpha:= kf_\alpha(x)\!\oplus\!ky \in \cX \ \ \ \  \ \ \text{for all} \ \alpha \in k^\times.\]
Note that the subalgebra $\fe_0:=k[c,x]\!+\!ky$ is elementary abelian. If $\dim_k\fe_0 =2$, then
\[ \varphi_\fe : k \lra \EE(2,\fg) \ \ ; \ \ \alpha \mapsto kf_\alpha(x)\!\oplus ky\]
is a morphism such that $\varphi_\fe(k^\times) \subseteq \cX$. Hence $\fe_0 \in \cX\cap\EE(2,\fg_{(0)})$, so that $\EE(2,\fg_{(0)}) \subseteq \cX$.
Alternatively, there is $\beta \in k^\times$ such that $[c,x]=\beta y$. Then $\fh := kc\!+\!\fe$ is a $p$-trivial, non-abelian subalgebra, so that $\dim_k\fh=3$.
Hence $\EE(2,\fh)$ is connected such that $\fe, kc\!\oplus\!ky \in \EE(2,\fh)$. This shows $\EE(2,\fh)\subseteq \cX$ as well as $\cX\cap \EE(2,\fg_{(0)}) \ne
\emptyset$. Thus, $\EE(2,\fg_{(0)}) \subseteq \cX$.

In the remaining case, where $d_\cX=2$, we get $\cX=\EE(2,\fg_{(0)})$.

As a result, every connected component $\cX$ of $\EE(2,\fg)$ contains $\EE(2,\fg_{(0)})$, implying that there is only one component. Hence the variety $\EE(2,\fg)$ is connected.

We may thus assume that $kc$ is an ideal for every $c \in \Sw(\fg)\!\smallsetminus\!\{0\}$. Our assumption on $\Sw(\fg)$ then implies the existence of two linearly independent elements $c_1,c_2 \in \Sw(\fg)$. 
As each $kc_i$ is an ideal, we have $[c_1,c_2]=0$, so that $\fe:=kc_1\!\oplus\!kc_2$ is an elementary abelian ideal of $\fg$. Now Corollary \ref{Bc3} shows that the variety $\EE(2,\fg)$ is connected. \end{proof}

\bigskip

\subsection{Proof of Theorem A} We are now in a position to prove our main result concerning the connectedness of $\EE(2,\fg)$.
If $S \subseteq \fg$ is a subset, we denote by
\[ C_\fg(S) := \{ x \in \fg \ ; \ [x,s]=0 \ \ \ \forall \ s \in S\}\]
the \textit{centralizer} of $S$ in $\fg$.

\bigskip

\begin{Thm} \label{CRL1} Let $p\!\geq\!5$. If $(\fg,[p])$ is an admissible restricted Lie algebra, then the variety $\EE(2,\fg)$ is connected. \end{Thm}

\begin{proof} If $\Sw(\fg)$ is not an ideal of dimension $\le 1$, then Proposition \ref{SL2} implies the result. If $\Sw(\fg)=\{0\}$, a consecutive application of Theorem \ref{ACL1} and Lemma \ref{ACL2}
shows that $\fg' := [\fg,\fg]$ is classical semisimple along with $\EE(2,\fg)=\EE(2,\fg')$. Our result now follows from Corollary \ref{AL5}.

We may therefore assume that $\fa:=\Sw(\fg)=kc$ for some $c \in \Sw(\fg)\!\smallsetminus\{0\}$ is an ideal. Since every abelian ideal $\fn \unlhd \fg$ is contained in $\Sw(\fg)$, we conclude that $\fa$ is
the only non-zero abelian ideal of $\fg$.

We consider the restricted Lie algebra $\fg':= \fg/\fa$. Thanks to Theorem \ref{ACL0}, there exists an element $c' \in \fg\!\smallsetminus\!\fa$ such that $(\ad (c'\!+\!\fa))^2(\fg') \subseteq k(c'\!+\!\fa)$. It
follows that $(\ad c')^2(\fg) \subseteq kc'\!\oplus\!\fa$ while $(\ad c')^3(\fg) \subseteq \fa$.

If $C(\fg)\ne (0)$, then $C(\fg)=kc=\fa$, so that $(\ad c')^3=0$. Since $\fg$ is admissible, the subalgebra $\fe:=kc'\!\oplus\!\fa$ is elementary abelian. Moreover, $\fe \subseteq \cN_3(\fg)$, so that
$\fg$ being admissible in conjunction with Corollary \ref{Bc2} implies that $\EE(2,\fg)$ is connected.

Alternatively, $\fg$ is centerless, and the centralizer $C_\fg(\fa)$ of $\fa$ in $\fg$ is a $p$-ideal of codimension $1$. Hence there is $x_0 \in \fg$ such that $\fg = kx_0\!\oplus\! C_\fg(\fa)$ and $[x_0,c]=c$.

Let $c'=c'_s\!+\!c'_n$ be the Jordan-Chevalley-Seligman decomposition of $c'$. Since $c'_s$ is a $p$-polynomial in $c'$ without terms of degree $1$ \cite[(II.3.5)]{SF},
we see that $[c'_s,\fg]\subseteq \fa$. Hence the $p$-subalgebra $(kc'_s)_p$ is a torus such that $[(kc'_s)_p,\fg] \subseteq \fa$. Let $t \in (kc'_s)_p$ be a toral element. Then $\fb_t:= kt\!+\!\fa$ is a
$2$-dimensional $p$-ideal of $\fg$.

If $\fb_t$ is not abelian, then $\fb_t$ is centerless with all derivations of $\fb_t$ being inner, so that
\[\fg=\fb_t\!\oplus\!C_\fg(\fb_t)\]
is a direct sum of $p$-ideals. Our present assumption $\Sw(\fg)=\fa$ yields $\Sw(C_\fg(\fb_t)) = \{0\}$. Hence the arguments above show that
$\EE(2,C_\fg(\fb_t))$ is a connected subspace of $\EE(2,\fg)$.

Since $V(\fg) = \{\alpha c\!+\! x \ ; \ x \in V(C_\fg(\fb_t)), \alpha \in k\}$, there is a morphism
\[ \varphi : \PP(V(C_\fg(\fb_t))) \lra \EE(2,\fg) \ \ ; \ \ kx \mapsto kx\!\oplus\!kc\]
of projective varieties. By Carlson's Theorem \cite{Ca}, $\im\varphi$ is connected, and we let $\cX$ be the connected component of $\EE(2,\fg)$ containing
$\im\varphi$.

If there is $\fe_0 \in \EE(2,C_\fg(\fb_t))$, then $\EE(2,\fe_0\!\oplus\!kc) \cong \Gr_2(\fe_0\!\oplus\!kc)$ is connected, while $\EE(2,\fe_0\!\oplus\!kc)\cap \cX \ne
\emptyset$, whence $\fe_0 \in \cX$. Consequently, we always have $\EE(2,C_\fg(\fb_t)) \subseteq \cX$.

Let $\fe \in \EE(2,\fg)$. Then there are $\alpha, \beta \in k$ and $x,y \in V(C_\fg(\fb_t))$ such that $\fe = k(x\!+\!\alpha c)\!\oplus\!k(y\!+\!\beta c)$. If $x$ and $y$ are
linearly independent, then $\fe \in \EE(2,kx\!\oplus\!ky\!\oplus\!kc)$. Since this space is connected and intersects $\cX$, we see that $\fe \in \cX$. Alternatively, $\fe
\in \im\varphi \subseteq \cX$. This implies that $\cX=\EE(2,\fg)$ is connected.

Hence we may assume that $\fb_t$ is abelian for every toral element $t \in (kc'_s)_p$. Thus, $\fb_t=\fa$, so that $(\ad t)^2=0$. This yields $t=0$.
Since the torus $k(c'_s)_p$ is generated by toral elements \cite[(II.3.6)]{SF}, it follows that $(kc'_s)_p=(0)$, whence $c'_s=0$. We conclude that $(\ad c')$ is nilpotent,
so that $(\ad c')^4=0$.

Writing $c' = d'\!+\!\alpha x_0$, with $\alpha \in k$ and $d' \in C_\fg(\fa)$, we obtain,
\[ 0 = (\ad c')(c) = \alpha (\ad x_0)(c)= \alpha c,\]
so that $\alpha = 0$ and $c' \in C_\fg(\fa)$.

Recall that $(\ad c')^2(\fg) \subseteq kc'\!\oplus\!\fa=:\fe$. As $c'\in C_\fg(\fa)$, we obtain $(\ad c')^3=0$. Thus, $\fe \in \EE(2,\fg)$ is such that $\fe = \langle \fe\cap \cN_{\frac{p+1}{2}}(\fg)\rangle$
and Corollary \ref{Bc2} implies that $\EE(2,\fg)$ is connected. \end{proof}

\bigskip
\noindent
We record the proof of Theorem A, which generalizes earlier work by the first author \cite[(3.4.10)]{Ch}.

\bigskip

\begin{Cor} \label{CRL2} Let $p\!\ge\!5$. If $\fg$ is a restricted Lie algebra, then $\EE(2,\fg/C(\fg))$ is connected. \end{Cor}

\begin{proof} This follows by applying Lemma \ref{SL1} and Theorem \ref{CRL1} consecutively. \end{proof}

\bigskip

\begin{Example} The foregoing result does in general not hold for $p\!=\!2$. We consider the semidirect product
\[\fg:= \fsl(2)\!\ltimes\!L(1),\]
of $\fsl(2)$ with its standard module $L(1)$, the so-called {\it Schr\"odinger algebra}. Note that $\fg$ is centerless with unique $p$-map given by
\[ (x,v)^{[2]} := (x^{[2]},x.v) \ \ \ \ \ \ \forall \ x \in \fsl(2), v \in L(1).\]
In particular, $L(1)$ is isomorphic to an elementary abelian ideal of $\fg$.

Since $\EE(2,\fsl(2))\!=\!\emptyset$, we have $C_{\fsl(2)}(x)\cap V(\fsl(2)) = kx$ for every $x \in V(\fsl(2))\!\smallsetminus\!\{0\}$. We claim that 
\[ \EE(2,\fg) = \{L(1)\} \cup \{ \fe \in \EE(2,\fg) \ ; \ \dim_k \fe \cap \fsl(2) \ge 1\}.\]

Let $\fe \ne L(1)$ be an element of $\EE(2,\fg)$. Then $\fe = k(x,v)\!\oplus\!k(y,w)$, where $x,y \in V(\fsl(2))$, $x\!\ne\!0$ and $v,w \in L(1)$. If $\{v,w\}$ is linearly dependent, then we may assume without loss
of generality that $v\!=\!\alpha w$ for some $\alpha \in k$. Consequently, $0\!\ne\! (x,v)\!-\!\alpha(y,w)\! =\! (x\!-\!\alpha y, 0) \in \fe\cap\fsl(2)$.

We may thus assume that $L(1)\!=\!kv\!\oplus\!kw$. Since
\[ 0 =[(x,v),(y,w)] = ([x,y],x.w\!-\!y.v)  \ \ \text{and} \ \ 0 = (a,u)^{[2]} = (a^{[2]},a.u),  \ \ \ (a,u) \in \{(x,v),(y,w)\},\]
the observation above provides $\alpha \in k$ such that $y\!=\!\alpha x$ and we obtain $0\!=\!x.(w\!-\!\alpha v)$ and $0\!=\!x.v$. Hence $x.w\!=\!0$, so that $x.L(1)\!=\!(0)$, a contradiction. This establishes our claim.

It now follows from \cite[(7.3)]{Fa04} that $\EE(2,\fg)$ is the disjoint union of two non-empty closed subsets. \end{Example}

\bigskip

\subsection{Irreducible components and dimension}\label{S:IrrD} Elaborating on \cite[(1.7)]{CFP13}, we address the case where $\fg=\Lie(G)$ is the Lie algebra of a reductive group $G$. To that end, we
collect a few results concerning equidimensional varieties and principal fiber bundles. For an algebraic variety $X$, we let $\Irr(X)$ be the (finite) set of its irreducible components.

\bigskip

\begin{Lem} \label{IrrD1} Let $X$ be equidimensional, $\cO \subseteq X$ be a nonempty open subset. Then the following statements hold:
\begin{enumerate}
\item $\cO$ is equidimensional, $\dim \cO = \dim X$, and $\Irr(\cO) = \{\cO\cap C \ ; \ C \in \Irr(X) \ \text{and} \ \cO\cap C \ne \emptyset\}$.
\item If $\cO\cap C \ne \emptyset$ for every $C \in \Irr(X)$, then $\overline{\cO} = X$ and the map
\[ \Irr(X) \lra \Irr(\cO) \ \ ; \ \ C \mapsto C\cap\cO\]
is a bijection. \end{enumerate} \end{Lem}

\begin{proof} (1) Suppose that $\{C_1,\ldots, C_n\}$ is the set of those irreducible components of $X$ that meet $\cO$. Then $\cO = \bigcup_{i=1}^n\cO\cap C_i$, and each subset $\cO\cap C_i$
is open and dense in $C_i$. Hence $\cO\cap C_i$ is irreducible and $\dim \cO\cap C_i = \dim C_i = \dim X$. Since $\cO\cap C_i$ is closed in $\cO$,
this readily implies that $\cO$ is equidimensional and $\Irr(\cO)=\{C\cap\cO \ ; \ C \in \Irr(X) \ \text{and} \ C\cap\cO \neq \emptyset\}$.

(2) In view of (1), the given map is surjective. Since $ \overline{\cO\cap C} =C$ for all $C \in \Irr(X)$, it is also injective. Moreover, we have
\[ \overline{\cO} = \bigcup_{C \in \Irr(X)} \overline{\cO\cap C} = \bigcup_{C \in \Irr(X)}C = X,\]
as asserted. \end{proof}

\bigskip

\begin{Lem} \label{IrrD2} Let $G$ be a connected algebraic group acting simply on an equidimensional variety $X$. If $\varphi : X \lra Y$ is a dominant morphism such that
$\varphi^{-1}(\varphi(x))=G\dact x$ for all $x \in X$, then the following statements hold:
\begin{enumerate}
\item $Y$ is equidimensional of dimension $\dim Y = \dim X\!-\!\dim G$.
\item The map
\[ \Phi : \Irr(X) \lra \Irr(Y) \ \ ; \ \ C \mapsto \overline{\varphi(C)}\]
is bijective. \end{enumerate} \end{Lem}

\begin{proof} (1) Let $C \subseteq X$ be an irreducible component. Since $G$ is connected, the variety $C$ is $G$-stable, so that
\[ \varphi|_C : C \lra \overline{\varphi(C)}\]
is a dominant morphism such that $\varphi^{-1}(\varphi(c))=G\dact c \subseteq C$ for every $c \in C$. The fiber dimension theorem thus yields
\[ \dim \overline{\varphi(C)} = \dim C\!-\!\dim G = \dim X\!-\!\dim G.\]
As $Y=\bigcup_{C \in \Irr(X)}\overline{\varphi(C)}$, the variety $Y$ is equidimensional of dimension $\dim X\!-\!\dim G$ and $\Irr(Y) =\{\overline{\varphi(C)} \ ; \ C \in \Irr(X)\}$.

(2) By the above, the map $\Phi : \Irr(X) \lra \Irr(Y)$ is well-defined and surjective.

Suppose that $C_1,C_2 \in \Irr(X)$ are such that $\Phi(C_1)=\Phi(C_2)$. Then we have $\overline{\varphi(C_1)} = \overline{\varphi(C_2)}=:Z$ and Chevalley's Theorem provides
subsets $U_i \subseteq \varphi(C_i)$ that are open and dense in $Z$. Hence $U:=U_1\cap U_2$ enjoys the same properties. Given $u \in U$, there are $c_i \in C_i$ such that $u=\varphi(c_i)$.
Hence there is $g \in G$ such that $c_2=g.c_1 \in C_1$. It follows that $U \subseteq \varphi(C_1\cap C_2)$, so that
\[ \overline{\varphi(C_1)} = \overline{U} = \overline{\varphi(C_1\cap C_2)}.\]
Note that $C_1\cap C_2$ is $G$-stable, so that this also holds for each irreducible component $D \subseteq C_1\cap C_2$. It follows that
\[ \dim \overline{\varphi(D)} = \dim D\!-\!\dim G,\]
and hence
\begin{eqnarray*}
\dim C_1 & = & \dim Y\!+\!\dim G = \dim Z\!+\!\dim G = \dim \overline{\varphi(C_1\cap C_2)}\!+\!\dim G \\
                & = & \max\{\dim \overline{\varphi(D)} \ ; \ D \in \Irr(C_1\cap C_2)\}\!+\!\dim G = \max\{\dim D \ ; \ D \in \Irr(C_1\cap C_2)\}\\
                & = &\dim(C_1\cap C_2).
\end{eqnarray*}
Since $C_1$ is irreducible, we conclude that $C_1=C_1\cap C_2$, so that $C_1$ being a component implies $C_1=C_2$.

As a result, the map $\Phi$ is also injective. \end{proof}

\bigskip
\noindent
The foregoing observations will be applied in the following context, cf.\ \cite[\S1]{CFP13}. For a restricted Lie algebra $(\fg,[p])$ we denote by
\[ \cC_2(V(\fg)) := \{(x,y) \in V(\fg)\!\times\!V(\fg) \ ; \ [x,y]=0\}\]
the {\it commuting variety} of its nullcone $V(\fg)$. Note that the general linear group $\GL_2(k)$ acts on $\cC_2(V(\fg))$ via
\[ \left(\begin{smallmatrix}\alpha & \beta \\ \gamma & \delta \end{smallmatrix}\right)\dact (x,y) := (\alpha x\!+\!\beta y, \gamma x\!+\!\delta y).\]
Let
\[ \cO_2(V(\fg)) := \{(x,y) \in \cC_2(V(\fg)) \ ; \ \dim_kkx\!+\!ky=2\}.\]
Then $\cO_2(V(\fg))$ is a $\GL_2(k)$-stable, open subset of $\cC_2(V(\fg))$ on which $\GL_2(k)$ acts simply. The canonical morphism
\[ \varphi : \cO_2(V(\fg)) \lra \EE(2,\fg) \ \ ; \ \ (x,y) \lra kx\!\oplus\!ky\]
is surjective and such that $\varphi^{-1}(\varphi(x,y)) = \GL_2(k)\dact (x,y)$ for all $(x,y) \in \cO_2(V(\fg))$.

Following \cite[(II.1.17)]{Ja03}, we refer to a homomorphism $\pi : \hat{G} \lra G$ of connected algebraic groups as a {\it covering}, provided
\begin{enumerate}
\item[(i)] $\pi$ is surjective, and
\item[(ii)] $\ker\pi$ is a finite subgroup scheme of the scheme-theoretic center $\cZ(\hat{G})$ of $\hat{G}$. \end{enumerate}
We denote by $\hat{\fg}$ and $\fg$ the Lie algebras of $\hat{G}$ and $G$, respectively. Given $\hat{g} \in \hat{G}$, we have
\[ \pi\circ \kappa_{\hat{g}} = \kappa_{\pi(\hat{g})}\circ \pi,\]
where $\kappa_{\hat{g}}$ and $\kappa_{\pi(\hat{g})}$ denote the conjugation by $\hat{g}$ and $\pi(\hat{g})$, respectively. Differentiation thus implies
\[ \msd(\pi)\circ \Ad(\hat{g})= \Ad(\pi(\hat{g}))\circ \msd(\pi),\]
so that $\msd(\pi)$ is equivariant in the sense that
\[ (\ast) \ \ \ \ \ \ \pi(\hat{g})\dact\msd(\pi)(\hat{x}) = \msd(\pi)(\hat{g}\dact \hat{x}) \ \ \ \ \ \ \forall \ \hat{g} \in \hat{G}, \hat{x} \in \hat{\fg}.\]
The Coxeter number of a reductive group $G$ will be denoted $h(G)$, cf.\ \cite[(II.6.2)]{Ja03}.

Let $G$ be connected, semisimple with root system $\Phi$ and Cartan matrix $C_\Phi$. We require the following facts:
\begin{enumerate}
\item[(a)] If $\cZ(G)$ is the scheme-theoretic center of $G$, then $\det(C_\Phi)\!=\!\ord(\cZ(G))\ord(\pi_1(G))$.
\item[(b)] If $\pi : \hat{G} \lra G$ is a simply connected covering such that $p\!\nmid\!\ord(\pi_1(G))$, then the differential $\msd(\pi): \Lie(\hat{G}) \lra \Lie(G)$ is an isomorphism of restricted Lie algebras. \end{enumerate}
Fact (a) follows from the definition of $\pi_1(G)$ and \cite[(II.1.6)]{Ja03}. Fact (b) is a consequence of \cite[(2.23)]{Mil17} and the observation that the scheme-theoretic kernel of $\pi$ is \'{e}tale whenever 
$p\!\nmid\!\ord(\pi_1(G))$.

\bigskip

\begin{Lem} \label{IrrD3} Let $\pi : \hat{G} \lra G$ be a simply connected covering of the connected semisimple algebraic group $G$. If $p\!\nmid \ord(\pi_1(G))$, then the differential $\msd(\pi) : \hat{\fg} \lra \fg$ induces 
isomorphisms
\begin{enumerate}
\item $\msd(\pi) : V(\hat{\fg}) \lra V(\fg)$, and
\item $\Psi_\pi : \EE(2,\hat{\fg}) \lra \EE(2,\fg) \ \ ; \ \ \hat{\fe} \mapsto \msd(\pi)(\hat{\fe})$ \end{enumerate}
of varieties. \end{Lem}

\begin{proof} This follows directly from (b) above. \end{proof}

\bigskip
\noindent
Let $\fg =\Lie(G)$ be the Lie algebra of an algebraic group $G$. We say that $x \in V(\fg)$ is {\it distinguished}, provided every torus $T$ of the centralizer $C_G(x)$ of
$x$ in $G$ is contained in the center $Z(G)$ of $G$. In view of \cite[(21.4)]{Hu81}, the connected component $C_G(x)^\circ$ of a distinguished element $x$ is nilpotent.
Consequently, $C_G(x)^\circ$ is contained in a Borel subgroup $B$ of $G$. Moreover, the Lie-Kolchin Theorem implies that $C_G(x)^\circ$ is unipotent in case $Z(G)$ is finite.

We let $\rk(G)$ denote the {\it rank} of $G$, that is, the dimension of a maximal torus of $G$.

\bigskip

\begin{Lem} \label{IrrD4} Let $G$ be a connected, simply connected, semisimple algebraic group of rank $\rk(G)\!\ge\!2$. If $p\!\ge\!h(G)$, then we have
\[ C\cap \cO_2(V(\fg)) \ne \emptyset\]
for every $C \in \Irr(\cC_2(V(\fg)))$. \end{Lem}

\begin{proof} Since $p\!\ge\!h(G)$, the prime $p$ is good for $G$, (cf.\ \cite[p.80]{Hu90}, \cite[(2.6)]{Ja04}) and the nullcone $V(\fg)$ coincides with the nilpotent cone $\cN(\fg)$. Moreover, $p$ is a non-torsion prime for $G$.

Let $C \in \Irr(\cC_2(V(\fg)))$ be an irreducible component. Since $G$ is simply connected, Premet's Theorem \cite[Theorem]{Pr03} provides a distinguished element $x \in V(\fg)$ such that
$\{x\}\!\times\!(C_\fg(x)\cap V(\fg)) \subseteq C$. Let $\fb\! =\! \Lie(B)$ be a Borel subalgebra. As $\rk(G)\!\ge\!2$, it follows that $V(\fb) = \bigcup_{\fe \in \EE(2,\fb)}\fe$ and a two-fold application
of Theorem \ref{AL2} and its succeeding remark yields $V(\fg)\!=\!G\dact V(\fb)\!=\!G\dact (\bigcup_{\fe \in \EE(2,\fb)}\fe)\!=\!\bigcup_{\fe \in \EE(2,\fg)}\fe$, so that $V(C_\fg(x))\ne kx$. Hence there 
is $(x,y) \in \cO_2(V(\fg))\cap C$. \end{proof}

\bigskip

\begin{Remark} Let $G$ be connected, semisimple with root system $\Phi$ and such that $p\!>\!h(G)$. Let $G_1,\ldots, G_n$ be the almost simple components of $G$, with corresponding root systems
$\Phi_1,\ldots,\Phi_n$. Then $h(G_i)\!\ge\!\det(C_{\Phi_i})$ for all $i \in \{1,\ldots, n\}$, so that $p$ is not a divisor of $\det(C_\phi)\!=\!\prod_{i=1}^n\det(C_{\Phi_i})$. In view of (a), we obtain $p\!\nmid\!\ord(\pi_1(G))$.
\end{Remark}

\bigskip
\noindent
Let $G$ be a connected, reductive algebraic group. Then we put $\pi_1(G)\!:=\!\pi_1((G,G))$. Let $x \in V(\fg)$. Then $\EE(2,\fg,x)\!:=\!\{\fe \in \EE(2,\fg) \ ; \ x \in \fe\}$ is a closed subset of $\EE(2,\fg)$. 

\bigskip

\begin{Thm} \label{IrrD5} Let $G$ be a reductive connected algebraic group such that $p\!\ge\!h(G)$ and $p\!\nmid\!\ord(\pi_1(G))$. Then the following statements hold:
\begin{enumerate}
\item The variety $\EE(2,\fg)$ is equidimensional of dimension $\dim \EE(2,\fg)\!=\!\dim (G,G)\!-\!4$.
\item Let $x_1, \ldots, x_n$ be representatives for the distinguished orbits of $V(\fg)$. Then
\[ \Irr(\EE(2,\fg)) = \{\overline{G\dact \EE(2,\fg,x_i)} \ ; \ 1\!\le\!i\!\le\!n\}.\]
\end{enumerate} \end{Thm}

\begin{proof} We let $Z(G)^\circ$ be the identity component of the center $Z(G)$ of $G$.  In view of \cite[(19.5),(27.5)]{Hu81}, the group $G=Z(G)^\circ (G,G)$ is a product, whose factors
intersect in a finite set. Moreover, $Z(G)^\circ$ is a torus, while $(G,G)$ is semisimple and connected. Consequently, the restriction of the canonical projection $\pi : G \lra G/(G,G)$ yields a surjection
$Z(G)^\circ \twoheadrightarrow G/(G,G)$, implying that the factor group $G/(G,G)$ is a torus, cf.\ \cite[(21.3C)]{Hu81}.

Setting $G':= (G,G)$ and $\fg':=\Lie(G')$, we obtain an exact sequence
\[ (0) \lra \fg' \lra \fg \stackrel{\msd(\pi)}{\lra} \Lie(G/G') \lra (0)\]
of restricted Lie algebras, whose second term from the right is a torus. This readily implies $V(\fg)=V(\fg')$ as well as $\EE(2,\fg)=\EE(2,\fg')$, while $h(G)=h(G')\!\le\!p$.

Let $x \in V(\fg)$ be distinguished for $G$. If $T \subseteq C_{G'}(x)$ is a maximal torus, then $T\subseteq Z(G)^\circ\cap G'$ is finite, so that $T=\{1\}=Z(G')^\circ$. Hence $x \in V(\fg')$ is
distinguished for $G'$. Conversely, if $x \in V(\fg')$ is distinguished, then $C_{G'}(x)^\circ$ is unipotent, so that every torus $T \subseteq C_G(x)=Z(G)^\circ C_{G'}(x)$ is contained in $Z(G)^\circ$. Since
$G\dact\EE(2,\fg,x)\!=\!G'\dact\EE(2,\fg',x)$ for every $x \in V(\fg)\!=\!V(\fg')$, it suffices to verify (1) and (2) for $G$ being connected and semisimple.

We consider a simply connected covering $\pi : \hat{G} \lra G$, so that $\hat{G}$ is in particular connected and semisimple. Setting $\hat{\fg}\!:=\!\Lie(\hat{G})$ and $\fg\!:=\! \Lie(G)$, we apply Lemma \ref{IrrD3} in
conjunction with ($\ast$) to see that $\msd(\pi) : \hat{\fg} \lra \fg$ induces equivariant isomorphisms
\[ \msd(\pi) : V(\hat{\fg}) \lra V(\fg) \ \ \text{and} \ \ \Psi_\pi : \EE(2,\hat{\fg}) \lra \EE(2,\fg).\]

(1) Since $p\!\ge\!h(G)\!=\!h(\hat{G})$, the nullcone $V(\hat{\fg})$ coincides with the nilpotent cone $\cN(\hat{\fg})$. Hence \cite[Theorem]{Pr03} ensures that the variety $\cC_2(V(\hat{\fg}))$ is equidimensional and of 
dimension $\dim \hat{G}\!=\!\dim G$. In view of (\ref{IrrD4}), (\ref{IrrD1}), and (\ref{IrrD2}), it follows that $\EE(2,\hat{\fg})$ is equidimensional of dimension $\dim \EE(2,\hat{\fg})=\dim \hat{G}\!-\!4$. As a result,
$\EE(2,\fg)$ has the same properties.

(2) Since $\msd(\pi) : V(\hat{\fg}) \lra V(\fg)$ is equivariant, it follows that $G\dact\msd(\pi)(\hat{x})\!=\!\msd(\pi)(\hat{G}\dact \hat{x})$ and $\pi(C_{\hat{G}}(\hat{x}))\!=\!C_G(\msd(\pi)(\hat{x}))$ for every $\hat{x} \in V(\hat{\fg})$. 
Using \cite[(21.3C)]{Hu81}, we see that $\hat{x} \in V(\hat{\fg})$ is distinguished if and only if $\msd(\pi)(\hat{x})\in V(\fg)$ exhibits this property. 

By the above, $\msd(\pi)$ sends a set of distinguished orbit representatives for $V(\hat{\fg})$ bijectively to  a set of distinguished orbit representatives for $V(\fg)$, while $\Psi_\pi(\hat{G}\dact \EE(2,\hat{\fg},\hat{x}))
\!=\!G\dact\EE(2,\fg,\msd(\pi)(\hat{x}))$ for every $\hat{x} \in V(\hat{\fg})$. It follows that we may assume that $G\!=\!\hat{G}$ is simply connected.

Let $x_1,\ldots, x_n$ be representatives for the distinguished orbits of $V(\fg)$, and put
\[ \fC(x_i) := \overline{G\dact(\{x_i\}\!\times\!(C_\fg(x_i)\cap V(\fg)))} \ \ \ \ 1\!\le\!i\!\le\!n.\]
Premet's aforementioned theorem asserts that $\Irr(\cC_2(V(\fg)))=\{\fC(x_i) \ ; \ 1\!\le\! i\!\le\!n\}$.  If $\EE(2,\fg)\ne \emptyset$, then $\rk(G)\!\ge\!2$ and a consecutive application of Lemma \ref{IrrD4} and Lemma \ref{IrrD1} 
shows that the $n$ irreducible components of $\cO_2(V(\fg))$ are of the form $\fC(x_i)\cap\cO_2(V(\fg))$ for $i \in \{1,\ldots, n\}$.

Recall that the morphism
\[ \varphi : \cO_2(V(\fg)) \lra \EE(2,\fg) \ \ ; \ \ (x,y) \mapsto kx\!\oplus\!ky\]
is surjective with fibers given by the orbits of the canonical $\GL_2(k)$-action. In view of Lemma \ref{IrrD2}, we have
\[ \Irr(\EE(2,\fg)) = \{ \overline{\varphi(\fC(x_i)\cap\cO_2(V(\fg)))} \ ; \ 1\!\le\!i\!\le\!n\}.\]
Since $\varphi$ is $G$-equivariant with respect to the canonical action of $G$ on $\cO_2(V(\fg))$, while $\EE(2,\fg,x_i)=\varphi((\{x_i\}\!\times\!C_\fg(x_i))\cap \cO_2(V(\fg))$, it follows that
$\varphi(G\dact(\{x_i\}\!\times\!C_\fg(x_i))\cap \cO_2(V(\fg))) = G\dact\EE(2,\fg,x_i)$, whence
\[ \fC(x_i) \cap \cO_2(V(\fg)) \subseteq \overline{G\dact(\{x_i\}\!\times\!C_\fg(x_i))\cap\cO_2(V(\fg))} \subseteq \varphi^{-1}(\overline{G\dact\EE(2,\fg,x_i)}).\]
As a result,
\[ \overline{G\dact\EE(2,\fg,x_i)} = \overline{\varphi(G\dact(\{x_i\}\!\times\!C_\fg(x_i))\cap \cO_2(V(\fg))} \subseteq \overline{\varphi(\fC(x_i)\cap \cO_2(V(\fg)))} \subseteq \overline{G\dact\EE(2,\fg,x_i)},\]
so that the irreducible components of $\EE(2,\fg)$ have the asserted form. \end{proof}

\bigskip
\noindent
The foregoing result shows in particular, that the number of irreducible components of $\EE(2,\fg)$ coincides with the number of distinguished nilpotent orbits. The Bala-Carter theory relates these orbits
to the distinguished parabolic subgroups of $G$. Letting ${\rm pa}(n) \in \{0,1\}$ denote the parity of $n \in \NN$, we record below for $p\!\ge\!h(G)$ those cases, where the variety of $2$-dimensional 
elementary abelian subalgebras of a classical Lie algebra is irreducible.

\bigskip

\begin{Cor} \label{IrrD6} Let $\fg\!=\!\Lie(G)$ be the Lie algebra of an almost simple algebraic group.
\begin{enumerate}
\item If $\fg\!=\!\fsl(n)$ and $p\!\ge\!n\!\ge\!3$, then $\EE(2,\fg)$ is irreducible.
\item If $\fg\!=\!\fso(n) \ \ (n\!\ge\!5)$ and $p\!\ge\!n\!-\!2\!+\!{\rm pa}(n)$, then $\EE(2,\fg)$ is irreducible if and only if $n \in \{5,6,7\}$.
\item If $\fg\!=\!\fsp(2n) \ \ (n\!\ge\!2)$ and $p\!\ge\! 2n$, then $\EE(2,\fg)$ is irreducible if and only if $n\!=\!2$.
\item If $G$ of type $E_6,E_7,E_8,F_4,G_2$ and $p\!\ge \!h(G)$, then $\EE(2,\fg)$ is not irreducible. \end{enumerate} \end{Cor}

\begin{proof} Let $\Phi$ be the root system of $G$. Since $G$ is almost simple, we have $h(G)\!\ge\!\det(C_\Phi)\!\ge\!\ord(\pi_1(G))$, with equality throughout only if $\Phi$ is of Dynkin type $A$.  

(1) In this case, we have $h(\SL(n))\!=\!n$, while $\ord(\pi_1(\SL(n))\!=\!1$. Our assertion thus follows directly from Theorem \ref{IrrD5} and \cite[(4.1)]{Ja04}.

(2) Let $\fg\!=\!\fso(n)$ for $n\!\ge\!3$, so that $h(G)\!=\!n\!-\!2\!+\!{\rm pa}(n)$. Thanks to \cite[(4.2)]{Ja04}, the Lie algebra $\fso(n)$ possesses only $1$ distinguished orbit for $n\!\le\!7$, while the pairs of
partitions $(7\!+\!2m,1),(5\!+\!2m,3)$ and $(9\!+\!2m),(5\!+\!2m,3,1)$ give rise to two distinguished orbits, whenever $n\!\ge\!8$ is even and $n\!\ge\!9$ is odd, respectively. The assertion now follows from Theorem 
\ref{IrrD5}.

(3) Let $\fg\!=\!\fsp(2n)$ for $n\!\ge\!2$, so that $h(G)\!=\!2n$. In view of \cite[(4.2)]{Ja04}, $\fsp(2n)$ possesses $1$ distinguished orbit for $n\!=\!2$, while the partitions $(2(m\!+\!3)),(2(m\!+\!2),2)$ define
two distinguished orbits for $m\!\ge\!0$. The assertion now follows from Theorem \ref{IrrD5}.

(4) This is a consequence of Theorem \ref{IrrD5} in conjunction with \cite[p.175-177]{Car}.\end{proof}

\bigskip

\begin{Remarks} (1) We refer the reader to \cite[(5.1)]{Wa} for a summary of those cases, where $\EE(r,\fgl(n))$ is known to be irreducible.

(2) If $(\fg,[p])$ is a restricted Lie algebra such that $\EE(2,\fg)$ is irreducible, then the conical closed subset $V_{\EE(2,\fg)}\!:=\!\bigcup_{\fe \in \EE(2,\fg)}\fe \subseteq V(\fg)$ is irreducible, cf.\ \cite[(1.5)]{Fa04} and Lemma 
\ref{LCF2} below. The preceding result provides examples of Lie algebras $\fb$, where $V_{\EE(2,\fb)}$ is a linear subspace of $\fb$ while $\EE(2,\fb)$ is not irreducible: Let $\fb$ be a Borel subalgebra of $\fsp(2n)$, where 
$p\!\ge\!2n\!\ge\!6$. If $\EE(2,\fb)$ is irreducible, then Theorem \ref{AL2} implies that $\EE(2,\fsp(2n))\!=\!\Sp(2n)\dact\EE(2,\fb)$ is also irreducible, a contradiction. Since $p\!\ge\!h(\Sp(2n))$, it follows that $V_{\EE(2,\fb)}\!=\!
V(\fb)$ coincides with the unipotent radical of $\fb$. 

(3) Theorem \ref{IrrD5} fails for $G\!:=\!\PGL(p)$. Since there is a simply connected covering $\pi : \SL(p) \lra \PGL(p)$, we have $h(G)\!=\!p\!=\!\ord(\pi_1(G))$. Then the restriction $\msd(\pi) : V(\fsl(p)) \lra V(\fpgl(p))$ of the
differential $\msd(\pi) : \fsl(p) \lra \fpgl(p)$ is bijective (cf.\ \cite[(2.7)]{Ja04}), and the arguments of the proof of Theorem \ref{IrrD5}(2) show that $\msd(\pi)$ maps distinguished orbits of $V(\fsl(p))$ to distinguished orbits of 
$V(\fpgl(p))$. The validity of Theorem \ref{IrrD5} in conjunction with \cite[(4.1)]{Ja04} would imply that $\EE(2,\fpgl(p))$ is irreducible of dimension $p^2\!-\!5$. 

On the other hand, \cite[(2.1)]{CF} implies that $\Psi_\pi : \EE(2,\fsl(p)) \lra \EE(2,\fpgl(p))$ is an injective morphism of projective varieties, so that Theorem \ref{IrrD5} ensures that $\im\Psi_\pi$ is a closed, irreducible subset
of dimension $p^2\!-\!5$. It follows that $\Psi_\pi$ is surjective. Let $B_0 \subseteq \SL(p)$ be the standard Borel subgroup, $\fb_0\!:=\!\Lie(B_0)$. In view of Theorem \ref{AL2}, we have 
\[\EE(2,\fpgl(p))\!=\!\Psi_\pi(\SL(p)\dact \EE(2,\fb_0)) \subseteq \PGL(p)\dact\EE(2,\msd(\pi)(\fb_0)) \subseteq \PGL(p)\dact \EE(2,\Lie(\pi(B_0))).\] 
This, however, contradicts the remarks following Theorem \ref{AL2}. \end{Remarks}

\bigskip

\section{Locally constant functions on $\Gr_2(\fg)^{V(\fg)}$}\label{S:LCF}
Given a restricted Lie algebra $(\fg,[p])$, we consider the subset
\[ \Gr_2(\fg)^{V(\fg)} := \{ \fv \in \Gr_2(\fg) \ ; \ \fv \subseteq V(\fg)\}\]
of the Grassmannian $\Gr_2(\fg)$ of $2$-planes in $\fg$. Since the map
\[ \Gr_2(\fg) \lra \NN_0 \ \ ; \ \ \fv \mapsto \dim\fv\cap V(\fg)\]
is upper semicontinuous (cf.\ \cite[(7.3)]{Fa04}), it follows that $\Gr_2(\fg)^{V(\fg)} = \{\fv \in  \Gr_2(\fg) \ ; \ \dim \fv\cap V(\fg) \ge 2\}$ is closed and thus in particular
a projective variety.

Let $\cX \subseteq \Gr_2(\fg)^{V(\fg)}$ be a closed subset, $V_\cX:= \bigcup_{\fv \in \cX}\fv$. In this section, we associate to every morphism $\varphi : \PP(V_\cX) \lra
\PP^m$ a locally constant function on $\cX$. In case $\EE(2,\fg) \subseteq \cX \subseteq \Gr_2(\fg)^{V(\fg)}$ is connected, this will enable us define degrees for
the so-called $\fg$-modules of constant $j$-rank.

We begin with some general preliminary observations. For a finite-dimensional vector space $V$ over $k$, we denote by $S(V)$ the symmetric algebra on $V$.
Thus, $S(V^\ast)$ is the $\ZZ$-graded algebra of polynomial functions on $V$. Being a polynomial ring in $\dim_kV$ variables, $S(V^\ast)$ is a unique factorization domain.
A subspace $W \subseteq V$ defines a homomorphism
\[ \res_W : S(V^\ast) \lra S(W^\ast) \ \ ; \ \ f \mapsto f|_W\]
of commutative graded $k$-algebras.

Let $\cX$ be a topological space. A map $\varphi : \cX \lra S$ with values in some set $S$ is called {\it locally constant}, provided there exists for every $x \in \cX$
an open neighborhood $U_x$ of $x$ such that $\varphi|_{U_x}$ is constant. Note that a map $\varphi : \cX \lra \NN_0$ is locally constant if and only if it is upper
semicontinuous and lower semicontinuous.

Let $\cX \subseteq \PP(V)$ be locally closed. A morphism $\varphi : \cX \lra \PP^m$ is said to be {\it homogeneous}, if there exist homogeneous elements $f_0,
\ldots, f_m \in S(V^\ast)_d$ such that
\[ \varphi(x) = (f_0(x):\cdots :f_m(x)) \ \ \ \ \ \ \forall \ x \in \cX.\]
The $(m\!+\!1)$-tuple of polynomials $(f_0,\ldots,f_m)$ is called a {\it defining system} for $\varphi$. Such a system is said to be {\it reduced}, provided $\gcd(f_0,\ldots,f_m)=1$.

If $\cX\subseteq \PP(V)$ is open, then any morphism $\varphi: \cX \lra \PP^m$ is homogeneous and the common degree $\deg(\varphi)$ of the non-zero $f_i$'s
does not depend on the choice of the reduced defining system $(f_0,\ldots, f_m) \in S(V^\ast)^{m+1}$, cf.\ \cite[\S1]{Fa17} for more details.

\bigskip

\begin{Lemma} \label{LCF1} Let $W \subseteq V$ be a subspace, and $\fU\subseteq \PP(V)$ be locally closed such that $\fU\cap \PP(W) \ne \emptyset$ is open in $\PP(W)$. If $\varphi : \fU \lra \PP^m$
is a homogeneous morphism with defining system $(f_0,\ldots,f_m)\in S(V^\ast)^{m+1}_d$, then
\[\deg(\varphi|_{\fU\cap\PP(W)}) = d\!-\!\deg(\gcd(\res_W(f_0),\ldots, \res_W(f_m))).\] \end{Lemma}

\begin{proof} Since $\fU\cap \PP(W)$ is open, the morphism $\varphi|_{\fU\cap \PP(W)}$ is homogeneous and its degree may be computed via a reduced defining system.

Note that $(\res_W(f_0), \ldots, \res_W(f_m)) \in S(W^\ast)^{m+1}_d$ is a defining system for the morphism $\varphi|_{\fU\cap \PP(W)}$. Setting $r:= \deg(\gcd(\res_W(f_0),\ldots, \res_W(f_m)))$, while
observing \cite[(1.1.1)]{Fa17}, the element $h:=  \gcd(\res_W(f_0),\ldots, \res_W(f_m))$ belongs to $S(W^\ast)_r$. Then there exist homogeneous elements $g_0,\ldots,$ $g_m \in S(W^\ast)_{d-r}$ such
that $\res_W(f_i)=hg_i$ for $0\!\le\!i\!\le\!m$.
Consequently, $\gcd(g_0,\ldots,g_m)=1$, and for $u \in \fU\cap\PP(W)$, we obtain
\[ \varphi(u)= (\res_W(f_0)(u):\cdots : \res_W(f_m)(u)) = (g_0(u):\cdots : g_m(u)).\]
This shows that $(g_0,\ldots, g_m)$ is a reduced defining system for $\varphi|_{\fU\cap\PP(W)}$, whence
\[ \deg(\varphi|_{\fU\cap\PP(W)})= \deg(g_i) = d\!-\!r = d\!-\!\deg(\gcd(\res_W(f_0),\ldots, \res_W(f_m))),\]
whenever $g_i\ne 0$. \end{proof}

\bigskip

\begin{Lemma} \label{LCF2} Let $(\fg,[p])$ be a restricted Lie algebra, $\cX \subseteq \Gr_2(\fg)^{V(\fg)}$ be a closed subset. Then the following
statements hold:
\begin{enumerate}
\item  $V_\cX := \bigcup_{\fv \in \cX}\fv$ is a conical, closed subset of $V(\fg)$.
\item If $U \subseteq V_\cX$ is a conical open subset of $V_\cX$, then
\[ \cX_U :=\{\fv \in \cX \ ; \ \fv\cap U \ne \emptyset\}\]
is an open subset of $\cX$. \end{enumerate} \end{Lemma}

\begin{proof} (1) It is well-known that
\[ \Sigma := \{(x,\fv) \in V(\fg)\!\times\!\cX \ ; \ x \in \fv\}\]
is a closed subset of $V(\fg)\!\times\!\cX$. Since the variety $\cX$ is projective, the projection map $\pi : V(\fg)\!\times\!\cX \lra V(\fg)$ is closed. As a result,
$V_\cX=\pi(\Sigma)$ is closed.

(2) Owing to (1), the set $A:=V_\cX\!\smallsetminus\!U$ is a closed, conical subset of $\fg$. In view of \cite[(7.3)]{Fa04}, the map $\cX \lra \NN_0 \ ; \ \fv \mapsto
\dim\fv\cap A$ is upper semicontinuous, so that the set
\[ \cC := \{ \fv \in \cX \ ; \ \dim \fv\cap A \ge 2\}\]
is closed. As each $\fv \in \cX$ is irreducible, we have $\cC = \{ \fv \in \cX \ ; \ \fv \subseteq A\}$, implying that
\[ \cX_U = \cX\!\smallsetminus\!\cC\]
is open. \end{proof}

\bigskip

\begin{Theorem} \label{LCF3} Let $(\fg,[p])$ be a restricted Lie algebra, $\cX \subseteq \Gr_2(\fg)^{V(\fg)}$ be a closed subset, $\varphi : \PP(V_\cX) \lra \PP^m$ be a morphism.
\begin{enumerate}
\item The map
\[ \msdeg_\varphi : \cX \lra \NN_0 \ \ ; \ \ \fv \mapsto \deg(\varphi|_{\PP(\fv)})\]
is locally constant.
\item If $\cX$ is connected, then $\msdeg_\varphi$ is constant. \end{enumerate} \end{Theorem}

\begin{proof} (1) Let $\fv_0 \in \cX$. General theory (cf.\ \cite[(1.65)]{GW}) provides an open subset $\fU$ of $\PP(V_\cX)$ and $(f_0,\ldots, f_m) \in
S(\fg^\ast)^{m+1}_d$ such that
\begin{enumerate}
\item[(a)] $\PP(\fv_0) \cap \fU \ne \emptyset$, and
\item[(b)] $(f_0,\ldots, f_m) \in S(\fg^\ast)^{m+1}_d$ is a defining system for the homogeneous morphism $\varphi|_{\fU}$. \end{enumerate}
In particular, $\fU$ is a locally closed subset of $\PP(\fg)$ such that $\PP(\fv)\cap\fU$ is open in $\PP(\fv)$ for every $\fv \in \cX$. We denote by
$U\!:=\!\{x \in V_\cX\!\smallsetminus\!\{0\} \ ; \ [x] \in \fU\}$ the cone of $\fU$. This is a conical, open subset of $V_\cX$. According to Lemma \ref{LCF2}, $\cX_U$ is an open
subset of $\cX$ that contains $\fv_0$.

Let $W\subseteq \fg$ be a subspace such that $\fv_0\!\oplus\!W = \fg$. As before, we consider the canonical map $\res_W : S(\fg^\ast) \lra S(W^\ast)$. An application of \cite[(7.3)]{Fa04} shows that
\[ \cU_W := \{ \fv \in \cX \ ; \ \fv\!\oplus\!W=\fg\}\]
is an open subset of $\cX$ that contains $\fv_0$.

Now let $\fv \in \cU_W\cap\cX_U$. Then $\fv\cap U$ is a non-empty, open subset of $\fv$. Hence it lies dense in $\fv$ and $\fv\cap U \ne \{0\}$. As a result, $\PP(\fv)\cap\fU \ne \emptyset$. The decomposition
$\fg\!=\!\fv\!\oplus\!W$ induces a decomposition $\fg^\ast\!=\!\fv^\ast\!\oplus\!W^\ast$. Let $\pi_W : S(\fg^\ast) \lra S(\fg^\ast)/I_W$ be the canonical projection, where $I_W\!:=\! S(\fg^\ast)S(W^\ast)^\dagger$ is 
generated by the augmentation ideal $S(W^\ast)^\dagger$ of $S(W^\ast) \subseteq S(\fg^\ast)$. Since $S(\fg^\ast)\!=\!S(\fv^\ast)\!\oplus\!I_W$, the restriction $(\pi_W)|_{S(\fv^\ast)} : S(\fv^\ast) \lra S(\fg^\ast)/I_W$ 
is an isomorphism (of degree $0$) such that $\pi_W \circ \res_{\fv}\!=\!\pi_W$. Consequently, 
\[ \gcd(\pi_W(f_0),\ldots,\pi_W(f_m))\!=\!\gcd(\pi_W(\res_{\fv}(f_0)),\ldots,\pi_W(\res_{\fv}(f_m)))\!=\!\pi_W(\gcd(\res_{\fv}(f_0),\ldots,\res_{\fv}(f_m))),\] 
Thus, setting $r\!:=\!\deg(\gcd(\pi_W(f_0),\ldots,\pi_W(f_m)))$,  Lemma \ref{LCF1} yields 
\[\deg(\varphi|_{\PP(\fv)\cap\fU})\! =\!d\!-\!\deg(\gcd(\res_{\fv}(f_0),\ldots,\res_{\fv}(f_m)))\!=\! d\!-\!r.\] 
Now \cite[(1.1.2)]{Fa17} implies $\msdeg_\varphi(\fv)\!=\!d\!-\!r$. As a result, the map $\msdeg_\varphi$ is locally constant.

(2) Let $n \in \NN$. According to (1), each fiber $\msdeg_\varphi^{-1}(n)$ is open in $\cX$. In view of
\[ \msdeg_\varphi^{-1}(n) = \cX\!\smallsetminus\!(\bigcup_{m\ne n}\msdeg_\varphi^{-1}(m)),\]
the fiber is also closed. As $\cX$ is connected, each non-empty fiber coincides with $\cX$, so that $\msdeg_\varphi$ is constant. \end{proof}

\bigskip

\section{Degree functions}
Let $(\fg,[p])$ be a restricted Lie algebra with restricted enveloping algebra $U_0(\fg)$. By definition,
\[ U_0(\fg):= U(\fg)/(\{x^p\!-\!x^{[p]} \ ; \ x \in \fg\})\]
is a finite-dimensional quotient of the ordinary enveloping algebra $U(\fg)$. The Lie algebra $\fg$ is a subalgebra of the commutator algebra $U_0(\fg)^-$ and the
$U_0(\fg)$-modules are precisely those $\fg$-modules $M$ for which $x^{[p]}$ acts via the $p$-th power of the action of $x \in \fg$.

To a $U_0(\fg)$-module $M$ of constant $j$-rank, we will associate a degree function
\[\msdeg^j_M : \Gr_2(\fg)^{V(\fg)} \lra \NN_0,\]
which will allow us to generalize results of \cite{Fa17} concerning modules of constant $j$-rank over Lie algebras with smooth nullcones.

For a $U_0(\fg)$-module $M$ and $u \in U_0(\fg)$, we denote by
\[ u_M : M \lra M \ \ ; \ \ m \mapsto u.m\]
the operator associated to $u$. Given $j \in \{1,\ldots,p\!-\!1\}$, the {\it generic $j$-rank} of $M$ is defined via $\rk^j(M)\!:=\!\max\{\rk(x_M^j) \ ; \ x \in V(\fg)\}$. Semicontinuity of ranks
shows that the set $U_{M,j}\!:=\! \{ [x] \in \PP(V(\fg)) \ ; \ \rk(x_M^j)\!=\!\rk^j(M)\}$ is open. We say that $M$ has {\it constant $j$-rank}, provided $U_{M,j}=\PP(V(\fg))$. These modules
were first investigated by Friedlander and Pevtsova in \cite{FP10}.

Given $d \in \NN_0$, we denote by
\[ \mspl_M : \Gr_d(M) \lra \PP(\bigwedge^d(M))\]
the {\it Pl\"ucker embedding}.  Let $j \in \{1,\ldots,p\!-\!1\}$. Following \cite{Fa17}, we associate to a $U_0(\fg)$-module $M$ the morphism
\[ \msim_M^j : U_{M,j} \lra \Gr_{\rk^j(M)}(M) \  \ ; \  \ [x] \mapsto \im x_M^j\]
of quasi-projective varieties. If $\fv \subseteq V(\fg)$ is a linear subspace of $\fg$, then $U_{M,j}\cap\PP(\fv)$ is an open subset of a projective space and the morphism
\[\mspl_M \circ \msim_M^j|_{\PP(\fv)} : U_{M,j}\cap\PP(\fv) \lra \PP(\bigwedge^{\rk^j(M)}(M))\]
is homogeneous whenever $U_{M,j}\cap\PP(\fv)\ne \emptyset$. It is thus is given by a reduced defining system, whose degree  $\deg(\mspl_M \circ \msim_M^j|_{\PP(\fv)})$ does not depend
on the choice of the system, cf.\ \cite[\S 1.1]{Fa17} for more details.

Suppose that $M$ is a $U_0(\fg)$-module of constant $j$-rank. Given a subspace $\fv \subseteq V(\fg)$ of dimension $\dim_k\fv \ge 2$, there results a morphism
\[ \msim_M^j|_{\PP(\fv)} : \PP(\fv) \lra \Gr_{\rk^j(M)}(M)\]
of projective varieties and we put
\[ \msdeg^j_M(\fv) := \deg(\mspl_M\circ\msim^j_M|_{\PP(\fv)}).\]
If $M$ is a $U_0(\fg)$-module and $\fh \subseteq \fg$ is a $p$-subalgebra, then $M|_\fh$ denotes the restriction of $M$ to $U_0(\fh)\subseteq U_0(\fg)$. Suppose that
$M$ has constant $j$-rank and let $\fe \subseteq \fg$ be an elementary abelian subalgebra of dimension $\ge\!2$. Then $\msdeg^j_M(\fe)$ is just the $j$-degree $\deg^j(M|_\fe)$
that was introduced in \cite{Fa17}.

\bigskip

\begin{Definition} Let $(\fg,[p])$ be a restricted Lie algebra such that $\msrke(\fg)\!\ge\! 2$, $M$ be a $U_0(\fg)$-module of constant $j$-rank. Then
\[ \msdeg^j_M : \EE(2,\fg) \lra \NN_0 \ \ ; \ \ \fe \mapsto \deg^j(M|_\fe)\]
is called the {\it $j$-degree function of $M$}. \end{Definition}

\bigskip
\noindent
The $j$-degree function of $M$ is expected to provide information concerning the nilpotent operators $x_M$, where $x \in V_{\EE(2,\fg)}$.

\bigskip

\begin{Remarks} (1) Let $M$ be a $U_0(\fg)$-module of constant $j$-rank. If $\fe \in \EE(r,\fg)$ for some $r\!\ge\!2$, then $\deg^j(M|_\fe)=\deg^j(M|_\ff)$ for every two-dimensional
subalgebra $\ff \subseteq \fe$, cf.\ \cite[(4.1.2)]{Fa17}. Accordingly, the consideration of degree functions
\[ \EE(r,\fg) \lra \NN_0 \ \ ; \ \ \fe \mapsto \deg^j(M|_\fe)  \ \ \ \ \ \ (r \ge 3)\]
provides no additional information. Moreover, if $\fg$ is reductive, then $\EE(\msrke(\fg),\fg)$ and $\EE(\msrke(\fg)\!-\!1,\fg)$ are usually not connected, cf.\ \cite{Pa,PS}.

(2) Since $V(\fe)\!=\!\fe$ for all $\fe \in \EE(2,\fg)$ it is of course possible to define degree functions
\[ \msdeg^j_M : \EE(2,\fg)_{\tilde{U}_{M,j}} \lra \NN_0 \ \ \ \ \ \ \ (j \in \{1,\ldots,p\!-\!1\})\]
on the open subset $\EE(2,\fg)_{\tilde{U}_{M,j}}$ of $\EE(2,\fg)$ (cf.\ Lemma \ref{LCF2}) for an arbitrary $U_0(\fg)$-module $M$. (Here $\tilde{U}_{M,j}:=\{x \in V_{\EE(2,\fg)}\!\smallsetminus\!\{0\} \ ; \ [x] \in U_{M,j}\}$.)
However, at this juncture the behavior of these functions remains somewhat obscure. \end{Remarks}

\bigskip

\subsection{Criteria for $\msdeg^j_M$ being constant}\label{S:CBC}
Throughout, $(\fg,[p])$ denotes a restricted Lie algebra. Given a $p$-subalgebra $\fh \subseteq \fg$, we let
\[ \Nor_\fg(\fh):=\{x \in \fg \ ; \ [x,\fh] \subseteq \fh\}\]
be the {\it normalizer} of $\fh$ in $\fg$. Recall that $\Nor_\fg(\fh)$ is a $p$-subalgebra of $\fg$.

\bigskip

\begin{Lem} \label{CBC1} Suppose that $\msrke(\fg)\ge 2$, and let $M$ be a $U_0(\fg)$-module of constant $j$-rank.
\begin{enumerate}
\item If the morphism $\mspl_M\circ\msim^j_M$ is homogeneous, then $\msdeg^j_M$ is constant.
\item If $M$ is self-dual, then $\msdeg^j_M$ is constant.
\item If there exists a closed connected subspace $\EE(2,\fg)\subseteq \cX \subseteq \Gr_2(\fg)^{V(\fg)}$, then $\msdeg^j_M$ is constant.
\item If $V(\fg)$ is a linear subspace of $\fg$, then $\msdeg^j_M$ is constant. \end{enumerate} \end{Lem}

\begin{proof} (1) Given $\fe \in \EE(2,\fg)$, the canonical inclusion $\fe \subseteq \fg$ defines a morphism $\iota_\fe : \PP(\fe) \lra \PP(\fg)$ of degree $1$ that
factors through $\PP(V(\fg))$. We have
\[ \msdeg^j_M(\fe) = \deg^j(M|_\fe) = \deg(\mspl_M\circ\msim^j_{M|_{\fe}}) = \frac{\deg(\mspl_M\circ\msim^j_{M|_{\fe}})}{\deg(\iota_\fe)}.\]
Thanks to \cite[(1.1.5)]{Fa17}, the right-hand fraction is independent of the choice of the morphism $\iota_\fe : \PP(\fe) \lra \PP(\fg)$ ($\fe \in \EE(2,\fg)$). As a
result, the map $\msdeg^j_M$ is constant.

(2) Let $\fe \in \EE(2,\fg)$. Since $M$ is self-dual, \cite[(4.2.2)]{Fa17} implies
\[ \msdeg^j_M(\fe) = \deg^j(M|_\fe) = \frac{j\rk^j(M|_\fe)}{2} = \frac{j\rk^j(M)}{2},\]
so that $\msdeg^j_M$ is constant.

(3) We consider the morphism
\[ \varphi^j_M : \PP(V_\cX) \lra \PP(\bigwedge^{\rk^j(M)}(M)) \ \ ; \ \ x \mapsto \mspl_M\circ\msim^j_M(x).\]
Since $\cX$ is connected, Theorem \ref{LCF3} shows that the map
\[ \msdeg_{\varphi^j_M} : \cX \lra \NN_0\]
is constant. Hence this also holds for $\msdeg^j_M=\msdeg_{\varphi^j_M}|_{\EE(2,\fg)}$.

(4) If $V(\fg)$ is a linear subspace, then $\Gr_2(\fg)^{V(\fg)}\cong \Gr_2(V(\fg))$ is irreducible. Hence (3) yields the result.  \end{proof}

\bigskip
\noindent
We are now in a position to verify Theorem B.

\bigskip

\begin{Thm} \label{CBC2} Suppose that $\msrke(\fg)\!\ge\! 2$ and let $M$ be a $U_0(\fg)$-module of constant $j$-rank with $j$-degree function
$\msdeg^j_M : \EE(2,\fg) \lra \NN_0$.
\begin{enumerate}
\item If $\fe_0 \in \EE(2,\fg)$, then $\msdeg^j_M$ is constant on $\EE(2,\Nor_\fg(\fe_0))$.
\item If $\fg\!=\!\Lie(G)$ is an algebraic Lie algebra, then $\msdeg^j_M$ is constant.
\item If $p\!\ge\!5$ and $\dim V(C(\fg))\!\ne\!1$, then $\msdeg^j_M$ is constant. \end{enumerate} \end{Thm}

\begin{proof} (1) Since $\fe_0$ is an elementary abelian $p$-ideal of $\Nor_\fg(\fe_0)$, the assertion follows from Corollary \ref{Bc3} and Lemma \ref{CBC1}.

(2) This is a consequence of Lemma \ref{CBC1} and Theorem \ref{AL3}.

(3) We first assume that $\dim V(C(\fg))\!=\!0$, so that $C(\fg)$ is a torus. We consider the factor algebra $\fg' := \fg/C(\fg)$ along with the canonical projection $\pi : \fg \lra \fg'$. Let
\[ \CC(2,\fg) := \{\fv \in \Gr_2(\fg)^{V(\fg)} \ ; \ [\fv,\fv] \subseteq C(\fg)\}.\]
The Lie bracket defines a linear map
\[ b : \bigwedge^2(\fg) \lra \fg \ \ ; \ \ v\wedge w \mapsto [v,w].\]
Letting $\mspl_\fg : \Gr_2(\fg) \lra \PP(\bigwedge^2(\fg))$ be the Pl\"ucker embedding, we conclude that
\[\CC(2,\fg) = \{ \fv \in \Gr_2(\fg)^{V(\fg)} \ ; \ \mspl_\fg(\fv) \in \PP(b^{-1}(C(\fg)))\} = \mspl_\fg^{-1}(\PP(b^{-1}(C(\fg))))\cap\Gr_2(\fg)^{V(\fg)}\]
is closed.

Since $C(\fg)$ is a torus, we have $\fv\cap C(\fg) = (0)$ for every $\fv \in \Gr_2(\fg)^{V(\fg)}$. Hence $\pi$ induces a morphism
\[ \Gr_2(\fg)^{V(\fg)} \lra \Gr_2(\fg') \ \ ; \ \ \fv \mapsto \pi(\fv).\]
Note that this map restricts to a morphism
\[ \pi_\ast : \CC(2,\fg) \lra \EE(2,\fg')\]
of projective varieties.

Let $\fe' \in \EE(2,\fg')$. Then $\ff\! :=\! \pi^{-1}(\fe')$ is a $p$-subalgebra of $\fg$ containing $C(\fg)$ such that $[\ff,\ff] \subseteq C(\fg)$ and $\ff^{[p]} \subseteq C(\fg)$.
This implies that $[p] : \ff \lra \ff$ is $p$-semilinear along with $\ff = V(\ff)\!\oplus\!C(\fg)$. Consequently, $V(\ff) \in \CC(2,\fg)$ and $\pi_\ast(V(\ff))
= \fe'$. As a result, the morphism $\pi_\ast$ is surjective.

Let $\fv \in \CC(2,\fg)$. Then we have $\pi^{-1}(\pi_\ast(\fv)) = \fv\!\oplus\!C(\fg)$, so that $\fv = V(\pi^{-1}(\pi_\ast(\fv)))$. Consequently, $\pi_\ast$ is
also injective.

We conclude that $\pi_\ast$ is in fact a homeomorphism. In view of Corollary \ref{CRL2}, the variety $\CC(2,\fg)$ is connected and Lemma \ref{CBC1} ensures that the $j$-degree function
$\msdeg^j_M$ is constant.

If $\dim V(C(\fg))\!\ge\! 2$, then there are linearly independent elements $x,y \in V(C(\fg))$. Hence $\fe:=kx\!\oplus\!ky \in \EE(2,\fg)$ is a $p$-ideal and the assertion follows from (1).  \end{proof}

\bigskip
\noindent
The following example shows that the variety $\EE(2,\fg)$ may be disconnected even if $\Gr_2(\fg)^{V(\fg)}$ is irreducible.

\bigskip

\begin{Example} We consider the $8$-dimensional vector space
\[ \fg := \fe\!\oplus\!\ff\!\oplus\!T(\fg),\]
where $\fe\!:=\!kx_1\!\oplus\!kx_2$, $\ff\!:=\!ky_1\!\oplus\!ky_2$ and $T(\fg)\!:=\!\bigoplus_{i,j=1}^2kz_{ij}$. The Lie bracket and the $p$-map are given
by
\[ [\fe,\fe] = (0) = [\ff,\ff] \ \ ; \  \ [T(\fg),\fg] = (0) \ \ ; \  \ [x_i,y_j] = z_{ij}\]
as well as
\[ \fe^{[p]}=\{0\}=\ff^{[p]} \ \ ; \  \ z_{ij}^{[p]}=z_{ij}.\]
Since $p\!\ge\!3$, we have $V(\fg)\!=\!\fe\!\oplus\!\ff$, so that $\Gr_2(\fg)^{V(\fg)} \cong \Gr_2(\fe\!\oplus\!\ff)$ is irreducible.

Note that the map
\[ \zeta : \fe\otimes_k\!\ff \lra T(\fg) \   \   ;   \     \  a\otimes b \lra [a,b]\]
is surjective and hence bijective.

Suppose that $a,b \in V(\fg)$ are linearly independent and such that $[a,b]\!=\!0$. Writing $a\!=\!x\!+\!y$ and $b\!=\!x'\!+\!y'$ with $x,x'\in \fe$ and $y,y' \in \ff$,
we obtain $0\! =\! [x,y']\!+\![y,x'] = \zeta(x\otimes y'\!-\!x'\otimes y)$, so that
\[ (\ast) \ \ \ \ \ \ \ \ x\otimes y'\!-\!x'\otimes y=0.\]
If $\dim_k(kx\!+\!kx')\!=\!1$, we may assume that $x\!\ne\!0$ and $x'\!=\!\alpha x$. Then $0\!=\!x\otimes(y'\!-\!\alpha y)$, whence $y'\!=\!\alpha y$. This implies
$b\!=\!\alpha a$, a contradiction.

In view of ($\ast$), the assumption $\dim_k(kx\!+\!kx')\!=\!2$ yields $y\!=\!0\!=\!y'$, so that $ka\oplus kb\!=\!\fe$. Alternatively, $\dim_k(kx\!+\!kx')\!=\!0$ and $ka\!\oplus\!kb\!=\!\ff$. As a result, the variety 
$\EE(2,\fg)\!=\!\{\fe,\ff\}$ is not connected. \end{Example}

\bigskip
\noindent
Given a $U_0(\fg)$-module $M$ and $j \in \{1,\ldots,p\!-\!1\}$, we put
\[\fK^j(M) :=\sum_{[x] \in \PP(V(\fg))} \ker x^j_M.\]
If $M$ has constant $j$-rank and $\PP(V(\fg))$ is irreducible, then $\fK^j(M)$ coincides with the $j$-th power generic kernel defined in \cite[\S9]{CFS}. We also define $\fK(M)\!:=\!\fK^1(M)$.

\bigskip

\begin{Cor} \label{CBC3} Let $(\fg,[p])$ be a restricted Lie algebra with $\msrke(\fg)\ge 2$, $M$ be a $U_0(\fg)$-module of constant $1$-rank.
\begin{enumerate}
\item The function
\[ \fK_M : \EE(2,\fg) \lra \NN_0 \ \ ; \ \ \fe \mapsto \dim_k\fK(M|_\fe)\]
is locally constant.
\item If $\fg\!=\!\Lie(G)$ is algebraic, or if $p\! \ge\! 5$ and $\dim V(C(\fg))\! \ne \! 1$, then $\fK_M$ is constant. \end{enumerate}.\end{Cor}

\begin{proof} Let $\fe \in \EE(2,\fg)$. In view of \cite[(6.2.8)]{Fa17}, we have $\deg^1(M|_\fe) = \dim_kM/\fK(M|_\fe)$. Hence
\[\fK_M(\fe)=\dim_kM\!-\!\msdeg^1_M(\fe),\]
so that our assertions follow from Theorem \ref{LCF3} and Theorem \ref{CBC2}, respectively. \end{proof}

\bigskip

\subsection{Miscellaneous observations}
We shall show in the next section that the function $\msdeg^j_M$ of a $U_0(\fg)$-module $M$ of constant $j$-rank may not be constant in case the variety $V(C(\fg))$ is one-dimensional.
In view of Lemma \ref{CBC1}, this also implies that the morphism $\mspl_M\circ \msim^j_M : \PP(V(\fg)) \lra \PP(\bigwedge^{\rk^j(M)}(M))$ is not necessarily homogeneous.

Recall that the finite-dimensional Hopf algebra $U_0(\fg)$ is a Frobenius algebra, which is symmetric if and only if $\tr(\ad x)=0$ for all $x \in \fg$, cf.\ \cite{Hu78}. The reader is referred to \cite{ARS} for the theory of 
almost split sequences. We require the following basic subsidiary results. 

\bigskip

\begin{Lem} \label{MO1} Let $(\fg,[p])$ be a restricted Lie algebra such that $U_0(\fg)$ is symmetric. If $S$ is a self-dual simple $U_0(\fg)$-module, then the following statements hold:
\begin{enumerate}
\item The projective cover $P(S)$ of $S$ is self-dual.
\item We have $\Rad(P(S)) \cong (P(S)/\Soc(P(S)))^\ast$.
\item The heart $\Ht(P(S)):= \Rad(P(S))/\Soc(P(S))$ is self-dual. \end{enumerate} \end{Lem}

\begin{proof} (1) Since $P(S)^\ast$ is an indecomposable, injective module with $S\cong S^\ast \subseteq P(S)^\ast$, it follows that $P(S)^\ast \cong I(S)$ is the injective hull
of $S$. As $U_0(\fg)$ is symmetric, there is an isomorphism $P(S)\cong I(S)$, so that $P(S)$ is self-dual.

(2) Owing to (1), dualization of the exact sequence
\[ (0) \lra S \lra P(S) \lra P(S)/\Soc(P(S)) \lra (0)\]
gives an exact sequence
\[ (0) \lra (P(S)/\Soc(P(S)))^\ast \lra P(S) \lra S \lra (0),\]
so that $(P(S)/\Soc(P(S)))^\ast \cong \Rad(P(S))$.

(3) In view of (1) and (2), dualization of the standard almost split sequence \cite[(V.5.5)]{ARS} provides an almost split sequence
\[ (0) \lra \Rad(P(S)) \lra \Ht(P(S))^\ast\!\oplus\!P(S)\lra P(S)/\Soc(P(S)) \lra (0).\]
The unicity of almost split sequences \cite[(V.1.16)]{ARS} thus yields $\Ht(P(S))^\ast\!\oplus\! P(S) \cong \Ht(P(S))\!\oplus\!P(S)$. Since $P(S)$ is indecomposable, the Theorem of
Krull-Remak-Schmidt implies our assertion.\end{proof}

\bigskip
\noindent
Given an automorphism $\lambda \in \Aut_p(\fg)$ of a restricted Lie algebra $(\fg,[p])$, we denote the induced automorphism of $U_0(\fg)$ also by $\lambda$. If $M \in
\modd U_0(\fg)$, let $M^{(\lambda)}$ be the $U_0(\fg)$-module with underlying $k$-space $M$ and action
\[ u\dact m := \lambda^{-1}(u).m \ \ \ \ \ \ \ \ \ \forall \ u \in U_0(\fg), \, m \in M.\]

\bigskip

\begin{Lem} \label{MO2} Let $M$ be a $U_0(\fg)$-module of constant $j$-rank. Then the following statements hold:
\begin{enumerate}
\item If $\lambda \in \Aut_p(\fg)$, then $M^{(\lambda)}$ has constant $j$-rank, and we have
\[ \msdeg^j_{M^{(\lambda)}}(\fe)= \msdeg^j_M(\lambda^{-1}(\fe))\]
for all  $\fe \in \EE(2,\fg)$.
\item If $\varphi : M \lra N$ is an isomorphism of $U_0(\fg)$-modules, then $N$ has constant $j$-rank and $\msdeg^j_N = \msdeg^j_M$. \end{enumerate} \end{Lem}

\begin{proof} (1) Let $\lambda \in \Aut_p(\fg)$ and $\fe \in \EE(2,\fg)$. Given $x \in V(\fg)\!\smallsetminus\!\{0\}$, we have $\lambda^{-1}(x) \in V(\fg)\!\smallsetminus\!\{0\}$ and
\[ x^j_{M^{(\lambda)}} = \lambda^{-1}(x)^j_M,\]
so that
\[ \msim^j_{M^{(\lambda)}}([x]) = \msim_M^j([\lambda^{-1}(x)]).\]
Thus, the linear map $\lambda^{-1}$ induces a morphism $\lambda^{-1} :\PP(\fe) \lra \PP(\lambda^{-1}(\fe))$ of degree $1$ such that  $\msim^j_{M^{(\lambda)}}|_\fe =
(\msim^j_M\circ \lambda^{-1})|_\fe$. It now follows from \cite[(1.1.5)]{Fa17} that
\begin{eqnarray*}
 \msdeg^j_{M^{(\lambda)}}(\fe) & = & \deg(\mspl_M\circ \msim^j_{M^{(\lambda)}}|_\fe) = \deg(\mspl_M\circ\msim^j_M|_{\lambda^{-1}(\fe)}\circ \lambda^{-1}|_\fe) \\
 & = &  \deg(\mspl_M\circ\msim^j_M|_{\lambda^{-1}(\fe)}) = \msdeg^j_M(\lambda^{-1}(\fe)),
 \end{eqnarray*}
as desired.

(2) Given $x \in V(\fg)$, we have $x^j_N\circ \varphi = \varphi\circ x^j_M$, so that $N$ has constant $j$-rank $\rk^j(N)=\rk^j(M)$ and
\[ \msim^j_N = \kappa_\varphi \circ \msim^j_M,\]
where $\kappa_\varphi(V)=\varphi(V)$ for all $V \in \Gr_{\rk^j(M)}(M)$. Note that $\varphi$ induces an isomorphism $\eta : \PP(\bigwedge^{\rk^j(M)}(M))$ $\lra \PP(\bigwedge^{\rk^j(N)}(N))$ such that
\[ \eta \circ \mspl_M = \mspl_N\circ \kappa_\varphi.\]
We thus obtain
\[ \mspl_N\circ \msim^j_N = \mspl_N\circ \kappa_\varphi \circ \msim_M^j = \eta \circ (\mspl_M\circ \msim_M^j).\]
Since the isomorphism $\eta$ necessarily has degree $1$, our assertion follows from \cite[(1.1.5)]{Fa17}.\end{proof}

\bigskip

\subsection{The Example $\fsl(2)_s$}\label{S:Ex}
In the sequel, we let $\{e,h,f\}$ be the standard basis of $\fsl(2)$, i.e.,
\[ e:= \left(\begin{smallmatrix} 0 & 1 \\ 0 & 0 \end{smallmatrix}\right) \  \ ; \  \   h:= \left(\begin{smallmatrix} 1 & 0 \\ 0 & -1 \end{smallmatrix}\right) \  \ ;  \ \
 f:= \left(\begin{smallmatrix} 0 & 0 \\ 1 & 0 \end{smallmatrix}\right).\]
We consider a particular central extension $\fsl(2)_s\!:=\! \fsl(2)\!\oplus\!kc_0$ of $\fsl(2)$ by a one-dimensional elementary abelian ideal. By definition, this extension splits as an extension of
ordinary Lie algebras. The $p$-map is given by the semisimple linear form $\psi : \fsl(2) \lra k$, that satisfies $\psi(h)\!=\!1$ and $\ker\psi\!=\!ke\!\oplus\!kf$. Accordingly, we have
\[ (x,\alpha c_0)^{[p]} = (x^{[p]},\psi(x)^pc_0)\]
for all $x \in \fsl(2)$ and $\alpha \in k$. Since $\ker \psi\!=\! ke\!\oplus\!kf$, this readily implies
\[ V(\fsl(2)_s) = \{ (x,\lambda c_0) \ ; \ x \in V(\fsl(2))\cap\ker \psi, \, \lambda \in k\} = ke\!\oplus\!kc_0 \cup kf\!\oplus\!kc_0,\]
whence
\[ \EE(2,\fsl(2)_s) = \{ \fe_e, \fe_f\}=\Gr_2(\fsl(2)_s)^{V(\fsl(2)_s)},\]
where $\fe_x\!:=\!kx\!\oplus\!kc_0$ for $x \in \{e,f\}$.

We denote by
\[ \omega : \fsl(2)_s \lra \fsl(2)_s  \  \  ;  \  \ e\mapsto f \ , \ f \mapsto e \ , \ h \mapsto -h \ , \ c_0 \mapsto -c_0\]
the ``Cartan involution'' of $\fsl(2)_s$. Since $U_0(\fsl(2)_s)$ has simple modules $L(0), \ldots, L(p\!-\!1)$, where $\dim_kL(i)\!=\!i\!+\!1$, we have
\[ L(i)^{(\omega)} \cong L(i) \]
for all $i \in \{0,\ldots,p\!-\!1\}$. Moreover, $\omega(\fe_e)=\fe_f$ and $\omega(\fe_f)=\fe_e$, so that $\omega \not \in \Aut_p(\fsl(2)_s)^\circ$.

\bigskip
\noindent
Let $P(i)$ be the projective cover of $L(i)$. As $L(i)$ is $\omega$-stable, we obtain $P(i)\cong P(i)^{(\omega)}$. According to \cite[(1.5)]{FS}, the module $P(i)$ is $2p^2$-dimensional, so that
$P(i)|_{\fe_x} \cong 2 U_0(\fe_x)$ is self-dual. This readily implies that $P(i)$ has constant $j$-rank $\rk^j(P(i))=2p(p\!-\!j)$, while $\deg^j(P(i)|_{\fe_x})=jp(p\!-\!j)$ for $j \in
\{1,\ldots,p\!-\!1\}$ and $x \in \{e,f\}$, cf.\ \cite[\S4]{Fa17}.

A short exact sequence
\[ (0) \lra N \lra E \lra M \lra (0)\]
of $U_0(\fsl(2)_s)$-modules is referred to as {\it locally split}, provided the restricted sequence
\[ (0) \lra N|_{U_0(kx)} \lra E|_{U_0(kx)} \lra M|_{U_0(kx)} \lra (0)\]
splits for every $x \in V(\fsl(2)_s)\!\smallsetminus\!\{0\}$.

\bigskip

\begin{Lem} \label{Ex1} Let $0\! \le\! i\!\le\! p\!-\!2$ and $1\!\le\! j\!\le\! p\!-\!i\!-\!1$. Then the following statements hold:
\begin{enumerate}
\item $\Rad(P(i))$ is a $U_0(\fsl(2)_s)$-module of constant $j$-rank $\rk^j(\Rad(P(i)))\! = \! 2p(p\!-\!j)\!-\!i\!-\!1$.
\item We have $\deg^j(\Rad(P(i))|_{\fe_x})\!=\!  j(p(p\!-\!j)\!-\!i\!-\!1)$ for $x \in \{e,f\}$.
\item $P(i)/\Soc(P(i))$ is a $U_0(\fsl(2)_s)$-module of constant $j$-rank such that $\deg^j((P/\Soc(P(i)))|_{\fe_x})\!=\! \deg^j(P(i)|_{\fe_x})$ for $x \in \{e,f\}$.
\item $\Ht(P(i))$ is a self-dual module of constant $j$-rank $\rk^j(\Ht(P(i)))\!=\!2p(p\!-\!j)\!-\!2i\!-\!2$ such that $\deg^j(\Ht(P(i))|_{\fe_x})\!=\! j(p(p\!-\!j)\!-\!i\!-\!1)$ for
         $x \in \{e,f\}$. \end{enumerate}\end{Lem}

\begin{proof} Let $j \in \{1, \ldots,p\!-\!i\!-\!1\}$. We claim that

\medskip
\[ (\ast) \  \  \  \  \  \  \  \  \fK^j(P(i)|_{\fe_e}) \subseteq \Rad(P(i)).\]

\smallskip
\noindent
Since $P(i)$ is a projective $U_0(kc_0)$-module, we have $\ker (c_0)_{P(i)}^j = \im (c_0)_{P(i)}^{p-j} \subseteq \im (c_0)_{P(i)}$. As $c_0$ is a central, nilpotent element of $U_0(\fsl(2)_s)$,
it follows that $\im (c_0)_{P(i)} \subseteq \Rad(P(i))$.

By virtue of \cite[(1.5)]{FS}, $Q(i):= P(i)/c_0.P(i)$ is a principal indecomposable $U_0(\fsl(2))$-module. Let $\pi : P(i) \lra Q(i)$ be the canonical projection. Let $m \in \ker(\alpha e \!+\!\beta c_0)^j_{P(i)}$.
If $\alpha\! \ne\! 0$, the binomial formula yields $e^j\dact m \in \im (c_0)_{P(i)}$, so that $\pi(m) \in \ker e^j_{Q(i)}$. If $\alpha \!=\!0$ and $\beta\!\ne\!0$, then $m \in \im (c_0)_{P(i)}$, so that $\pi(m)=0$.
We therefore obtain
\[ \pi(\fK^j(P(i)|_{\fe_e})) \subseteq \ker e^j_{Q(i)}.\]

Let $T \subseteq \SL(2)$ be the standard maximal torus of diagonal matrices and recall that $\ZZ$ is the character group of $T$. Then $Q(i)$ is an $\SL(2)$-module (cf.\ \cite[Thm.3]{Hu73}) and
$\ker e^j_{Q(i)}$ is a $T$-submodule of $Q(i)$. Suppose that $\ker e^j_{Q(i)}\!\smallsetminus\! \Rad(Q(i)) \ne \emptyset$ and let $j_0 \in \{1,\ldots, j\}$ be minimal subject to this property. As
$\ker e^{j_0}_{Q(i)} \not \subseteq \Rad(Q(i))$, there is a weight vector $v_\lambda \in Q(i)_\lambda$ such that $v_\lambda \in  \ker e^{j_0}_{Q(i)}\!\smallsetminus\!\Rad(Q(i))$. Let $\lambda \in \ZZ$
be maximal subject to this property and consider the projection $\sigma : Q(i) \lra  L(i)$. Then we have $\sigma (v_\lambda)\! \ne\! 0$, while the choice of $\lambda$ in conjunction with $e.v_\lambda \in
(\ker e^{j_0}_{Q(i)})_{\lambda+2}$ gives $e.\sigma(v_\lambda)\!=\!0$. This implies $\lambda\! =\! i$.  Since $v_i \not \in \Rad(Q(i))$, we obtain $Q(i)\!=\! U_0(\fsl(2))v_i\! =\! U_0(kf)(\sum_{\ell=0}^{j_0-1}ke^\ell.v_i)$.
This shows that the set $\wt(Q(i)) \subseteq \ZZ$ of weights of $Q(i)$ satisfies $\wt(Q(i)) \subseteq \{i\!+\!2s\!-\!2t \ ; \ (s,t) \in \{0,\ldots,j_0\!-\!1\} \!\times\!\{0,\ldots,p\!-\!1\}\}$.  If $i\!+\!2j_0\!-\!2$ is not a
weight of $Q(i)$, then $e^{j_0-1}.v_i\! =\! 0$, so that $v_i \in \ker e^{j_0-1}_{Q(i)}\!\smallsetminus \! \Rad(Q(i))$. By choice of $j_0$, this implies $j_0\!=\!1$, whence $i\!+\!2j_0\!-\!2\! =\!i$. This contradicts our
assumption that the former number is not a weight. We thus conclude that $i\!+\!2j_0\!-\!2$ is the maximal weight of $Q(i)$.

According to the proof of \cite[Thm.3]{Hu73}, the $\SL(2)$-module $Q(i)$ has a presentation $[Q(i)]\!=\! 2[L(i)]\!+\![L(2p\!-\!2\!-\!i)]$ in the Grothendieck group of $\SL(2)$. Consequently,
\[ 2p\!-\!2\!-\!i = i\!+\!2j_0\!-\!2,\]
which implies $j \!\ge \!j_0\! =\! p\!-\!i$, which contradicts $j\!\le\!p\!-\!i\!-\!1$.

We thus have $\pi(\fK^j(P(i)|_{\fe_e})) \subseteq \ker e^j_{Q(i)} \subseteq \Rad(Q(i))\!=\!\pi(\Rad(P(i)))$, so that $\fK^j(P(i)|_{\fe_e}) \subseteq \Rad(P(i))\!+\!\ker\pi \subseteq \Rad(P(i))$, as desired. \hfill $\diamond$

\medskip
\noindent
(1), (2) Recall that $P(i)$ is a $U_0(\fsl(2)_s)$-module of constant $j$-rank $\rk^j(P(i))\!=\!2p(p\!-\!j)$. We consider the exact sequence
\[ (0) \lra \Rad(P(i))|_{\fe_e} \lra P(i)|_{\fe_e} \lra L(i)|_{\fe_e} \lra (0)\]
of $U_0(\fe_e)$-modules. In view of ($\ast$), we may apply \cite[(4.1.4)]{Fa17} to see that $\Rad(P(i))|_{\fe_e}$ has constant $j$-rank $\rk^j(\Rad(P(i))|_{\fe_e}) = 2p(p\!-\!j)-\!i\!-\!1$ and $j$-degree
$\deg^j(\Rad(P(i))|_{\fe_e}) = jp(p\!-\!j)\!-\!ji\!-\!j= j(p(p\!-\!j)\!-\!i\!-\!1)$.

Since $\Rad(P(i))$ is stable under the Cartan involution, Lemma \ref{MO2} ensures that we obtain the same formulae for $\Rad(P(i))|_{\fe_f}$. In particular, $\Rad(P(i))$ has constant $j$-rank
$\rk^j(\Rad(P(i))) = 2p(p\!-\!j)\!-\!i\!-\!1$.

(3) Recall that $U_0(\fsl(2)_s)$ is a symmetric algebra such that every simple module is self-dual. According to Lemma \ref{MO1}, we therefore have $P(i)/\Soc(P(i)) \cong \Rad(P(i))^\ast$, implying that
the former module has constant $j$-rank $\rk^j(\Rad(P(i)))$. The rank-degree formula \cite[(4.2.2)]{Fa17} implies for $x \in \{e,f\}$
\begin{eqnarray*}
\deg^j((P(i)/\Soc(P(i)))|_{\fe_x}) & = & j\rk^j(\Rad(P(i)))\!-\!\deg^j(\Rad(P(i))|_{\fe_x})\\
  & = &j(2p(p\!-\!j)\!-\!i\!-\!1)\!-\!j(p(p\!-\!j)\!-\!i\!-\!1)\\
 & = & jp(p\!-\!j) = \deg^j(P(i)|_{\fe_x}).
 \end{eqnarray*}

(4) Since the standard AR-sequence is locally split (cf.\ for instance \cite[(2.3)]{Fa14}), $\Ht(P(i))$ is a module of constant $j$-rank
\begin{eqnarray*}
 \rk^j(\Ht(P(i))) & = & \rk^j(\Rad(P(i)))\!+\!\rk^j(P(i)/\Soc(P(i))\!-\!\rk^j(P(i))\\
                     &= & 2\rk^j(\Rad(P(i)))\!-\!\rk^j(P(i)) = 2p(p\!-\!j)\!-\!2i\!-\!2.\end{eqnarray*}
Lemma \ref{MO1} ensures that $\Ht(P(i))$ is a self-dual module, so that \cite[(4.2.2)]{Fa17} implies
\[ \deg^j(\Ht(P(i))|_{\fe_x})=\frac{1}{2}j\rk^j(\Ht(P(i)))=j(p(p\!-\!j)\!-\!i\!-\!1),\]
as desired. \end{proof}

\bigskip
\noindent
We put $\fb_s\! :=\! k(h\!+\!c_0)\!\oplus\!ke$ and $\fb^-_s:\!=\!k(h\!+\!c_0)\!\oplus\!kf$. Observe that $\omega(\fb_s)\!=\!\fb^-_s$. For $i \in \{0,\ldots, p\!-\!1\}$, we define {\it baby Verma modules}
\[ Z(i) := U_0(\fsl(2)_s)\!\otimes_{U_0(\fb_s)}\!k_i \ \ \text{and} \ \ Z'(i) := U_0(\fsl(2)_s)\!\otimes_{U_0(\fb^-_s)}\!k_i,\]
with $h\!+\!c_0$ acting on $k_i$ via $i$. Let $i \in \{0,\ldots,p\!-\!2\}$. By virtue of \cite[(3.1)]{FS}, the module $Z(i)$ is uniserial of Loewy length $\ell\ell(Z(i))=2p$ with top $L(i)$ and socle
$L(p\!-\!2\!-\!i)$.

Since $U_0(\fsl(2)_s)\!:\!U_0(\fb_s)$ is a Frobenius extension, general theory (cf.\ \cite[(1.15)]{Ja00}) implies that $Z(i)^\ast \cong Z(p\!-\!2\!-\!i)$, so that $Z(i)$ is not self-dual for $i \in \{0,\ldots,p\!-\!2\}$,
while $Z(p\!-\!1)$ is self-dual. On the other hand, the restriction $Z(i)|_{\fe_f} \cong U_0(\fe_f)$ is self-dual.

\bigskip

\begin{Lem} \label{Ex2} Setting $Z'(p)\!:=\!Z'(0)$, we have
\[ Z(i)^{(\omega)} \cong Z'(p\!-\!i)\]
for every $i \in \{0,\ldots,p\!-\!1\}$. In particular, $Z'(i)$ is uniserial for $2\!\le\!i\!\le\!p$. \end{Lem}

\begin{proof} We consider the bilinear map
\[ \gamma : U_0(\fsl(2)_s) \! \times \! k_i \lra Z'(p\!-\!i) \  \ ; \  \ (u,\alpha) \mapsto \omega(u)\!\otimes\!\alpha.\]
This map is $U_0(\fb_s)$-balanced:
\[ \gamma (u(h\!+\!c_0),\alpha) = \omega(u)(-(h\!+\!c_0))\!\otimes\!\alpha = \omega(u)\!\otimes\!-(p\!-\!i)\alpha = \gamma (u,i\alpha) = \gamma(u,(h\!+\!c_0).\alpha).\]
There thus results a bijective $k$-linear map
\[ \hat{\gamma} : Z(i) \lra Z'(p\!-\!i) \  \ ; \  \ u\!\otimes\!\alpha \mapsto \omega(u)\!\otimes\!\alpha.\]
Given $a \in U_0(\fsl(2)_s)$, we have
\[ \hat{\gamma}(a\dact (u\!\otimes\!\alpha)) = \hat{\gamma}(\omega^{-1}(a)u\!\otimes\!\alpha) = a\omega(u)\!\otimes\!\alpha = a.\hat{\gamma}(u\!\otimes\!\alpha).\]
As a result, $\hat{\gamma}$ is an isomorphism $Z(i)^{(\omega)}\cong Z'(p\!-\!i)$ of $U_0(\fsl(2)_s)$-modules. \end{proof}

\bigskip
\noindent
Since $\Top(Z(i)) \cong L(i)$, while $\Soc(Z(i))\cong L(p\!-\!2\!-\!i)$ (cf.\ \cite[(3.1)]{FS}), the $\omega$-stable module $P(i)$ is a projective cover of $Z(i)$ and an injective hull of
$Z(p\!-\!2\!-\!i)$. In view of Lemma \ref{Ex2}, we thus have $U_0(\fsl(2)_s)$-linear maps
\[ \iota : Z'(i\!+\!2) \hookrightarrow P(i) \  \ \text{and} \  \ \pi : P(i) \twoheadrightarrow Z(i)\]
for every $i \in \{0,\ldots,p\!-\!2\}$.

\bigskip
\noindent
Let $(\fg,[p])$ be a restricted Lie algebra, $M$ be a $U_0(\fg)$-module. Then
\[ V(\fg)_M :=\{ x \in V(\fg) \ ; \ M|_{kx} \ \text{is not free}\}\cup\{0\}\]
is called the \textit{rank variety} of $M$. The following result provides the decomposition of the hearts of the principal indecomposable $U_0(\fsl(2)_s)$-modules.

\bigskip

\begin{Lem} \label{Ex3} Let $i \in \{0,\ldots,p\!-\!2\}$. Then the following statements hold:
\begin{enumerate}
\item We have $\Rad(Z(i))^{(\omega)} \not \cong \Rad(Z(i))$.
\item We have
\[ \Ht(P(i)) \cong (Z'(i\!+\!2)/\Soc(Z'(i\!+\!2)))\!\oplus\!\Rad(Z(i)).\]  \end{enumerate} \end{Lem}

\begin{proof} (1) The presence of an isomorphism $\Rad(Z(i))^{(\omega)} \cong \Rad(Z(i))$ implies that its restrictions are isomorphisms
\[ \Rad^\ell(Z(i))^{(\omega)} \cong \Rad^\ell(Z(i)) \ \ \ \text{for all} \ \ell \ge 1,\]
so that in particular
\[(\ast) \ \ \ \ \ (\Rad^2(Z(i))/\Rad^4(Z(i)))^{(\omega)} \cong \Rad^2(Z(i))/\Rad^4(Z(i)) .\]
Let $Z_{\fsl(2)}(i)\!:=\!U_0(\fsl(2))\!\otimes_{U_0(kh\oplus ke)}\!k_i$ be the baby Verma module of $U_0(\fsl(2))$ with highest weight $i$. In view of ($\ast$), an
application of \cite[(3.1)]{FS} now yields isomorphisms
\[ Z_{\fsl(2)}(i)^{(\omega)} \cong Z_{\fsl(2)}(i),\]
where $\omega$ denotes the Cartan involution of $\fsl(2)$. Since the rank varieties of these $U_0(\fsl(2))$-modules are $kf$ and $ke$, respectively, we have
reached a contradiction.

(2) In view of \cite[(6.3)]{FS}, the block $\cB(i) \subseteq U_0(\fsl(2)_s)$ containing $P(i)$ is special biserial, so that
\[ \Ht(P(i)) \cong E_1\!\oplus\!E_2\]
is a direct sum of two uniserial modules. We put $J\!:=\!\Rad(U_0(\fsl(2)_s)$. Thanks to \cite[(1.6)]{FS}, we have $(0)\! \ne\! J^{2p}P(i) \subseteq \Soc(P(i)) \cong L(i)$. Thus,
$\Rad^{2p}(P(i))\!=\!\Soc(P(i))$, whence $\ell\ell(E_i)\le \ell\ell(\Ht(P(i)))\!=\! \ell\ell(P(i))\!-\!2\! =\!2p\!-\!1$.

The canonical projection $P(i) \lra Z(i)$ induces a surjective homomorphism $\pi : \Rad(P(i)) \lra \Rad(Z(i))$. Since the module $Z(i)$ is uniserial of Loewy length $2p$, we obtain
\[ \pi(\Soc(P(i)))=\pi(J^{2p-1}(\Rad(P(i)))) = J^{2p-1}\Rad(Z(i))=(0),\]
and there results a surjective homomorphism
\[ \hat{\pi} : \Ht(P(i)) \lra \Rad(Z(i)).\]
Thus, $\Rad(Z(i))\!=\!\hat{\pi}(E_1)\!+\!\hat{\pi}(E_2)$, so that without loss of generality $\hat{\pi}(E_1)\!=\!\Rad(Z(i))$. Since $\ell(E_1)\!\le\! 2p\!-\!1\! =\! \ell(\Rad(Z(i)))$, we conclude that $E_1 \cong \Rad(Z(i))$.

The canonical injection $\iota : Z'(i\!+\!2) \hookrightarrow \Rad(P(i))$ induces a map $\hat{\iota} : Z'(i\!+\!2)/\Soc(Z'(i\!+\!2)) \lra \Ht(P(i))$.
Since $\iota^{-1}(\Soc(\Rad(P(i))) )\! =\! \Soc(Z'(i\!+\!2))$, the map $\hat{\iota}$ is injective. We write $\hat{\iota}\! =\! \binom{\hat{\iota}_1}{\hat{\iota}_2}$,
with linear maps $\hat{\iota}_r : Z'(i\!+\!2)/\Soc(Z'(i\!+\!2)) \lra E_r$. As $Z'(i\!+\!2)/\Soc(Z'(i\!+\!2))$ has a simple socle, there is $r \in \{1,2\}$ such that
$\hat{\iota}_r$ is injective. In view of $\ell(Z'(i\!+\!2)/\Soc(Z'(i\!+\!2)))\! =\! 2p\!-\!1\! \ge\! \ell(E_r)$, it follows that $Z'(i\!+\!2)/\Soc(Z'(i\!+\!2)) \cong E_r$.

A consecutive application of Lemma \ref{Ex2} and \cite[(3.2)]{FS} implies
\[ Z'(i\!+\!2)/\Soc(Z'(i\!+\!2)) \cong (Z(p\!-\!2\!-\!i)/\Soc(Z(p\!-\!2\!-\!i))^{(\omega)} \cong (\Rad Z(i))^{(\omega)}.\]
Thus, (1) in conjunction with $E_1 \cong \Rad(Z(i))$ yields $r\!=\!2$, so that
\[\Ht(P(i)) \cong (Z'(i\!+\!2)/\Soc(Z'(i\!+\!2)))\!\oplus\!\Rad(Z(i)),\]
as asserted. \end{proof}

\bigskip

\begin{Lem} \label{Ex4} Let $i \in \{0,\ldots, p\!-\!2\}$. If $1\!\le \!j\! \le\! p\!-\!i\!-\!1$, then the following statements hold:
\begin{enumerate}
\item The module $\Rad(Z(i))$ has constant $j$-rank $\rk^j(\Rad(Z(i)))\!=\!p(p\!-\!j)\!-\!i\!-\!1$.
\item The module $Z'(i\!+\!2)/\Soc(Z'(i\!+\!2))$ has constant $j$-rank $\rk^j(Z'(i\!+\!2)/\Soc(Z'(i\!+\!2)))\!=\!p(p\!-\!j)\!-\!i\!-\!1$. \end{enumerate} \end{Lem}

\begin{proof} Thanks to (4) of Lemma \ref{Ex1}, the module $\Ht(P(i))$ has constant $j$-rank. Since $\EE(1,\fg)$ is connected, \cite[(4.13)]{CFP13} shows that this holds for every direct
summand of $\Ht(P(i))$. Lemma \ref{Ex3} thus ensures that the modules $\Rad(Z(i))$ and $Z'(i\!+\!2)/\Soc(Z'(i\!+\!2))$ have constant $j$-rank.

We put $\hat{Z}(i)\!:=\!Z(i)/c_0Z(i)$. According to \cite[(3.1)]{FS}, this module is uniserial of dimension $p$. By the same token, we have
$\hat{Z}(i) \cong c_0^\ell Z(i)/c_0^{\ell+1}Z(i)$ for $0\le \ell \le p\!-\!1$. In particular, $\Top(c_0^jZ(i)) \cong L(i)$. Next, we observe that
$c_0^j(\Rad(Z(i)))=\Rad(c_0^jZ(i))$, so that $\rk^j(\Rad(Z(i)))\!=\! \dim_kc_0^jZ(i)\!-\!\dim_kL(i)\!=\! p^2\!-\!jp\!-\!i\!-\!1\!=\! p(p\!-\!j)\!-\!i\!-\!1\!=\!\frac{1}{2}\rk^j(\Ht(P(i))$.
Now Lemma \ref{Ex3} readily yields $\rk^j(Z'(i\!+\!2)/\Soc(Z'(i\!+\!2)))\!=\!p(p\!-\!j)\!-\!i\!-\!1$. \end{proof}

\bigskip
\noindent
Our next result shows in particular that the function $\deg^j_{\Rad(Z(i))} : \EE(2,\fsl(2)_s) \lra \NN_0$ is not constant whenever $1\!\le\! j\! \le\! p\!-\!i\!-\!1$.

\bigskip

\begin{Prop}\label{Ex5} Let $i \in \{0,\ldots, p\!-\!2\}$ and $j \in \{1,\ldots, p\!-\!i\!-\!1\}$. Then the following statements hold:
\begin{enumerate}
\item The module $\Rad(Z(i))|_{\fe_f}$ has constant $j$-rank and $\deg^j(\Rad(Z(i))|_{\fe_f})\! =\! j(\frac{p(p-j)}{2}\!-\!i\!-\!1)$.
\item The module $(Z'(i\!+\!2)/\Soc(Z'(i\!+\!2)))|_{\fe_f}$ has constant $j$-rank and $j$-degree \newline $\deg^j(Z'(i\!+\!2)/\Soc(Z'(i\!+\!2))|_{\fe_f})\! =\! j\frac{p(p-j)}{2}$.
\item The module $\Rad(Z(i))|_{\fe_e}$ has constant $j$-rank and $\deg^j(\Rad(Z(i))|_{\fe_e})\!=\! j\frac{p(p-j)}{2}$.
\end{enumerate} \end{Prop}

\begin{proof} Let $\ell \in \{1,\ldots,p\!-\!1\}$. We first show that

\medskip
($\ast$) \ \ \ \ the module $Z(i)|_{\fe_f}$ is projective and $\deg^\ell(Z(i)|_{\fe_f}) = \frac{\ell p(p\!-\!\ell)}{2}$.

\smallskip
\noindent
General theory implies that $V(\fsl(2)_s)_{Z(i)} \subseteq V(\fb_s) = ke$. Consequently, $V(\fsl(2)_s)_{Z(i)}\cap \fe_f = \{0\}$, so that $Z(i)|_{\fe_f}$
is projective. Since $\dim_kZ(i)=p^2$, we see that $Z(i)|_{\fe_f} \cong U_0(\fe_f)$. Hence $\deg^\ell(Z(i)|_{\fe_f}) = \frac{\ell p(p\!-\!\ell)}{2}$.

(1) Setting $\hat{Z}(i):= Z(i)/c_0Z(i)$, we let $\pi : Z(i) \lra \hat{Z}(i)$ be the canonical projection. Since $\hat{Z}(i)$ is a projective $U_0(kf)$-module of dimension $p$, we have $\dim_k \ker f^j_{\hat{Z}(i)}\!=\!j$. Recall that 
there is a non-split exact sequence
\[ (0) \lra L(p\!-\!2\!-\!i) \lra \hat{Z}(i) \lra L(i) \lra (0)\]
of $U_0(\fsl(2))$-modules. Since $L(p\!-\!2\!-\!i)|_{U_0(kf)} \cong [p\!-\!1\!-\!i]$ is the cyclic module of dimension $p\!-\!1\!-\!i$, our assumption $j\!\le\!p\!-\!i\!-\!1$ yields $\ker f^j_{\hat{Z}(i)} \subseteq L(p\!-\!2\!-\!i)\! =\!
\Rad(\hat{Z}(i))$. As a result, $\ker f^j_{Z(i)} \subseteq \Rad(Z(i))$.

Let $m \in \ker(\alpha f\!+\!\beta c_0)^j_{Z(i)}$, where $(\alpha,\beta) \in k^2\!\smallsetminus\!\{0\}$. If $\alpha\!\ne\! 0$, then $\pi(m) \in \ker f^j_{\hat{Z}(i)} \subseteq \Rad(\hat{Z}(i))$, and $m \in \Rad(Z(i))$. Alternatively, 
the $U_0(kc_0)$-projectivity of $Z(i)$ yields $m \in \ker (c_0)^j_{Z(i)} \subseteq \im (c_0)_{Z(i)}\!= \!\ker \pi$. This implies that $\fK^j(Z(i)|_{\fe_f})\subseteq \Rad(Z(i))$ and \cite[(4.1.4)]{Fa17} in conjunction with ($\ast$) gives
\[ \deg^j(\Rad(Z(i))|_{\fe_f}) = \deg^j(Z(i)|_{\fe_f})\!-\!j\dim_kL(i) = \frac{jp(p\!-\!j)}{2}\!-\!j(i\!+\!1) = j(\frac{p(p\!-\!j)}{2}\!-\!i\!-\!1),\]
as asserted.

(2) In view of Lemma \ref{Ex3} and Lemma \ref{Ex1}, we obtain
\begin{eqnarray*}
\deg^j((Z'(i\!+\!2)/\Soc(Z'(i\!+\!2)))|_{\fe_f}) & = &\deg^j(\Ht(P(i))|_{\fe_f})\!-\!\deg^j(\Rad(Z(i))|_{\fe_f})\\
                                                                   &=  & j(p(p\!-\!j)\!-\!i\!-\!1)\!-\! j(\frac{p(p\!-\!j)}{2}\!-\!i\!-\!1) =  j\frac{p(p\!-\!j)}{2},
\end{eqnarray*}
as asserted.

(3) Since $\Ht(P(i))^{(\omega)} \cong \Ht(P(i))$, Lemma \ref{Ex3} in conjunction with the Theorem of Krull-Remak-Schmidt yields $\Rad(Z(i))^{(\omega)} \cong
Z'(i\!+\!2)/\Soc(Z'(i\!+\!2))$. Observing Lemma \ref{MO2}, we thus obtain
\[ \deg^j(\Rad(Z(i))|_{\fe_e}) = \deg^j(\Rad(Z(i))^{(\omega)}|_{\fe_f}) = \deg^j((Z'(i\!+\!2)/\Soc(Z'(i\!+\!2)))|_{\fe_f}) = j\frac{p(p\!-\!j)}{2},\]
as desired. \end{proof}

\bigskip

\section{Categories of Modules of constant $j$-rank}
Let $(\fg,[p])$ be a restricted Lie algebra. In this section, we apply our results to study full subcategories of the category $\modd U_0(\fg)$ of (finite-dimensional) $U_0(\fg)$-modules, whose objects $M$ satisfy various 
conditions on the operators $x^j_M : M \lra M \ \ (x \in V(\fg)\!\smallsetminus\!\{0\})$.

Given $j \in \{1,\ldots,p\!-\!1\}$, we let $\CR^j(\fg)$ be the category of $U_0(\fg)$-modules of constant $j$-rank. The defining property for the objects $M$ of the full subcategory
$\EIP^j(\fg) \subseteq \CR^j(\fg)$ of the modules with the \textit{equal $j$-images property} is given by
\[ \im x_M^j = \im y^j_M \ \ \ \ \ \ \ \ \ \forall \ x,y \in V(\fg)\!\smallsetminus\!\{0\}.\]
In the context of elementary abelian Lie algebras, these modules naturally generalize the equal $1$-images modules discussed in \cite{CFS}. Being closed under images of morphisms and direct sums,
$\EIP^j(\fg)$ is closed under sums and direct summands. Following \cite{CFS}, we refer to the objects of $\EIP^1(\fg)$ as modules having the {\it equal images property}.

Let $x \in V(\fg)\!\smallsetminus\!\{0\}$. Then $U_0(kx)\cong k[T]/(T^p)$ is a truncated polynomial ring, so that the modules $[i]:=U_0(kx)/U_0(kx)x^i$ for $i \in \{1,\ldots, p\}$ form a complete set
of representatives for the isoclasses of indecomposable $U_0(kx)$-modules. For a $U_0(\fg)$-module $M$ there thus are $a_i(x) \in \NN_0$ such that
\[ M|_{kx} \cong \bigoplus_{i=1}^pa_i(x)[i].\]
The right-hand side is the {\it local Jordan type $\Jt(M,x)$ of $M$ at $x$}. Following \cite{CFP08}, we say that $M$ has {\it constant Jordan type} $\Jt(M)$, provided $\Jt(M,x)=\Jt(M)$ for all $x \in
V(\fg)\!\smallsetminus\!\{0\}$. Basic Linear Algebra tells us that the category $\CJT(\fg)$ of modules of constant Jordan type coincides with $\bigcap_{j=1}^{p-1}\CR^j(\fg)$.

\bigskip

\subsection{Equal $j$-images modules for elementary abelian Lie algebras}
We record the following generalization of \cite[(2.5)]{CFS}:

\bigskip

\begin{Lem} \label{EJIM1} Let $M$ be a $U_0(\fe_r)$-module, $j \in \{1,\ldots, p\!-\!1\}$. Then the following statements are equivalent:
\begin{enumerate}
\item $M$ has the equal $j$-images property.
\item $\im x_M^j\!=\!\Rad^j(M)$ for all $x \in \fe_r\!\smallsetminus\!\{0\}$. \end{enumerate} \end{Lem}

\begin{proof} (1) $\Rightarrow$ (2) Let $J\! :=\! \Rad(U_0(\fe_r))$ be the Jacobson radical of $U_0(\fe_r)$. By assumption, there exists a subspace $V \subseteq M$ such that $\im x_M^j\! =\! V$ for every
$x \in \fe_r\!\smallsetminus\!\{0\}$. Thus, $V \subseteq \Rad^j(M)$.

Thanks to \cite[(1.17.1)]{Be}, we have $J^j\! =\! (\{ x^j \ ; \ x \in \fe_r\})$. Given $u \in J^j$, we thus write $u\! =\! \sum_{i=1}^n x_i^ja_i$, where $x_i \in \fe_r$ and
$a_i \in U_0(\fe_r)$. This yields
\[ u.m = \sum_{i=1}^n (x_i)_M^j(a_i.m) \in \sum_{i=1}^n \im (x_i)^j_M \subseteq V,\]
whence
\[ V \subseteq \Rad^j(M) = J^jM \subseteq V,\]
as desired.

(2) $\Rightarrow$ (1) This is clear. \end{proof}

\bigskip

\begin{Prop} \label{EJIM2} Suppose that $M \in \EIP^j(\fe_r)$ has the equal $j$-images property for some $j \in \{1, \ldots, p\!-\!1\}$. Then the following statements hold:
\begin{enumerate}
\item $\Rad^{j-1}(M)$ has the equal images property.
\item $M$ has the equal $\ell$-images property for all $\ell\!\in \{j,j\!+\!1,\ldots,p\!-\!1\}$.
\item $M$ has Loewy length $\ell\ell(M)\!\le\!p$.
\item Suppose that $r\!\ge\!2$. There exist $c_{j+1},\ldots, c_{\ell\ell(M)} \in \NN$ and functions $a_i : \fe_r\!\smallsetminus\!\{0\} \lra \NN_0$ such that
\[ \Jt(M,x) = \bigoplus_{i=1}^ja_i(x)[i] \oplus \bigoplus_{i=j+1}^{\ell\ell(M)} c_i[i]\]
for every $x \in \fe_r\!\smallsetminus\!\{0\}$. \end{enumerate} \end{Prop}

\begin{proof} (1) Let $U\!:=\! \Rad^{j-1}(M)$. Given $x \in \fe_r\!\smallsetminus\!\{0\}$, Lemma \ref{EJIM1} implies
\[ \Rad(U) = \Rad^j(M) = \im x_M^j  = x_M(\im x_M^{j-1}) \subseteq x_M(U) = \im x_U \subseteq \Rad(U),\]
so that $U$ has the equal images property.

(2) Let $\ell\!\ge\!1$. It follows from (1) and \cite[(2.8)]{CFS} that $\Rad^s(M)$ enjoys the equal images property for all $s\!\ge\!j$, so that Lemma \ref{EJIM1} implies
\[ \Rad^{j+\ell}(M) = \Rad^\ell(\Rad^j(M))=x^\ell_M(\Rad^j(M)) = x^\ell_M(x^j_M(M))=\im x_M^{j+\ell}\]
for every $x \in \fe_r\!\smallsetminus\!\{0\}$. Hence $M$ has the equal $(j\!+\!\ell)$-images property for all $\ell\! \ge\!1$.

(3) In view of (2), the module $M$ has the equal $(p\!-\!1)$-image property. Given $x \in \fe_r\!\smallsetminus\!\{0\}$, we obtain
\[ x_M(\Rad^{p-1}(M)) = x_M(x_M^{p-1}(M))=(0),\]
so that $\Rad^p(M)=\Rad(U_0(\fe_r))\Rad^{p-1}(M)=(0)$. Consequently, $\ell\ell(M) \le p$.

(4) Let $U:=\Rad^{j-1}(M)$. If $U\!=\!(0)$, then $\ell\ell(M)\!\le\!j\!-\!1$ and there is nothing to be shown. Alternatively, (1) implies that  $U$ is an equal images module of Loewy length $\ell':=\ell\ell(U)=\ell\ell(M)\!-\!j\!+\!1$. 
Thus, \cite[(5.2)]{CFS} provides $b_1, \ldots, b_{\ell'} \in \NN$ such that $\Jt(U,x) = \bigoplus_{i=1}^{\ell'}b_i[i]$ for every $x \in \fe_r\!\smallsetminus\!\{0\}$.

Let $x \in \fe_r\!\smallsetminus\!\{0\}$ and write $\Jt(M,x)\! =\! \bigoplus_{i=1}^pa_i(x)[i]$. Since $\Rad(U)\! =\!\Rad^j(M)\!=\!\im x^j_M$, it follows that
\[ \Jt(\Rad(U),x) = \bigoplus_{i=j+1}^p a_i(x)[i\!-\!j] \ \ \ \text{as well as} \ \  \ \Jt(\Rad(U),x) = \bigoplus_{i=2}^{\ell'} b_i[i\!-\!1].\]
We thus obtain
\[ a_i(x) = b_{i-j+1}\]
for $i \in \{j\!+\!1, \ldots, \ell\ell(M)\}$, as desired.  \end{proof}

\bigskip

\begin{Remark} In the situation above, we have
\[ b_1 = \dim_kU\!-\!2\rk^1(U)\!+\!\rk^2(U) = \dim_kU\!-\!2\rk^j(M)\!+\!\rk^{j+1}(M) = a_j(x)\!+\!\dim_kU\!-\!\dim_k x^{j-1}_M\]
for every $x \in \fe_r\!\smallsetminus\!\{0\}$. Thus, if $M$ has constant $(j\!-\!1)$-rank, then $a_j(x)$ does not depend on the choice of $x$. This always happens for $j\!=\!1$, where we have
$U\!=\!M\!=\!\im x^{j-1}_M$ and hence $a_1(x)\!=\!b_1\!\ne\!0$. \end{Remark}

\bigskip
\noindent
General theory ensures that the morphism
\[ \msim^j_M : \PP(\fe_r) \lra \Gr_{\rk^j(M)}(M)\]
associated to a $U_0(\fe_r)$-module $M$ of constant $j$-rank is either constant or finite. The following result provides more detailed information concerning the size of the fibers.

\bigskip

\begin{Lem} \label{EJIM3} Let $M$ be a $U_0(\fe_2)$-module of constant $j$-rank. If the morphism $\msim_M^j$ affords a fiber with at least $j\!+\!1$ elements,
then $\msim^j_M$ is constant and $M \in \EIP^j(\fe_2)$.\end{Lem}

\begin{proof} By assumption, there is $[x] \in \PP(\fe_2)$ such that $|(\msim_M^j)^{-1}(\msim^j_M([x]))|\!\ge\! j\!+\!1$. Hence there is $[y]\ne [x] \in \PP(\fe_2)$ such that $V\!:=\! \msim^j_M([x])\! =\!
\msim^j_M([y])$. Thus, $\fe_2\!=\!kx\!\oplus\!ky$, and there are $\alpha_i,\beta_i \in k\!\smallsetminus\!\{0\}$ for $2\!\le\! i\!\le\! j$ such that $\msim^j_M([\alpha_ix\!+\!\beta_iy])\!=\!V$ and $|\{[\alpha_ix\!+\!\beta_iy] \ ; \ 2\!\le\!i\!
\le\!j\}\cup\!\{[x],[y]\}|\!=\!j\!+\!1$. Without loss of generality, we may assume that $\beta_i\!=\!1$ for all $i \in \{2,\ldots,j\}$. Setting $\alpha_1\!:=\!0$ and $a_{i\ell}\!:=\! \binom{j}{\ell} \alpha_i^\ell$ for $0\!\le\!\ell\!\le\!j$ and 
$1\!\le \! i\! \le\! j$, we have, observing $\alpha_1^0\!=\!1$,
\[ (\alpha_ix\!+\!y)^j= \sum_{\ell=0}^j a_{i\ell}x^\ell y^{j-\ell} \ \ \ \ \ \ \ \ \text{for} \ 1\!\le\! i\!\le\! j.\]
Setting $z_i:=(\alpha_ix\!+\!y)^j\!-\!\alpha_i^jx^j$, we have $\im (z_i)_M \subseteq V$ as well as
\[ z_i = \sum_{\ell=0}^{j-1} a_{i\ell} x^\ell y^{j-\ell} \ \ \ \ \ \ \ \ \text{for} \ 1\!\le\! i\!\le\! j.\]
Let ${\rm Vd}(\alpha_1,\ldots,\alpha_j)$ be the Vandermonde matrix of $(\alpha_1,\ldots,\alpha_j) \in k^j$. Since
\[ \det((a_{i\ell})) = [\prod_{\ell=0}^{j-1}\binom{j}{\ell}] \det((\alpha_i^\ell)) = [ \prod_{\ell=0}^{j-1}\binom{j}{\ell}] \det({\rm Vd}(\alpha_1,\ldots,\alpha_j)) \ne 0,\]
it follows that $x^\ell y^{j-\ell} \in \sum_{i=1}^j kz_i$ for $0\!\le\! \ell\! \le\! j\!-\!1$. Consequently, $\im (x^\ell y^{j-\ell})_M \subseteq V$ for $0\!\le\! \ell\! \le\! j\!-\!1$.

For $(\gamma,\delta) \in k^2\!\smallsetminus\!\{(0,0)\}$ we therefore obtain
\[ \im(\gamma x\!+\!\delta y)^j_M \subseteq \sum_{\ell =0}^j \im(x^\ell y^{j-\ell})_M \subseteq V.\]
Since $M$ has constant $j$-rank, we have equality. As a result, the morphism $\msim^j_M$ is constant, so that $M \in \EIP^j(\fe_2)$.  \end{proof}

\bigskip

\begin{Remark} When combined with Proposition \ref{EJIM2}, the foregoing result implies that the fibers of the morphism $\msim^j_M$ associated to $M \in \CR^j(\fe_2)$ of Loewy
length $\ge p\!+\!1$ have at most $j$ elements. \end{Remark}

\bigskip

\subsection{Equal $j$-images modules for non-abelian $p$-trivial Lie algebras}
A restricted Lie algebra $(\fg,[p])$ is referred to as being \textit{$p$-trivial}, provided $[p]\!=\!0$. Engel's theorem tells us that every $p$-trivial Lie algebra is necessarily nilpotent. For our purposes,
the $3$-dimensional Heisenberg algebra $\fh_0:= kx\!\oplus\!ky\!\oplus\!kz$ with trivial $p$-map is the most important example.

\bigskip

\begin{Lem}\label{EITriv1} Let $(\fg,[p])$ be a non-abelian $p$-trivial restricted Lie algebra. Then there is an embedding $\fh_0 \hookrightarrow \fg$ of restricted Lie algebras. \end{Lem}

\begin{proof} Let $(\fg^n)_{n\ge 1}$ be the descending central series of $\fg$. Since $\fg$ is nilpotent and not abelian, there is a natural number $n\! \ge\! 2$ which is maximal subject to $\fg^n \ne
(0)$. Thus, $\fg^n \subseteq C(\fg)$, and we can find $x \in \fg$ and $y \in \fg^{n-1}$ such that $[x,y]=z \in C(\fg)\!\smallsetminus\!\{0\}$. Consequently, $kx\!\oplus\!ky\!\oplus\!kz$ is a
$3$-dimensional $p$-subalgebra of $\fg$ which is isomorphic to the $p$-trivial Heisenberg algebra $\fh_0$.  \end{proof}

\bigskip
\noindent
Let $j \in \{1,\ldots,p\!-\!1\}$. Lower semicontinuity of ranks implies that every $U_0(\fh_0)$-module $M$ gives rise to a dense conical open subset
\[ \cU^j_M := \{ a \in \fh_0 \ ; \ \rk(a^j_M)=\rk^j(M)\}\]
of $\fh_0$.

\bigskip

\begin{Lem} \label{EITriv2} Let $M \in \modd U_0(\fh_0)$. If there exists a conical dense open subset $\cO^j_M \subseteq \fh_0$ such that
\begin{enumerate}
\item[(a)] $z \in \cO^j_M$, and
\item[(b)] $\im a^j_M = \im z^j_M$ for all $a \in \cO^j_M$, \end{enumerate}
then $\rk^j(M)=0$. \end{Lem}

\begin{proof} We refer to modules satisfying (a) and (b) as having the generic equal $j$-images property. Given such a module  $M \in \modd U_0(\fh_0)$, we observe that
$\cO^j_M\cap\cU^j_M\ne \emptyset$, so that
\[ \rk^j(M)=\rk(z^j_M).\]
We proceed by verifying the following statement:

\medskip

($\dagger$) \ \textit{Let $j \in \{2,\ldots,p\!-\!1\}$. The submodule $z.M$ has the generic equal $(j\!-\!1)$-images property and generic $(j\!-\!1)$-rank $\rk^{j-1}(z.M)=\rk^j(M)$.}

\smallskip
\noindent
Since the subset $\fh_0\!\smallsetminus\!\{0\}$ is open and conical, we may assume without loss of generality that $\cO^j_M \subseteq \fh_0\!\smallsetminus\!\{0\}$.

Let $a \in \cO^j_M\!\smallsetminus\!kz$. If $[a,b]=0$ for all $b \in \cO^j_M$, then $\cO^j_M \subseteq C_{\fh_0}(a)$, the centralizer of $a$ in $\fh_0$. As $\cO^j_M$ lies dense in $\fh_0$, it follows
that $C_{\fh_0}(a)=\fh_0$. Hence $a \in C(\fh_0)=kz$, a contradiction. Consequently, there are $b' \in \cO^j_M$ and $\lambda \in k^\times$ such that $[a,b']=\lambda z$. As $\cO^j_M$ is
conical, $b\!:=\!\lambda^{-1}b' \in \cO^j_M$ and $[a,b]\!=\!z$.

Let $V\!:=\!\im a^j_M$. In view of property (b), $V\!=\!\im z^j_M$ is a $U_0(\fh_0)$-submodule of $M$. The Cartan-Weyl identity \cite[(I.1.3)]{SF} yields $ba^j\! =\! a^jb\!-\!ja^{j-1}z$, so that
\[ a^{j-1}z.M \subseteq V = z^j.M.\]
Thus, setting $N\!:=\! z.M$, we obtain
\[ \im a^{j-1}_N \subseteq \im z^{j-1}_N \ \ \ \ \ \ \ \ \ \forall \ a \in \cO^j_M\!\smallsetminus\!kz.\]
Since the latter inclusion also holds for $a \in kz\!\smallsetminus\!\{0\}$, our assumption $\cO^j_M \subseteq \fh_0\!\smallsetminus\!\{0\}$ entails
\[ (\ast) \ \ \ \ \ \ \ \im a^{j-1}_N \subseteq \im z^{j-1}_N \ \ \ \ \ \ \forall \ a \in \cO^j_M.\]
Recall that $\cU^{j-1}_N\! :=\! \{a \in \fh_0 \ ; \ \rk(a^{j-1}_N)\!=\!\rk^{j-1}(N)\}$ is a conical dense open subset $\fh_0$. Hence $\cO^{j-1}_N\!:=\! \cO^j_M\cap\cU^{j-1}_N$ enjoys the same properties. For $a \in
\cO^{j-1}_N$, inclusion ($\ast$) yields
\[ \rk^{j-1}(N) = \rk(a^{j-1}_N) \le \rk(z^{j-1}_N).\]
Consequently, we have equality, so that $z \in \cU^{j-1}_N$. Thus, $z \in \cO^{j-1}_N$ and ($\ast$) now implies
\[ \im a^{j-1}_N=\im z^{j-1}_N \ \ \ \ \ \ \forall \ a \in \cO^{j-1}_N,\]
proving that the $U_0(\fh_0)$-module $N$ has the generic equal $(j\!-\!1)$-images property.

By virtue of our observation above, we also have $\rk^j(M)\! =\! \rk(z^j_M)\!=\!\rk(z^{j-1}_N)\!=\!\rk^{j-1}(N)$, as asserted. \hfill $\diamond$

\medskip
\noindent
We first consider the case, where $j\!=\!1$. As before, there are elements $a,b \in \cO^1_M$ such that $z\!=\![a,b]$. We therefore obtain
\[ z.M \subseteq ab.M\!+\!ba.M = az.M\!+\!bz.M \subseteq za.M\!+\!zb.M \subseteq z^2.M.\]
Since $z$ is nilpotent, we get $z.M\!=\!(0)$. Thus, $\rk^1(M)\!=\!\rk(z_M)\!=\!0$.

Now let $j\!>\!1$. Repeated application of ($\dagger$) implies that $N\!:=\!z^{j-1}M$ has the generic equal $1$-images property, while $\rk^j(M)\!=\!\rk^1(N)\!=\!0$.  \end{proof}

\bigskip
\noindent
Let $\cC \subseteq \modd U_0(\fg)$ be a full subcategory that is closed under direct summands and images of isomorphisms. We denote by $k\langle X,Y\rangle$ the free $k$-algebra with non-commuting
variables $X,Y$. Following \cite[\S 2]{Si}, we say that $\cC$ is {\it wild}, provided there is a functor $F : \modd k\langle X,Y\rangle \lra \cC$ that preserves indecomposables and reflects isomorphisms.
In that case, the problem of classifying all indecomposable objects in $\cC$ is at least as complicated as finding a canonical form of two non-commuting matrices. This latter problem is deemed hopeless.
The algebra $U_0(\fg)$ has {\it wild representation type} if $\modd U_0(\fg)$ is wild.

\bigskip

\begin{Remark} The foregoing Lemma implies in particular that $\EIP^1(\fh_0)$ is just the category of trivial $U_0(\fh_0)$-modules, while $V(\fh_0)$ being $3$-dimensional entails
that the algebra $U_0(\fh_0)$ has wild representation type, cf.\ \cite[(4.1)]{Fa07}. By contrast, Benson \cite[(5.6.12),(5.5)]{Be} has shown that the full subcategory of $\EIP^1(\fe_3)$ whose objects
$M$ have Loewy length $\le 2$ is wild. \end{Remark}

\bigskip
\noindent
The following result, which generalizes \cite[(3.1.2)]{Fa17}, implies that, for many non-abelian restricted Lie algebras, the category $\EIP^j(\fg)$ is in fact the module category of a factor
algebra of $U_0(\fg)$, see Section \ref{S:cat} below. In particular, $\EIP^j(\fg)$ is closed under taking subobjects.

\bigskip

\begin{Prop} \label{EITriv3} Let $(\fg,[p])$ be a restricted Lie algebra, $M \in \EIP^j(\fg)$ be a module with the equal $j$-images property. If $\fg$ contains a non-abelian $p$-trivial
subalgebra, then $\rk^j(M)=0$. \end{Prop}

\begin{proof} Lemma \ref{EITriv1} shows that $\fh_0 \subseteq \fg$. By assumption, the restriction $M|_{\fh_0}$ has the generic equal $j$-images property. Hence Lemma \ref{EITriv2} yields
\[ \rk^j(M) = \rk^j(M|_{\fh_0})=0,\]
as desired. \end{proof}

\bigskip

\subsection{The restriction functor $\res : \modd U_0(\fg) \lra \modd U_0(\fe)$}
As $p\!\ge\!3$, work by Bissinger \cite[(4.2.3)]{Bi} readily implies that the category $\EIP^1(\fe_2)$ is wild. In view of Proposition \ref{EJIM2}(2), this also holds for $\EIP^j(\fe_2)$ for $j\!\ge\!2$.
The succeeding result shows that for large classes of Lie algebras these modules will rarely be restrictions of modules of constant $j$-rank.

\bigskip

\begin{Prop} \label{Res1} Let $(\fg,[p])$ be a restricted Lie algebra such that
\begin{enumerate}
\item[(a)] $\fg$ is algebraic, or $p\!\ge\!5$ and $\dim V(C(\fg))\ne 1$, and
\item[(b)] $\fg$ possesses a non-abelian $p$-trivial subalgebra. \end{enumerate}
Let $M \in \CR^j(\fg)$ be a module of constant $j$-rank. If there exists $\fe_0 \in  \EE(2,\fg)$ such that $M|_{\fe_0} \in \EIP^j(\fe_0)$, then $\rk^\ell(M)=0$ for $\ell \in \{j,\ldots,p\!-\!1\}$. \end{Prop}

\begin{proof} In view of (a), Theorem \ref{CBC2} ensures that the $j$-degree function $\msdeg^j_M$ is constant. Let $\fu \subseteq \fg$ be a non-abelian, $p$-trivial subalgebra. Then
$\dim_k\fu\! \ge \!3$ and $V(C(\fu))\!\ne\!\{0\}$, so that $\msrk_p(\fu)\!\ge\! 2$. If $\fe \in \EE(2,\fu)$, then $\deg^j(M|_\fe)\!=\!\deg^j(M|_{\fe_0})\!=\!0$. Thanks to \cite[(4.1.2)]{Fa17}, we have $\deg^j(M|_\fu)\!=\!0$,
whence $M|_\fu \in \EIP^j(\fu)$. Proposition \ref{EITriv3} thus yields $0\!=\!\rk^j(M|_\fu)\!=\!\rk^j(M)$, so that $\rk^\ell(M)\!=\!0$ for $\ell \in \{j,\ldots,p\!-\!1\}$. \end{proof}

\bigskip

\subsection{The categories $\modd^jU_0(\fg)$}\label{S:cat}
Proposition \ref{EITriv3} motivates the study of certain subcategories of $\EIP^j(\fg)$. Given $j \in \{1,\ldots,p\!-\!1\}$, we let $\modd^jU_0(\fg)$ be the full subcategory of $\modd U_0(\fg)$, whose objects
satisfy $x^j_M=0$ for all $x \in V(\fg)$. Thus, $\modd^jU_0(\fg) \cong \modd U_0(\fg)/I^j(\fg)$, where $I^j(\fg) := \sum_{x \in V(\fg)} U_0(\fg)x^jU_0(\fg)$ is the ideal generated by the $j$-powers of the
elements of $V(\fg)$. We put
\[U_0^j(\fg):= U_0(\fg)/I^j(\fg),\]
so that $\modd^jU_0(\fg)=\modd U^j_0(\fg)$. Observe that
\[ \modd^1U_0(\fg) \subseteq \modd^2U_0(\fg) \subseteq \cdots \subseteq \modd^{p-1}U_0(\fg) \subseteq \modd U_0(\fg)\]
provides a filtration of $\modd U_0(\fg)$.

Given a $U_0(\fg)$-module $M$, we denote by $\add(M)$ the full subcategory of $\modd U_0(\fg)$, whose objects are direct sums of direct summands of $M$.

\bigskip

\begin{Remarks} The category $\modd^jU_0(\fg)$ exhibits the following properties:
\begin{enumerate}
\item $\modd^jU_0(\fg)$ is closed under taking submodules, images, and duals.
\item Every $M \in \modd^jU_0(\fg)$ is {\it projective-free}, that is, $M$ contains no nonzero $U_0(\fg)$-projective submodules.
\item If $i,j \in \{1,\ldots,p\!-\!1\}$, then $M\!\otimes_k\!N \in \modd^{i+j-1}U_0(\fg)$ for all $M \in \modd^iU_0(\fg)$ and $N \in \modd^jU_0(\fg)$. (Here we set $\modd^jU_0(\fg)=\modd U_0(\fg)$ for
$j\!\ge\!p$.)
\item If $\fg\!=\!\Lie(G)$ is algebraic, then $I^j(\fg)$ is $G$-stable, and $M^{(g)} \in \modd^jU_0(\fg)$ for all $M \in \modd^jU_0(\fg)$ and $g \in G$ . \end{enumerate} \end{Remarks}

\bigskip

\begin{Examples} Let $j \in \{1,\ldots, p\!-\!1\}$.
\begin{enumerate}
\item If $\fg\! =\! \fe_r$ is elementary abelian, then \cite[(1.17.1)]{Be} yields $I^j(\fg)\!=\!\Rad^j(U_0(\fg))$, so that $\modd^jU_0(\fe_r)$ is the category of modules of Loewy length $\le\! j$. If $r\!\ge\!3$ and $j\!\ge\!2$, 
or $r\!=\!2$ and $j\!\ge\!3$, this category is known to be wild, cf.\ \cite[(I.10.10)]{Er}.
\item We consider $\fg\!:=\!\fsl(2)$ together with its standard basis $\{e,f,h\}$. It is well-known (cf.\ \cite{Se}) that the radical $\Rad(U_0(\fsl(2)))$ is generated by $\{e^{p-1}(h\!+\!1),
(h\!+\!1)f^{p-1}\}$. Accordingly, $\Rad(U_0(\fsl(2)))$ $ \subseteq I^j(\fsl(2))$, so that $\modd^j U_0(\fsl(2))$ is semisimple. Let $L(i)$ be the simple $U_0(\fsl(2))$-module of dimension $i\!+\!1$ ($i \in
\{0,\ldots,p\!-\!1\}$) and with highest weight $i \in \FF_p$. It follows that $\modd^jU_0(\fsl(2))\!=\!\add(\bigoplus_{i=0}^{j-1}L(i))$. \end{enumerate} \end{Examples}

\bigskip
\noindent
Throughout this section, we will be considering a restricted Lie algebra $(\fg,[p])$ with triangular decomposition
\[ \fg = \fg^-\!\oplus\!\fg_0\!\oplus\!\fg^+.\]
By definition, $\fg_0$ and $\fg^\pm$ are $p$-subalgebras such that
\begin{enumerate}
\item[(i)] $\fg_0$ is a torus, and
\item[(ii)] $\fg^\pm$ is unipotent and such that $\fg^\pm = \langle V(\fg^\pm) \rangle$, and
\item[(iii)] $[\fg_0,\fg^\pm] \subseteq \fg^\pm$. \end{enumerate}
Reductive Lie algebras and Lie algebras of Cartan type are known to afford such decompositions.

Let $X(\fg)$ be the set of algebra homomorphism $U_0(\fg) \lra k$, the so-called {\it character group} of $\fg$. Each character $\lambda \in X(\fg)$ defines a one-dimensional $U_0(\fg)$-module
$k_\lambda$. If $M$ is a $U_0(\fg)$-module, then its {\it annihilator} $\ann_\fg(M):=\{x \in \fg \ ; \ x_M = 0\}$ is a $p$-ideal of $\fg$.

\bigskip

\begin{Lem} \label{Triang1} The following statements hold:
\begin{enumerate}
\item If $S \in \modd^jU_0(\fg)$ is simple, then $\dim_kS \le \min\{j^{\dim_k\fg^-}, j^{\dim_k\fg^+}\}$.
\item We have $\modd^1U_0(\fg) = \add(\bigoplus_{\lambda \in X(\fg)}k_\lambda)$. \end{enumerate} \end{Lem}

\begin{proof} (1) We put $\fb^\pm := \fg_0\!\ltimes\!\fg^{\pm}$, so that (i) and (ii) imply that $\fb^\pm$ is a trigonalizable $p$-subalgebra of $\fg$. Hence there are vectors $v^+, v^- \in S$ such that
$S = U_0(\fg^-)v^+$ and $S = U_0(\fg^+)v^-$.

Note that $I^j(\fg^-)\subseteq I^j(\fg)$, so that $I^j(\fg^-)v^+=(0)$. In view of (ii), there exists a basis $\{x_1, \ldots, x_n\}$ of $\fg^-$ such that $\{x_1,\ldots, x_n\} \subseteq V(\fg^-)$.
We put $\tau_j:= (j\!-\!1,\ldots,j\!-\!1) \in \NN^n_0$ and use the standard multi-index conventions for computations in $U_0(\fg)$. If $a \in \NN^n_0$ is such that $a \not\le \tau_j$, then $x^a \in I^j(\fg^-)$,
whence
\[ S = \sum_{a \le \tau_j} kx^av^+.\]
This shows that $\dim_kS \le j^n = j^{\dim_k\fg^-}$. By the same token, we have $\dim_kS \le j^{\dim_k\fg^+}$.

(2) In view of (1), every simple object of $\modd^1U_0(\fg)$ is one-dimensional, so that the simple modules of $\modd^1U_0(\fg)$ are the $k_\lambda$ with $\lambda \in X(\fg)$.
If $M \in \modd^1U_0(\fg)$, then $V(\fg) \subseteq \ann_\fg(M)$, and condition (ii) yields $\fg^+\!\oplus\!\fg^- \subseteq \ann_\fg(M)$. Consequently, the canonical map $\fg_0 \lra \fg/\ann_\fg(M)$
is surjective, so that (i) implies that $\fg/\ann_\fg(M)$ is a torus. As a result, $M$ is semisimple, whence $M \in \add(\bigoplus_{\lambda \in X(\fg)}k_\lambda)$. \end{proof}

\bigskip

\begin{Example} Let $\fg:= \fsl(2)\!\oplus\!\fsl(2)$. Then $U_0(\fg) \cong U_0(\fsl(2))\!\otimes_k\!U_0(\fsl(2))$ and, setting $J:=\Rad(U_0(\fsl(2)))$, we have $\Rad(U_0(\fg))=J\!\otimes_k\!U_0(\fsl(2))\!+\!
U_0(\fsl(2))\!\otimes_k\!J$. Now let $M \in \modd^jU_0(\fg)$. Since the restriction $M|_\fh$ of $M$ to a direct summand $\fh\cong \fsl(2)$ of $\fg$ is semisimple, we conclude that $(J\!\otimes_k\!k1).M=
(0)= (k1\!\otimes_k\!J).M$. Consequently, $\Rad(U_0(\fg)).M=(0)$, so that $M$ is semisimple. As a result, $M \cong \bigoplus_{a+b\le j-1}n_{(a,b)}L(a)\!\otimes_k\!L(b)$. In particular, the
$j^2$-dimensional simple $U_0(\fg)$-module $L(j\!-\!1)\!\otimes_k\!L(j\!-\!1)$ does not belong to $\modd^jU_0(\fg)$ if $j\!\ge\!2$. Hence the converse of Lemma \ref{Triang1}(1) does not hold.
\end{Example}

\bigskip
\noindent
For an arbitrary natural number $n$, we denote by $W(n)\!:=\!\Der_k(k[X_1,\ldots,X_n]/(X_1^p,\ldots,X_n^p))$ the {\it Jacobson-Witt algebra} of dimension $np^n$. Note that there are natural embeddings $W(n') 
\hookrightarrow W(n)$, whenever $n'\!\le\!n$. 

There are four families of restricted simple Cartan type Lie algebras. Beside the family $(W(n))_{n\ge 1}$ of Jacobson-Witt algebras defined above, there are the {\it special algebras} $(S(n))_{n \ge 3}$, the 
{\it Hamiltonian algebras} $(H(2r))_{r\ge1}$, and the {\it contact algebras} $(K(2r\!+\!1))_{r \ge1}$. We refer the reader to \cite[(IV)]{SF} for more details on Lie algebras of Cartan type.

We consider the Witt algebra $W(1)\!=\!\bigoplus_{i=-1}^{p-2}ke_i\! =\!\Der_k(k[X]/(X^p))$, where $e_i(x^\ell)=\!\ell x^{i+\ell}$ for $x:\!=\!X\!+\!(X^p) \in k[X]/(X^p)$. This $\ZZ$-graded restricted Lie algebra has a
standard triangular decomposition such that $W(1)^-= W(1)_{-1}:=ke_{-1}$, $W(1)_0=ke_0$ and $W(1)^+:=\sum_{i=1}^{p-2}ke_i$ (see \cite[(IV)]{SF} for more details).

The simple $U_0(W(1))$-modules are well-known and were first determined by Chang \cite{Chang}. We denote by $L(\lambda)$ the simple $U_0(W(1))$-module such that $e_0$ acts on $L(\lambda)^{W(1)^+}$ via 
$\lambda \in \{0,\ldots,p\!-\!1\}$. We have $L(0)\cong k$ and $L(p\!-\!1) \cong k[x]/k1$, while $\dim_kL(\lambda)\!=\!p$ for $\lambda \in \{1,\ldots,p\!-\!2\}$. Consequently, Lemma \ref{Triang1} implies that $L(0)\!=\!k$ is 
the only simple module belonging to $\modd^jU_0(W(1))$ for $j\!\le\!p\!-\!2$. By the same token, $L(0)$ and $L(p\!-\!1)$ are the only simple objects of $\modd^{p-1}U_0(W(1))$.

\bigskip

\begin{Lem}\label{Triang2}The following statements hold:
\begin{enumerate}
\item We have $\modd^jU_0(W(1)))\! =\! \add(k)$ for $j \in \{1,\ldots, p\!-\!2\}$.
\item We have $\modd^{p-1}U_0(W(1))\!=\!\add(k\!\oplus\!L(p\!-\!1))$.
\item Suppose that $p\!\ge\!5$. If $M \in \modd^{p-1}U_0(W(1))$ has constant $j$-rank for some $j \in \{1,\ldots,p\!-\!2\}$, then $M \in \add(k)$.\end{enumerate} \end{Lem}

\begin{proof} If $p\!=\!3$, then $W(1)\cong \fsl(2)$, and assertions (1) and (2) follow from the examples above. Hence we assume that $p\!\ge\!5$.

(1) Let $j\! \le\! p\!-\!2$. In view of the above, the trivial module $k$ is the only simple $U_0^j(W(1))$-module, while \cite[(2.1)]{Ho} yields $\Ext^1_{U_0(W(1))}(k,k) \subseteq
(W(1)/[W(1),W(1)])^\ast = (0)$. As a result, $\modd^jU_0(\fg) = \add(k)$.

(2) As observed above, $k$ and $L(p\!-\!1)$ are the only simple objects in $\modd^{p-1}U_0(W(1))$. For the relevant $\Ext^1$-groups we have
\[ \Ext^1_{U_0(W(1))}(L(p\!-\!1),L(p\!-\!1)) =(0),\]
as well as
\[\dim_k\Ext^1_{U_0(W(1))}(k,L(p\!-\!1))=2= \dim_k\Ext^1_{U_0(W(1)}(L(p\!-\!1),k),\]
see \cite[(3.5)]{BNW} or \cite[(3.5),(3.6)]{Ri} for more details. Recall that $\Ext^1_{U_0(W(1)}(k,L(p\!-\!1))$ is given by equivalence classes of extensions
\[ (0) \lra L(p\!-\!1) \lra M \lra k \lra (0).\]
By general theory, the middle term $M\! =\!M_{\varphi}\!=\!L(p\!-\!1)\!\oplus\!k$ corresponds to a derivation $\varphi : U_0(W(1)) \lra L(p\!-\!1)$ of the augmented algebra $U_0(W(1))$.
Note that $\varphi$ is completely determined by its restriction $W(1)\lra L(p\!-\!1)$, which is also a derivation. The action of $W(1)$ on $M_\varphi$ is given by
\[ a.(v,\alpha) = (a.v\!+\!\alpha\varphi(a),0)  \ \ \ \ \ \ \ \ \forall \ a \in W(1), v \in L(p\!-\!1), \alpha \in k.\]
Thus, $M_\varphi \in \modd^{p-1}U_0(W(1))$ if and only if
\[ a^{p-2}.\varphi(a) = 0 \ \ \ \ \ \ \text{for all} \ a \in V(W(1)).\]
Suppose that $M_\varphi$ is a non-split extension of $k$ by $L(p\!-\!1)$. Then $\varphi: W(1) \lra L(p\!-\!1)$ is not an inner derivation, and \cite[(1.2)]{Fa88} implies that we may
assume that $\varphi(e_i) \in L(p\!-\!1)_i \ \ (i \in \{-1,\ldots,p\!-\!2\})$, where $L(p\!-\!1)_i$ is the weight space of $L(p\!-\!1)$ with weight $i \in \FF_p$ relative to the standard torus
$ke_0$ of $W(1)$. We let $\Der_k(W(1),L(p\!-\!1))_0$ be the space of these derivations. Recall that $L(p\!-\!1) \cong k[x]/k1 = \bigoplus_{i=1}^{p-1}k\bar{x}^i$. Since the Lie algebra $W(1)$
is generated by $\{e_{-1},e_2\}$, the map
\[ \Der_k(W(1),L(p\!-\!1))_0 \lra k\bar{x}^{p-1}\!\oplus\!k\bar{x}^2 \ \ ; \ \ \varphi \mapsto (\varphi(e_{-1}),\varphi(e_2))\]
is injective.

Suppose that $M_\varphi \in \modd^{p-1}U_0(W(1))$. Then $\varphi(e_{-1}) = \alpha \bar{x}^{p-1}$, while
\[ 0 = e^{p-2}_{-1}.\varphi(e_{-1}) = \alpha e^{p-2}_{-1}\bar{x}^{p-1}=(p\!-\!1)! \alpha \bar{x}\]
forces $\alpha\!=\!0$. Thus, $\varphi(e_{-1})\!=\!0$.

Recall that $ke_{-1} \oplus ke_0 \oplus ke_1$ is a $p$-subalgebra of $W(1)$ that is isomorphic to $\fsl(2)$. Since the category $\modd^{p-1}U_0(\fsl(2))$ is semisimple, the exact sequence
\[ (0) \lra L(p\!-\!1)|_{\fsl(2)} \lra (M_\varphi)|_{\fsl(2)} \lra k \lra (0)\]
splits, so that $\varphi|_{\fsl(2)}$ is an inner derivation. As $L(p\!-\!1)_0\!=\!(0)$, it follows that $\varphi|_{\fsl(2)}\! =\! 0$. In particular, $\varphi(e_1)\!=\!0$. We thus obtain
\[ 0 = \varphi(e_1) = \frac{1}{3}\varphi([e_{-1},e_2]) = \frac{1}{3} e_{-1}.\varphi(e_2).\]
Consequently,
\[ \varphi(e_2) \in L(p\!-\!1)_2\cap L(p\!-\!1)^{ke_{-1}} = L(p\!-\!1)_2\cap L(p\!-\!1)_1 = (0),\]
so that $\varphi=0$. As a result, $\Ext^1_{U^{p-1}_0(W(1))}(k,L(p\!-\!1))=(0)$, while duality implies the vanishing of $\Ext^1_{U^{p-1}_0(W(1))}(L(p\!-\!1),k)$. It follows that
$\modd^{p-1}U_0(W(1))\!=\!\add(k\!\oplus\!L(p\!-\!1))$.

(3) Since $\rk((e^j_{p-2})_{L(p\!-\!1)})\!=\!\delta_{j,1}$ and $\rk((e^j_{-1})_{L(p\!-\!1)})\!=\!p\!-\!1\!-\!j$, it follows that $L(p\!-\!1)$ does not have constant $j$-rank for $j \in \{1,\ldots,p\!-\!2\}$
whenever $p\!\ge\!5$. As the category of modules of constant $j$-rank is closed under taking direct summands (cf.\ \cite[(4.13)]{CFP13}), part (2) implies the assertion. \end{proof}

\bigskip
\noindent
Let $B_2\!:=\! k[X_1,X_2]/(X_1^p,X_2^p)$ and put $x_i\!:=\!X_i\!+\!(X_i^p)$. The Jacobson-Witt algebra $W(2)\!=\!\Der_kB_2$ of derivations of $B_2$ is simple and of dimension $2p^2$.
We let $\cP(2)$ be the {\it Poisson algebra}, whose underlying vector space is $B_2$. There is a homomorphism $D_H : \cP(2) \lra W(2)$ of restricted Lie algebras such that
$\ker D_H\! =\! k.1\! =\! C(\cP(2))$, cf.\ \cite[\S 4]{BFS}. Then $\hat{\cP}(2)\!:=\!\bigoplus_{0\le i+j\le 2p-3} kx_1^ix_2^j\!=\![\cP(2),\cP(2)]$ is a $p$-subalgebra of $\cP(2)$ containing $1$
and $H(2)\!:=\! D_H(\hat{\cP}(2))$ is a simple restricted Lie algebra of hamiltonian type, cf.\ \cite[Ex.4, Ex.5,p.169]{SF}.

\bigskip

\begin{Lem} \label{Triang3} Suppose that $p\!\ge\!5$. Then the following statements hold:
\begin{enumerate}
\item We have $\modd^jU_0(H(2))\!=\!\add(k)$ for $j \in \{1,\ldots,p\!-\!1\}$.
\item We have $\modd^jU_0(\hat{\cP}(2))\!=\!\add(k)$ for $j \in \{1,\ldots,p\!-\!1\}$. \end{enumerate} \end{Lem}

\begin{proof} (1) Let $S \in \modd^{p-1}U_0(H(2))$ be a non-trivial simple $U_0(H(2))$-module. Thanks to \cite[(2.6)]{HS2}, there is an embedding $W(1) \hookrightarrow H(2)$ such that $L(1)$ is a composition
factor of $S|_{W(1)}$. As $\modd^{p-1}U_0(W(1))$ is closed under submodules and images, it follows that $L(1) \in \modd^{p-1}U_0(W(1))$, which contradicts Lemma \ref{Triang2}(2). In view of 
$\Ext^1_{U_0(H(2))}(k,k) \hookrightarrow (H(2)/[H(2),H(2)])^\ast \!=\!(0)$, we conclude that $\modd^{p-1}U_0(H(2))\!=\!\add(k)$. Consequently, $\modd^jU_0(H(2))\!=\!\add(k)$ for every $j \in \{1,\ldots,p\!-\!1\}$.

(2) Note that $\fh:=kx_1\!\oplus\!kx_2\!\oplus\!k1 \subseteq \hat{\cP}(2)$ is a $p$-subalgebra, which is isomorphic to the Heisenberg algebra with toral center. Let $S \in \modd^{p-1}U_0(\fh)$ be a simple
module. The restriction $S|_\fa$ with respect to the abelian $p$-subalgebra $\fa\!:=\! kx_2\!\oplus\!k1$ possesses a one-dimensional submodule $k_\lambda$, so that there exists a surjection
\[ U_0(\fh)\!\otimes_{U_0(\fa)}\!k_\lambda \lra S.\]
The central element $1$ acts on $S$ via a scalar $\alpha \in k$, while $1=[x_1,x_2]$ implies $(\dim_kS) \alpha\! =\!\tr(1)\!=\!0 \in k$. If $\alpha\! \ne\! 0$, then $p\!\mid\! \dim_kS$, so that
\[ U_0(\fh)\!\otimes_{U_0(\fa)}\!k_\lambda \cong S.\]
This readily yields $(x_1)_S^{p-1}\!\ne\!0$, a contradiction. As a result, $1$ acts trivially on $S$.

Now let $M \in \modd^{p-1}U_0(\fh)$. Then every composition factor $S$ of $M$ belongs to $\modd^{p-1}U_0(\fh)$, so that $1$ acts nilpotently on $M$. As $k1$ is a torus, we conclude
that $1$ annihilates $M$.

Recall that there exists an exact sequence
\[ (0) \lra k1 \lra \hat{\cP}(2) \lra H(2) \lra (0)\]
of restricted Lie algebras and let $S \in \modd^{p-1}U_0(\hat{\cP}(2))$ be simple. By applying the observations above to $S|_\fh \in \modd^{p-1}U_0(\fh)$, we obtain $k1.S\!=\!(0)$. Hence $S \in \modd^{p-1}
U_0(H(2))\!=\!\add(k)$ is simple. It follows that $S\!=\!k$ is the trivial $U_0(\hat{\cP}(2))$-module. Since $k1 \subseteq [\hat{\cP}(2),\hat{\cP}(2)]$, it follows that $\hat{\cP}(2)$ is perfect, so that
$\Ext^1_{U_0(\hat{\cP}(2))}(k,k)\hookrightarrow (\hat{\cP}(2)/[\hat{\cP}(2),\hat{\cP}(2)])^\ast\!=\!(0)$. As before, we obtain $\modd^jU_0(\hat{\cP}(2))\!=\!\add(k)$ for $j \in \{1,\ldots,p\!-\!1\}$. \end{proof}

\bigskip
\noindent
Given a restricted Lie algebra $(\fg,[p])$, we let $\mu(\fg)$ be the maximal dimension of all tori $\ft \subseteq \fg$. The minimal dimension of all Cartan subalgebras of $\fg$ is referred to as the
{\it rank} $\rk(\fg)$. If $\fg$ possesses a self-centralizing torus, then $\mu(\fg)$ coincides with the rank $\rk(\fg)$ of $\fg$, cf.\ \cite[(3.5),(3.6)]{Fa04}. 

\bigskip

\begin{Lem} \label{Triang4} Suppose that $p\!\ge\!5$. Let $\fg$ be a restricted simple Lie algebra of Cartan type.
\begin{enumerate}
\item If $\fg$ is not of contact type, then there exists an embedding
\[ W(\mu(\fg)) \hookrightarrow \fg\]
of restricted Lie algebras.
\item If $\fg$ is of contact type, then there exists an embedding
\[ W(\mu(\fg)\!-\!1) \hookrightarrow \fg\]
of restricted Lie algebras. \end{enumerate} \end{Lem}

\begin{proof} (1) The statement is clear for $\fg\!=\!W(n)$, where $n\!=\!\mu(\fg)$. If $\fg\!=\!S(n)$ is a special Lie algebra, then $\mu(\fg)\!=\!n\!-\!1$ (cf.\ \cite{De1}) and \cite[(4.1)]{BFS}
provides an embedding $W(n\!-\!1) \hookrightarrow \fg$. Let $\fg\!=\!H(2r)$ be an algebra of Hamiltonian type. In view of \cite{De2}, we have $\mu(H(2r))\!=\!r$ and by \cite[(4.3)]{BFS} there
is an embedding $W(r) \hookrightarrow H(2r)$.

It remains to consider the Melikian algebra $\cM$ for $p\!=\!5$. Then $\mu(\cM)\!=\!2$ and \cite[\S 1]{Sk} provides an embedding $W(2) \hookrightarrow \fg$.

(2) Let $\fg\! =\! K(2r\!+\!1)$, so that \cite{De2} yields $\mu(\fg)=r\!+\!1$. Owing to \cite[(4.2)]{BFS} and \cite[(7.5.15)]{Str}, there are embeddings $W(r) \hookrightarrow \cP(2r)$ and
$\cP(2r) \hookrightarrow K(2r\!+\!1)$. Here $\cP(2r)$ denotes the Poisson algebra in $2r$ variables (see for instance \cite[\S4]{BFS}). There results an embedding
$W(\mu(\fg)\!-\!1) \hookrightarrow \fg$. \end{proof}

\bigskip

\begin{Prop} \label{Triang5}Let $(\fg,[p])$ be a restricted simple Lie algebra. Then the following statements hold:
\begin{enumerate}
\item Suppose that $\fg$ is of Cartan type.
\begin{enumerate}
\item We have $\modd^jU_0(\fg)\!=\!\add(k)$ for $j \in \{1,\ldots,p\!-\!2\}$.
\item If $\fg \not \cong W(1)$, then $\modd^{p-1}U_0(\fg)=\add(k)$. \end{enumerate}
\item Suppose that $p\!\ge\!5$ and let  $\fg$ be classical.
\begin{enumerate}
\item If $\rk(\fg)\!\ge\!p\!-\!1$, then $\modd^jU_0(\fg)=\add(k)$ for $j \in \{1,\ldots,p\!-\!2\}$.
\item If $\rk(\fg)\!\ge\!p^2\!-\!2$, then $\modd^{p-1}U_0(\fg)=\add(k)$. \end{enumerate} \end{enumerate} \end{Prop}

\begin{proof} Let $M \in \modd^jU_0(\fg)$.

(1) (a) Suppose that $\fg$ is simple of Cartan type. Then $\mu(\fg)\!\ge\! 1$ and Lemma \ref{Triang4} provides an embedding $W(1) \hookrightarrow \fg$ of restricted Lie algebras. Since the restriction
$M|_{W(1)}$ belongs to $\modd^jU_0(W(1))$, Lemma \ref{Triang2} implies $W(1) \subseteq \ann_\fg(M)$, so that the simplicity of $\fg$ forces $\ann_\fg(M)=\fg$. As a result, $M \in \add(k)$.

(b) If $\fg \not \cong W(1),K(3)$, then either $\mu(\fg) \ge 2$ or $\fg=H(2)$. In the former case, Lemma \ref{Triang4} shows that $W(2)$ is contained in $\fg$, so that there is an embedding $H(2) \hookrightarrow \fg$. As this also holds in the latter case, the foregoing arguments in conjunction with Lemma \ref{Triang3} now yield the assertion.

It remains to consider the case, where $\fg\cong K(3)$. We recall that there is an embedding $\hat{\cP}(2) \hookrightarrow K(3)$ of restricted Lie algebras. Lemma \ref{Triang3} ensures that
$\modd^{p-1}U_0(\hat{\cP}(2)) = \add(k)$, so that the simplicity of $K(3)$ yields the desired conclusion.

(2) (a) Chang's classification \cite{Chang} shows that $U_0(W(1))$ possesses exactly $p\!-\!2$ simple modules of dimension $p$. Hence one of these modules, $M$ say, has to be self-dual, which
ensures that it affords a non-degenerate $W(1)$-invariant bilinear form. Schur's Lemma implies that this form is either symmetric or symplectic and $\dim_kM$ being odd rules out the latter
alternative. Hence the representation afforded by $M$ provides an embedding $W(1) \hookrightarrow \fso(p)$. This fact was first observed by Herpel and Stewart, cf.\ \cite[(11.9)]{HS}.

Suppose first that $p\!\ge\!7$. Since $\fg$ is classical simple of rank $\rk(\fg)\!\ge\!p\!-\!1\!\ge\!6$, the Dynkin diagram of $\fg$ contains a copy of $A_{p-2}$. Hence there is an embedding $\fsl(p\!-\!1)
\hookrightarrow \fg$ of restricted Lie algebras ($\fsl(p\!-\!1)$ being the derived algebra of the Levi subalgebra corresponding to the subsystem generated by $A_{p-2}$). Hence the representation
$\varrho : W(1) \hookrightarrow \fgl(p\!-\!1)$ afforded by the $(p\!-\!1)$-dimensional simple $U_0(W(1))$-module $L(p\!-\!1)$ factors through $\fsl(p\!-\!1)$, so that there is an embedding $W(1)
\hookrightarrow \fg$. The arguments of (1) now yield the result.

If $p\!=\!5$, the foregoing argument works unless $\fg$ is of type $F_4$. In that case, $\fg$ contains the Lie algebra $\fso(5)$ of type $B_2=C_2$. The observations above provide an embedding $W(1)
\hookrightarrow \fg$, and we may argue as before.

(b) Suppose that $\rk(\fg)\!\ge\!p^2\!-\!2$, so that there is an embedding $\fsl(p^2\!-\!2) \hookrightarrow \fg$ of restricted Lie algebras. The adjoint representation of $H(2)$ provides an embedding
$H(2) \hookrightarrow \fsl(p^2\!-\!2)$. Hence the preceding arguments in conjunction with
Lemma \ref{Triang3} yield the asserted result. \end{proof}

\bigskip

\begin{Remark} In view of \cite[Theorem]{Hol} the adjoint representation of $H(2)$ defines the simple, non-trivial $U_0(H(2))$-module of minimal dimension. \end{Remark}

\bigskip
\noindent
Let $G$ be a connected reductive group with Lie algebra $\fg$. Since the derived group $(G,G)$ is semisimple, there are almost simple normal subgroups $G_1,\ldots, G_m$ of $(G,G)$ such that
$(G,G)=G_1\cdot G_2\cdots G_m$, cf.\ \cite[(9.4.1)]{Sp}. We define
\[ \msrkm(\fg):=\min\{\msrk(G_i) \ ; \ 1\!\le\!i\!\le\!m\}.\]

\bigskip

\begin{Prop} \label{Triang6}Suppose that $p\!\ge\!5$. Let $\fg$ be a reductive Lie algebra.
\begin{enumerate}
\item If $\msrkm(\fg)\ge p\!-\!1$, then we have $\modd^jU_0(\fg)= \add(\bigoplus_{\lambda \in X(\fg)}k_\lambda)$ for $j \in \{1,\ldots,p\!-\!2\}$.
\item If $\msrkm(\fg)\ge p^2\!-\!2$, then we have $\modd^{p-1}U_0(\fg)= \add(\bigoplus_{\lambda \in X(\fg)}k_\lambda)$. \end{enumerate}  \end{Prop}

\begin{proof} We write $\fg\!:=\!\Lie(G)$, with $G$ being connected and reductive. Then $(G,G)\!=\!G_1\cdots G_m$ is a product of normal connected almost simple subgroups and we put $\fg_i\!:=\! \Lie(G_i)$. Owing to 
\cite[\S13]{Hu67}, we have $\rk(G_i)\!=\!\mu(\fg_i)$, while $G$ being reductive implies that $\fg$ has a self-centralizing torus \cite[(11.9)]{Hu67}, so that $\mu(\fg_i)\!=\!\rk(\fg_i)$, cf.\ \cite[(3.5),(3.6)]{Fa04}. 

Thanks to \cite[(5.4)]{Hu67}, the restricted Lie algebra $\fg_i$ is classical simple unless $G_i$ is of type $A_{np-1}$. In that case, $C(\fg_i)$ is a torus, and $C(\fg_i) \subseteq [\fg_i,\fg_i]$ are the only possible proper
ideals of $\fg_i$.

Suppose that the root system of $G_i$ has type $A_{np-1}$. Then there is a simply connected covering $\SL(np) \stackrel{\pi}{\lra} G_i$, whose scheme-theoretic finite kernel $\cK$ is contained in the multiplicative group 
$\mu_{(np)}$ of $np$-th roots of unity. There results an exact sequence
\[ (0) \lra \Lie(\cK) \lra \fsl(np) \stackrel{\msd(\pi)}{\lra} \fg_i,\]
of restricted Lie algebras, where $\Lie(\cK)$ is a torus of dimension $\le\! 1$. Consequently, the composite $\msd(\pi) \circ \iota$ of the differential $\msd(\pi)$ with an embedding $\iota : \fsl(np\!-\!1) \hookrightarrow \fsl(np)$ 
provides an injection $\fsl(np\!-\!1) \hookrightarrow \fg_i$. It follows that there are embeddings $W(1) \hookrightarrow \fsl(p\!-\!1) \hookrightarrow \fg_i$, while the condition $\rk(\fg_i)\!\ge\! p^2\!-\!2$ implies that $H(2) 
\hookrightarrow \fsl(p^2\!-\!2)\hookrightarrow \fg_i$.

(1),(2) Let $\msrkm(\fg)\!\ge\!p\!-\!1$. Suppose that $M \in \modd^j U_0(\fg)$, where $j\!\le\!p\!-\!2$ in case $\msrkm(\fg)\!\le\!p^2\!-\!3$. Then $M|_{\fg_i} \in \modd^jU_0(\fg_i)$ for $i \in \{1,\ldots, m\}$. If $G_i$ has type 
$A_{np-1}$, then the foregoing observations in conjunction Lemma \ref{Triang2} and Lemma \ref{Triang3} yield $[\fg_i,\fg_i] \subseteq \ann_\fg(M)$. Alternatively, Proposition \ref{Triang5} implies $\fg_i \subseteq \ann_\fg(M)$, 
so that $[\fg_i,\fg_i] \subseteq \ann_{\fg}(M)$ in all cases.

Let $G'\!:=\!(G,G)$, and consider $\fg'\!:=\!\Lie(G')$, so that $\fg_i \subseteq \fg'$. If $T' \subseteq G'$ is a maximal torus, then \cite[(13.3)]{Hu67} ensures that $\ft'\!:=\!\Lie(T')$ is a maximal torus of $\fg'$.
We consider the root space decomposition
\[ \fg' = \ft'\!\oplus\!\bigoplus_{\alpha \in R}\! \fg'_\alpha\]
of $\fg'$ relative to $T'$ and recall that $\dim_k\fg'_\alpha\!=\!1$ for all $\alpha \in R$ (see \cite[(26.2)]{Hu81}). As $p\alpha \not \in R$ for every $\alpha \in R$ (cf.\ \cite[(9.1.7)]{Sp}), it follows that
$\fg'_\alpha \subseteq V(\fg')$.

As $\ann_{\fg'}(M)$ is an ideal of $\fg'$, \cite[(10.1)]{Hu67} provides a connected, closed, normal subgroup $N \unlhd G'$ and a subtorus $\fs' \subseteq \ft'$ such that the ideal $\fn:= \Lie(N)$ satisfies
\[ [\fn,\fn] \subseteq \ann_{\fg'}(M) \subseteq \fn\!+\!\fs' ,\]
while $[\fs',\fg'_\alpha] =(0)$ whenever $\fg'_\alpha \not \subseteq \ann_{\fg'}(M)$.  In view of \cite[(27.5)]{Hu81}, there is a subset $J_N \subseteq \{1,\ldots,m\}$ such that
\[ N = \prod_{i \in J_N}G_i \ \ \text{and} \ \ (G_i,N)=\{1\} \ \ \ \ \ \ \ \forall \ i \in \{1,\ldots,m\}\!\smallsetminus\!J_N.\]
Suppose that $i \in  \{1,\ldots,m\}\!\smallsetminus\!J_N$. Then we have $[\fg_i,\fn]\! =\! (0)$. Let $\alpha \in R$ be a root  such that $\fg'_\alpha \subseteq [\fg_i,\fg_i]$. In view of $[\fg_i,\fg_i] \subseteq
\ann_{\fg'}(M)$, it follows that $\fg'_\alpha \subseteq [\fg_i,\fg_i]\cap \fn$, so that $[\fg_i,\fg'_{\alpha}] = (0)$. Thus, $\fg'_\alpha \subseteq C(\fg_i)\cap V(\fg')=\{0\}$, a contradiction. Since $[\fg_i,\fg_i]$
is $T'$-stable, we conclude that $J_N=\{1,\ldots,m\}$, so that $N=G'$ and $\fn=\fg'$. As a result, $\ann_{\fg'}(M)$ contains $\bigoplus_{\alpha \in R}\fg'_\alpha$.

Since $G/(G,G)$ is a torus, it follows that all root subgroups of $G$ are contained in $G'$. Hence all root spaces of $\fg$ relative to some maximal torus $T\subseteq G$ are contained in $\fg'$, so that
$\sum_{\alpha \in X(T)\smallsetminus\{0\}} \fg_\alpha \subseteq \ann_\fg(M)$. As a result, $M$ is a module for the torus $\fg/\ann_\fg(M)$, whence $M \in \add(\bigoplus_{\lambda \in X(\fg)}k_\lambda)$.
\end{proof}

\bigskip

\subsection{Categories of modules of constant rank}
Let $r\!\ge\!2$. According to \cite[(5.6.12)]{Be} and \cite[(4.2.3)]{Bi}, the category $\EIP^1(\fe_r)\cap\modd^3U_0(\fe_r)$ has wild representation type. In view of Proposition \ref{EJIM2}(2), this entails
the wildness of the category of $\CJT(\fe_r)\cap \modd^3U_0(\fe_r)$, thereby strengthening earlier work \cite{BL} concerning modules of constant Jordan type for elementary abelian groups. In \cite{BL}
the authors conjecture that the subcategory of modules of constant Jordan type of a finite group is of wild representation type whenever the ambient module category enjoys this property.

The following results illustrate that for certain types of restricted Lie algebras of wild representation type, the aforementioned subcategories are in fact semisimple.

\bigskip

\begin{Prop} \label{Jt1} Suppose that $p\!\ge\!5$. Let $(\fg,[p])$ be a restricted Lie algebra such that
\begin{enumerate}
\item[(a)] there is an embedding $W(1) \hookrightarrow \fg$, and
\item[(b)] the factor algebra $\fg/\fv(\fg)$ of $\fg$ by the $p$-ideal $\fv(\fg)$ generated by $V(\fg)$ is a torus.  \end{enumerate}
Then we have $\CR^1(\fg)\cap\EIP^{p-1}(\fg)=\add(\bigoplus_{\lambda \in X(\fg)}k_\lambda)$. \end{Prop}

\begin{proof} As $p\!\ge\!5$, the space $\bigoplus_{i=1}^{p-2} ke_i$ is a non-abelian $p$-trivial subalgebra of $W(1)$. Let $M \in \CR^1(\fg)\cap\EIP^{p-1}(\fg)$. In view of (a), Proposition \ref{EITriv3}
yields $\rk^{p-1}(M)\!=\!0$. It now follows from Lemma \ref{Triang2}(3) that $M|_{W(1)}\in \add(k)$. As a result, $\rk^1(M)\!=\!\rk^1(M|_{W(1)})\!=\!0$, so that $\fv(\fg) \subseteq \ann_\fg(M)$. Our assertion now
follows from condition (b). \end{proof}

\bigskip

\begin{Remarks} (1) \ Simple restricted Lie algebras of Cartan type and reductive Lie algebras with $\msrkm(\fg)\ge p\!-\!1$ satisfy conditions (a) and (b) of Proposition \ref{Jt1}.

(2) \ Suppose that $W(1) \hookrightarrow \fg$. If $p\!\ge \!5$, then $\dim V(\fg) \ge \dim V(W(1))\!\ge\!3$, so that $\modd U_0(\fg)$ is wild, cf.\ \cite{Fa07}. \end{Remarks}

\bigskip
\noindent
Our final result demonstrates that already the "local" validity of the equal $(p\!-\!1)$-images property may significantly restrict the structure of modules belonging to $\CR^j(\fg)\cap\CR^{p-1}(\fg)$ or even
$\CR^{p-1}(\fg)$. Given $\fe \in \EE(2,\fg)$, we let $\CR^{p-1}(\fg)_\fe$ be the full subcategory of $\CR^{p-1}(\fg)$ whose objects satisfy $M|_{\fe} \in \EIP^{p-1}(\fe)$. Since $\EIP^1(\fe_r) \subseteq
\CR^{p-1}(\fe_r)_\fe$ for all $\fe \in \EE(2,\fe_r)$, the aforementioned work \cite{Be,Bi} guarantees that the category $\CR^{p-1}(\fe_r)_\fe$ is wild whenever $r\!\ge\!2$.

\bigskip

\begin{Prop} \label{Jt2} Suppose that $p\!\ge\!5$ and let $\fg$ be a reductive Lie algebra.
\begin{enumerate}
\item If $\msrkm(\fg)\!\ge\!p\!-\!1$, then $\CR^j(\fg)\cap\CR^{p-1}(\fg)_\fe=\add(\bigoplus_{\lambda \in X(\fg)}k_\lambda)$ for every $j \in \{1,\ldots,p\!-\!2\}$ and $\fe \in \EE(2,\fg)$.
\item If $\msrkm(\fg)\!\ge\!p^2\!-\!2$, then $\CR^{p-1}(\fg)_\fe=\add(\bigoplus_{\lambda \in X(\fg)}k_\lambda)$ for every $\fe \in \EE(2,\fg)$. \end{enumerate} \end{Prop}

\begin{proof} (1) Let $j \in \{1,\ldots,p\!-\!2\}$ and $\fe \in \EE(2,\fg)$. Proposition \ref{Res1} implies that $\CR^j(\fg)\cap\CR^{p-1}(\fg)_\fe \subseteq \modd^{p-1}U_0(\fg)$. Since $\msrkm(\fg)\!\ge\!p\!-\!1$, The arguments of 
Propositions (\ref{Triang5}) and (\ref{Triang6}) ensure the existence of an embedding $W(1) \hookrightarrow \fg$, and Lemma \ref{Triang2} thus yields $M|_{W(1)}\in \add(k)$ for every $M \in \CR^j(\fg)\cap
\CR^{p-1}(\fg)_\fe$. Consequently, $\CR^j(\fg)\cap\CR^{p-1}(\fg)_\fe \subseteq \modd^jU_0(\fg)$, and our result follows from Proposition \ref{Triang6}(1).

(2) This follows analogously, using Lemma \ref{Triang3} and Proposition \ref{Triang6}(2). \end{proof}

\bigskip

\begin{Remark} Let $M \in \CJT(\fg)$ be a module of constant Jordan type $\Jt(M)\!=\!\bigoplus_{i=1}^pa_i(M)[i]$. Then $M$ belongs to $\CJT(\fg)\cap\modd^jU_0(\fg)$ if and only if $a_\ell(M)\!=\!0$
for $j\!+\!1\!\le\ell\!\le\! p$. Proposition \ref{Jt2} shows in particular that the category $\CJT(\fg)\cap\modd^{p-1}U_0(\fg)$ may be semisimple. In contrast to the case of elementary abelian Lie algebras,
the representation type of $\CJT(\fg)$ may therefore depend only on the behaviour of those objects $M \in \CJT(\fg)$ such that $a_p(M)\!\ne\!0$. \end{Remark}

\end{document}